\documentclass[12pt]{amsart}
\usepackage[T1]{fontenc}
\usepackage[pdftex]{color,graphicx}
\usepackage{bbm} 
\usepackage{dsfont}
\usepackage{url}
\usepackage{mathtools}
\mathtoolsset{showonlyrefs}
\usepackage{anysize}
\usepackage{geometry}
\marginsize{2.5cm}{2.5cm}{2.5cm}{2.5cm}
\usepackage{amssymb}
\usepackage{hyperref}
\theoremstyle{plain}
\numberwithin{equation}{section}
\newtheorem{theorem}{Theorem}[section]
\newtheorem{lemma}[theorem]{Lemma}
\newtheorem{corollary}[theorem]{Corollary}
\newtheorem{proposition}[theorem]{Proposition}
\theoremstyle{remark}
\newtheorem{definition}[theorem]{Definition}
\newtheorem{remark}[theorem]{Remark}

\newtheorem{hypothesis}[theorem]{Hypothesis}
\newcommand{\Pe}{{\overline{P}}}
\newcommand{\Ce}{{\overline{C}}}
\newcommand{\Cce}{{\overline{\mc{C}}}}

\newcommand{\dtv}{\textnormal{d}_{\textnormal{TV}}} 
\newcommand{\abs}[1]{\left\vert #1 \right\vert}
\newcommand{\prt}[1]{\left( #1 \right)}
\newcommand{\crt}[1]{\left[ #1 \right]}  
\newcommand{\veps}{\varepsilon}

\newcommand{\mc}[1]{{\mathcal #1}}

\newcommand{\yyzz}[2]{{Y}_{#1}(#2)}

\newcommand{\yyy}[3]{{Y}^{#1}_{#2}(\infty)} 
\newcommand{\yyyz}[3]{{Y}^{#1}_{#2}(#3)}
\newcommand{\bb}[1]{{\mathbb #1}}

\newcommand{\Ind}[1]{\mathbbm{1}_{#1}}
\newcommand{\ud}{\mathrm{d}}
\newcommand{\sgn}{\mathrm{sgn}}
\newcommand{\ii}{\mathsf{i}}

\begin{document}
\title[Convergence for Langevin dynamics on a degenerate potential]{Gradual convergence for Langevin dynamics on a degenerate potential}

\author{Gerardo~Barrera}
\address{University of Helsinki, Department of Mathematics and Statistics.
P.O. Box 68, Pietari Kalmin katu 5, FI-00014. Helsinki, Finland.\\
\url{gerardo.barreravargas@helsinki.fi}\\
\url{https://orcid.org/0000-0002-8012-2600}}
\thanks{*Corresponding author: Gerardo Barrera.}
\author{Conrado~da-Costa}
\address{Department of Mathematical Sciences, Durham University. Upper Mountjoy Campus, Durham DH1 3LE, United Kingdom.\\
\url{conrado.da-costa@durham.ac.uk}\\
\url{https://orcid.org/0000-0003-3466-6751}}
\author{Milton~Jara}
\address{Instituto de Matem\'atica Pura e Aplicada, IMPA. Estrada Dona Castorina, 110, 22460-320, Rio de Janeiro, Brasil.\\
\url{mjara@impa.br}\\}


\begin{abstract}
  In this paper, we study an ordinary differential equation with a
  degenerate global attractor at the origin, to which we add a white
  noise with a small parameter that regulates its intensity. Under
  general conditions, for any fixed intensity, as time tends to
  infinity, the solution of this stochastic dynamics converges
  exponentially fast in total variation distance to a unique
  equilibrium distribution. We suitably accelerate the random dynamics
  and show that the preceding convergence is gradual, that is,
  the function that associates to each fixed $t\geq 0$ the total variation
  distance between the accelerated random dynamics at time $t$ and
  its equilibrium distribution converges, as the noise intensity tends to zero, to a decreasing function with values in $(0,1)$. Moreover, we prove that this  limit function for each fixed $t \geq 0$ corresponds to the total variation distance between the marginal, at time $t$,  of a stochastic
  differential equation that comes down from infinity and its corresponding equilibrium distribution.  This
  completes the classification of all possible behaviors of the total
  variation distance between the time marginal of the aforementioned
  stochastic dynamics and its invariant measure for one dimensional
  well-behaved convex potentials.  In addition, there is no cut-off
  phenomenon for this one-parameter family of random processes and
  asymptotics of the mixing times are derived.
\end{abstract}
\maketitle

\section{\textbf{Introduction}} 
The study of random dynamical systems and their convergence to
equilibrium is one of the most studied subject in probability theory
and mathematical physics with a vast literature such as stochastic control~\cite{Arapostathis}, slow-fast systems~\cite{Berglund},
small noise
limit~\cite{BAKHTIN,FW,Martinelli,Martinelli1}, small noise asymptotics
for invariant densities~\cite{BAinv,Biswas, FW,Monmarche}, sharp estimates on
transit and exit times~\cite{Giacomin}, couplings and quantitative
contraction rates for Langevin dynamics~\cite{Eberle,Iacobucci,Veretennikov},
convergence to equilibrium in Fokker--Planck
equations~\cite{Bogachev,Bolley,Ji}, random attractors for stochastic
dissipative systems~\cite{Kuehn}, numerical computations of geometric
ergodicity~\cite{LELI,LiWang}, multi-scale analysis, ergodicity and exponential loss
of memory of the initial condition~\cite{KU, pages,Peng}, regularity for Lyapunov
exponents~\cite{SI}, metastability and large
deviations~\cite{VARES, Bovier}, optimal transport~\cite{Villani}, etc.

The goal of this paper is the study of the \textit{convergence to
  equilibrium} in the so-called \textit{zero-noise limit} for a family
of stochastic small random perturbations of a given one-dimensional
dynamical system.  We consider an ordinary differential equation with
a degenerate (non-hyperbolic) global attractor at the origin.  Under
appropriate conditions on the dynamics, as time increases, for any
initial condition the solution of this differential equation tends to
the origin polynomially fast. We then consider a perturbation of the
deterministic dynamics by a Brownian motion of small intensity.  This
random dynamics possesses a unique invariant probability measure and
for any initial condition, the solution converges in the total
variation distance to such invariant probability measure as times
increases.  We prove that the convergence occurs gradually, that is, 
when the strength of the noise $\veps$ tends to zero, with a suitable scaling
of time ($a_\veps, \veps >0$), the function that associates to each fixed $t\geq 0$ the total
variation distance between the marginal of the random  dynamics at
  time $a_\epsilon t$ and its equilibrium tends, as $\veps \to 0$,  to a \textit{decreasing}
  function with values in the open interval $(0,1)$, see also
  Definition~\ref{def:gradual} below.  This fact, with the help
  of Proposition \ref{idea} 
implies no cut-off phenomenon in the context of random
processes.

\subsection{\textbf{The degenerate Langevin dynamics}}
In this subsection, we specify the \textit{degenerate Langevin
  dynamics} that we consider in this paper.  Here we say that a
Langevin dynamics is degenerate when its vector field possesses a
 degenerate fixed (critical) point.

The Langevin dynamics was introduced by P. Langevin in 1908 in his
seminal article \cite{Langevin1908}. It is perhaps one of the most
popular models in molecular systems. For details on its history
and phenomenological treatment, we refer to \cite{COFFEY,Pomeau} and
the references therein.

Let $\veps\in (0,1]$ be the parameter that controls the intensity
of the noise and let
$X^{\veps}(x):=\left(X^{\veps}_t(x),{t \geq 0}\right)$ be the
unique strong solution of the one-dimensional Stochastic
Differential Equation (for short SDE)
\begin{equation}
\label{modelo}
\left\{
\begin{array}{r@{\;=\;}l}
  \ud X^\veps_t & -V^{\prime}(X^\veps_t) \ud t + \sqrt{\veps}\ud B _t
                  \quad\textrm{ for } \quad  t\geq  0, \\
  X^\veps_0 & x,
\end{array}
\right.
\end{equation} 
where $x\in \mathbb{R}$ is a deterministic initial condition,
$B:=\left(B_t,t\geq 0\right)$ is a one-dimensional standard Brownian
motion defined on a probability space
$(\Omega,\mathcal{F},\mathbb{P})$ and $V:\mathbb{R} \to [0,\infty)$ is
a given function that will be referred to as \emph{the
    potential}.  In order to avoid technicalities and since we  want to be able to use
It\^o's formula, we assume the following conditions for $V$.
\begin{hypothesis}[Regularity]\label{hyp1}
We assume that the potential $V$
is a twice continuously differentiable, convex and even function with $V(0)=0$.
\end{hypothesis}
\noindent

Since for the dynamics~\eqref{modelo} we only consider from the
  potential $V$ its derivative $V'$, the value of $V$ at $0$ is not crucial
  and it is only imposed in Hypothesis~\ref{hyp1} to fix a unique
  potential once given its derivative. Moreover,
 since $V$ is even and differentiable, we have $V^{\prime}(0)=0$. We
recall that $0$ is a degenerate fixed point when
$V^{\prime\prime}(0)=0$. In what follows, we assume the following
local behavior at $0$.

\begin{hypothesis}[Local behavior at the origin]\label{hyp2}
There exist positive constants $C_0$ and $\alpha$
such that
\begin{equation}
\label{LB}
\lim\limits_{\lambda \to 0}
\sup\limits_{|z|\leq 1}\left|\frac{V^\prime(\lambda z)}{\lambda^{1+\alpha}}-C_0|z|^{1+\alpha}\sgn(z) \right|=0,
\end{equation} 
where  $\sgn(z):=\;z\;|z|^{-1}\mathbbm{1}_{\{z \neq 0\}}$. 
\end{hypothesis}
\begin{remark}
  An intuition for Hypothesis~\ref{hyp2} is to think of it as a
  generalization of the behavior of the  monomial potential $V_0: \mathbb{R} \to [0,\infty)$ given by $V_0(x) : =|x|^{2 + \alpha}$
  with $\alpha>0$ to smooth potentials $V: \bb{R} \to [0,\infty)$ with leading behavior at the origin given by $C_0|x|^{2 + \alpha}$ and suitable $C_0>0$ as described in~\eqref{LB}. That is, $V(x)=C_0|x|^{2 + \alpha}+\mathrm{o}(|x|^{2 + \alpha})$ as $x\to 0$. 
  For instance, for the case $V(x) = |x|^{2 + \alpha}$, $x\in \mathbb{R}$, we have
  $\frac{V'(\lambda z)}{\lambda^{1 + \alpha}} = (2 + \alpha) |z|^{1+ \alpha}\sgn(z)$ for $z\in \mathbb{R}$ and $\lambda\not=0$,
  and hence Condition~\eqref{LB} is satisfied with $C_0:= 2+\alpha$.
  
Roughly speaking, the local behavior of the potential at the origin captured in~\eqref{LB} controls the convergence to equilibrium in~\eqref{modelo}. In fact, the convex potential $V$ drives the trajectories of~\eqref{modelo} to the origin and the convergence to equilibrium depends on the intensity of the noise and the strength of drift determined by the potential. See 
Section~\ref{sec:heuristics} for further heuristics.

We also point out that Hypothesis~\ref{hyp2} is equivalent to 
\begin{equation}
\label{LB-k}
\tag{\textbf{L}}
\lim\limits_{\lambda \to 0}
\sup\limits_{|z|\leq K}\left|\frac{V^\prime(\lambda z)}{\lambda^{1+\alpha}}-C_0|z|^{1+\alpha}\sgn(z) \right|=0\quad \textrm{ for any } \quad  K>0.
\end{equation}     
\end{remark}
\noindent
Finally, in order to control the growth of $V^\prime$ around infinity
and to ensure that~\eqref{modelo} has a unique invariant probability
measure, we assume the following growth condition.
\begin{hypothesis}[Growth at infinity]  \label{hyp3}
There exist $c_0 ,R_0 \in (0,\infty)$, and $\beta\in (-1,\infty)$ such that
\begin{equation}\label{CGnew}\tag{\textbf{G}}
V^\prime(z)\geq c_0 z^{1+\beta}
\quad \textrm{ for all  }\quad z\geq R_0.
\end{equation}
\end{hypothesis}

An interesting
example, which satisfies Hypothesis~\ref{hyp1}, Hypothesis~\ref{hyp2}
and Hypothesis~\ref{hyp3}, is the one-well potential
$V(z)=|z|^{2+\alpha}$, $z\in \mathbb{R}$ for some $\alpha>0$.
When $\alpha=0$, we have that~\eqref{modelo} corresponds to the
Ornstein--Uhlenbeck process which exhibits profile cut-off for
$x\not=0$ and it does not when $x=0$, for further details
see~\cite{BA} and~\cite{BJ}.  For $\alpha>0$ we have
$V^{\prime}(0)=V^{\prime\prime}(0)=0$ and hence
Theorem~2.1 in~\cite{BJ} cannot be applied.  In fact, in the
degenerate case, for any initial condition the convergence to
equilibrium is gradual, in the sense of Definition~\ref{def:gradual} below, which implies no cut-off. This is in stark contrast with the
Ornstein--Uhlenbeck process and it is natural from the dynamical point
of view, since the fixed point changes from hyperbolic to
non-hyperbolic (degenerate).  For instance, the qualitative behavior
of hyperbolic systems and degenerate systems are very different, the
former are structurally stable whereas the latter are not.  In the
hyperbolic attracting case, it is shown in Theorem~2.2 in~\cite{BJ}
that profile cut-off phenomenon holds true and the proof relies on the
Hartman--Grobman theorem, which breaks down at degenerate points. In
the present degenerate setting, we introduce a time space scaling and
obtain gradual convergence to equilibrium, see
Theorem~\ref{T:silhouette} below. Moreover, there is a qualitative change of
behavior in the model: the cut-off phenomenon is not present in this
setting, see Corollary~\ref{C:nocut} below. We also give asymptotics
for the mixing times in~\eqref{ec:limz} in  Corollary~\ref{C:nocut}.

Another example which underlies our motivation to study this type of model is  the so-called Ginzburg--Landau potential. More precisely, for a given $\eta\in \mathbb{R}$ the \textit{Ginzburg--Landau potential} 
$V_\eta:\mathbb{R}\to \mathbb{R}$ is defined by
\begin{equation}
\begin{split}
V_\eta(x):&= \cosh(x) - \frac{1}{2}\eta x^2\\
&=1+\frac{(1-\eta)x^2}{2!}+\frac{x^4}{4!}+\frac{x^6}{6!}+\cdots\quad \textrm{ for any }\quad x\in \mathbb{R}.
\end{split}
\end{equation}
Note that $V^{\prime}_\eta(0)=0$ for all $\eta\in
\mathbb{R}$. Moreover, for $\eta\leq 1$, we have that $V_\eta$ is
convex and $V^{\prime\prime}_\eta(x)\geq (1-\eta)$ for all
$x\in \mathbb{R}$ with $V^{\prime\prime}_\eta(0)=1-\eta$.  For
$\eta<1$, $V_\eta$ is coercive and hence Theorem~2.1 in~\cite{BJ}
applies, yielding the cut-off phenomenon (abrupt convergence) to
equilibrium for~\eqref{modelo}.  For $\eta>1$, $V_\eta$ is no longer
convex and has the classical double-well shape used in models which
exhibit metastability, see \cite{Beltran,Bovier,Galves,Landim,VARES}.
We point out that metastable models do not exhibit cut-off phenomenon, see~\cite{BBF1,BBF}.
Finally, at the critical value $\eta=1$, up to a translation, the
potential $V_1$ satisfies Hypothesis~\ref{hyp1}, Hypothesis~\ref{hyp2}
and Hypothesis~\ref{hyp3}. Therefore, Theorem~\ref{T:silhouette} below
yields gradual convergence to equilibrium for~\eqref{modelo} when the
potential is $V_1$.

With the present result, we improve the classification of the
behaviors of the total variation distance between the marginal
$X^{\veps}_t$ given in~\eqref{modelo} and its invariant measure for
one dimensional smooth convex potentials $V$. More precisely,
provided the convex potential $V$ satisfies regularity conditions, such as it grows at infinity and is well-behaved at the origin, then we can classify the convergence to equilibrium as follows.
\begin{itemize}
\item[(1)] \textit{Cut-off phenomenon}. This occurs when the fixed point of
  the deterministic dynamics associated to~\eqref{modelo} is
  hyperbolic, i.e., $V^{\prime\prime}(0)>0$, see Theorem~2.1
  in~\cite{BJ}.
\item[(2)] \textit{Gradual convergence to equilibrium}. This  occurs when the
  fixed point of the deterministic dynamics associated
  to~\eqref{modelo} is degenerate, i.e., $V^{\prime\prime}(0)=0$, see
  Theorem~\ref{T:silhouette} below.
\end{itemize}
When the potential $V$ is not convex and possesses finitely many
(hyperbolic) stable equilibria, it is well-known that
\textit{metastability phenomenon} occurs,
see~\cite{Beltran,Bovier,Galves,Landim,VARES}.

By Hypothesis~\ref{hyp1} 
the SDE~\eqref{modelo} has a unique strong solution, see Theorem~3.5 in~\cite[p.58]{Mao} or Theorem 10.2.2 in~\cite[p.255]{SV}. Hence, $X^{\veps}(x)$ is a well-defined
stochastic process on the probability space
$(\Omega,\mathcal{F},\mathbb{P})$.  Furthermore,
Lemma~\ref{medidainvariante} below yields that~\eqref{modelo} is exponentially ergodic in total variation distance 
with a unique invariant probability measure $\mu^\veps$ given by 
\begin{equation}\label{formulamu}
  \mu^\veps(\ud z)=\frac{e^{-\frac{2}{\veps}V(z)}}{C_\veps}\ud z\qquad \text { with }
  \qquad C_\veps: = \int_{\bb{R}} e^{-\frac{2}{\veps}V(y)}\,\ud y.
\end{equation}

\subsection{\textbf{Results}}
In this section, before we state the main results of the paper,
we recall the definition of the total variation distance and fix some
conventions.

In the sequel, we adopt the convention that $\sgn(0)\infty=0$ and
since $\veps\in (0,1]$, for simplicity we write $\veps\to 0$
instead of $\veps\to 0^+$.  We point out that for any
$x\in \mathbb{R}$ and $t>0$, the marginal $X^\veps_t(x)$ is absolutely
continuous with respect to the Lebesgue measure on $\mathbb{R}$. Then
we measure the distance between the law of $X^\veps_t(x)$ and its
limiting distribution $\mu^\veps$ by the total variation distance,
defined by
\begin{equation}
\dtv(\nu_1,\nu_2):=\sup\limits_{F\in \mathcal{F}}|\nu_1(F)-\nu_2(F)|
\end{equation}
for any $\nu_1$ and $\nu_2$ probability measures in the same
measurable space $(\Omega,\mathcal{F})$.  For convenience, we do not
distinguish a random variable $X_1$ and its law $\mathbb{P}_{X_1}$ as
an argument of $\dtv$. In other words, for random variables $X_1$ and
$X_2$ and probability measure $\mu$ we write $\dtv(X_1,X_2)$ in place
of $\dtv(\mathbb{P}_{X_1},\mathbb{P}_{X_2})$ and write $\dtv(X_1,\mu)$ instead
of $\dtv(\mathbb{P}_{X_1},\mu)$.  For further details on the total
variation distance, we refer to~\cite[Ch.~2]{KU}
or~\cite[Sec.~3.3]{Reiss}.

To state our main result, we introduce the following definition.
\begin{definition}[Gradual convergence]\label{def:gradual}
  For each $\veps \in (0,1]$, let $X^\veps = (X^\veps_t, t \geq 0)$ be
  a stochastic process with unique invariant probability measure $\mu^\veps$, and
  fix a deterministic $a_\veps>0$. We say that the family of stochastic processes
  $(X^\veps, \veps \in(0,1])$ \emph{exhibits gradual convergence to
    equilibrium at scale $(a_\veps, \veps >0)$ with respect to the
    total variation distance as $\veps \to 0$}, when the map $d_\veps:(0,\infty)\to [0,1]$ defined by 
  \[
  d_\veps(t): =\dtv(X^\veps_{a_\veps t}, \mu^\veps),\quad t\geq 0,
  \]
  converges as $\veps$ tends to zero to a function $d_0: (0,\infty) \to (0,1)$ in the sense
  that for almost all $t>0$
  \[
  \lim_{\veps \to 0} d_\veps(t) = d_0(t).
  \]
  We say that the
  family of stochastic processes $(X^\veps, \veps \in(0,1])$
  \emph{exhibits gradual convergence to equilibrium with respect to the total
    variation distance} or simply \emph{exhibits gradual convergence}
  when there is a scale $(a_\veps, \veps >0)$ for which the family of
  stochastic processes $(X^\veps, \veps \in(0,1])$ \emph{exhibits
    gradual convergence to equilibrium at scale $(a_\veps, \veps >0)$
    with respect to the total variation distance as $\veps \to 0$}.
\end{definition}
\begin{remark}
We point out that gradual convergence and cut-off do not form a dichotomy. Indeed,  in principle, one could have a process with a unique
  invariant measure that converges after scaling to a profile function $f$  such that
  \begin{equation}
    \label{eq:no-gradual-no-cut}
    f(t):=
    \begin{cases}
      1 & \text{ if } 0\leq t \leq 1,\\
      \frac{1}{2} & \text{ if } t \in (1,2],\\
      0 & \text{ if } t >2.
    \end{cases}
  \end{equation}
  The above profile  function is not compatible with gradual
  convergence, where we require $f(t) \in (0,1)$ for all $t >0$ nor it
  is compatible with cut-off, as the function does not drop abruptly
  to $0$, since reaches a plateau with
  value $\frac{1}{2}$ for $ t \in (1,2]$. One example for this would be to consider two decks of cards, one red, one black. Assume that the shuffling for the black deck of cards converges to equilibrium after scaling at the deterministic time $1$ and that the shuffling for the red one, converges to equilibrium at the deterministic time $2$. If we select the deck of cards according to the outcome of a fair coin toss, red deck if the coin lands ``heads'' and the black deck if the coin lands ``tails''. This process will exhibit convergence to a profile \eqref{eq:no-gradual-no-cut} which is not gradual not cut-off. We point out that discontinuous profile functions may arise in which one still retains gradual convergence, see for instance \cite{CapQua20}.
\end{remark}
The main result of this paper whose proof is given in Section~\ref{outline} is the following. 
\begin{theorem}\label{T:silhouette}
  Assume that Hypothesis~\ref{hyp1}, Hypothesis~\ref{hyp2} and
  Hypothesis~\ref{hyp3} hold true.  For $\veps\in (0,1]$ and
  $x\in \mathbb{R}$, let $X^\veps(x)$ be the unique strong solution
  of~\eqref{modelo} and denote by $\mu^\veps$ its unique invariant
  probability measure.  Define the scaling parameter
\begin{equation}\label{eq:paraeps}
a_\veps:= \veps^{-\frac{\alpha}{2 + \alpha}}, \quad \textrm{ where $\alpha>0$ is given in } Hypothesis~\ref{hyp2}.
\end{equation}
Then for any $t>0$ it follows that
\begin{equation}\label{Gtsil0}
  \lim_{\veps\to 0}\dtv\prt{X^{\veps}_{ta_\veps}(x), \mu^\veps}=
  \dtv\prt{Y_t(\sgn(x)\infty) , \nu}\in (0,1),
\end{equation}
where $Y_t(\sgn(x)\infty):=\lim\limits_{r\to  \infty}Y_t(\sgn(x)r)$ and for $r >0$, $(Y_t(\sgn(x)r), t \geq 0)$ is the strong solution of the SDE 
\begin{equation}
\label{Y0sil} 
\left\{
\begin{array}{r@{\;=\;}l}
\ud Y_t & -C_0|Y_t|^{1+\alpha}\sgn(Y_t)\ud t+ \ud W_t\quad\textrm{ for } \quad t>  0,\\
Y_0 & \sgn(x)r,
\end{array}
\right.
\end{equation}
$(W_t,t\geq 0)$ is a standard Brownian motion, $\nu$ is the unique
invariant probability measure for~\eqref{Y0sil}, and the constant
$C_0$ is defined in Hypothesis~\ref{hyp2}.  Moreover, the map
\begin{equation}
\label{eq:mapa_decrescente}
 t \mapsto  \dtv\prt{Y_t(\sgn(x)\infty) , \nu} \text{
  is continuous and strictly decreasing.}
\end{equation}
\end{theorem}
The computation of the scaling~\eqref{eq:paraeps} is given in Subsection~\ref{sec:multiscale}.

We point out that $Y_t(\sgn(x)\infty)$ comes down from
infinity, that is, $Y_t(\sgn(x)\infty)\in \mathbb{R}$ for any $t>0$.
It is not surprising that an equation in the form of~\eqref{Y0sil} should ``come down from infinity'' and should admit a continuous Markovian extension. Since we did not find a reference with a full proof of this result, we have devoted Appendix~\ref{Zproof} to explain this in detail. The continuous
Markovian extension of the SDE~\eqref{Y0sil} to $\overline{\bb{R}}: =\bb{R} \cup \{\pm \infty\}$ is done in detail using
basic ODE/probabilistic techniques in~Appendix~\ref{Zproof} and here
we only outline the main steps.  First, based on a monotonic
comparison, which follows from the synchronous coupling, and uniform
second moment bounds for $x \in \bb{R}$, the SDE~\eqref{Y0sil} can be
extended to $\overline{\bb{R}}$, see Section~\ref{cdown}.  Then
because $\pm \infty$ are entrance boundaries for the dynamics in
$\bb{R}$, the extended family
$(Y(x):=(Y_t(x),t\geq 0), x \in \overline{\bb{R}})$ is Markovian, see
Section~\ref{mprop}.  The rigorous definition of~\eqref{Y0sil} is
given in Proposition~\ref{P:limitSDE} below.  Moreover,
Theorem~\ref{T:silhouette} actually provides the essentially unique
scale $(a_\veps,\veps\in (0,1])$ that captures the convergence to
equilibrium. In fact, since by the Chapman--Kolmogorov equation,
the map $t\mapsto\dtv\prt{X^\veps_t(x) , \mu^\veps}$ is
non-increasing, any sequence $(t_\veps, \veps\in (0,1])$ for which
\begin{equation}\label{eq:another_scale}
 \lim_{\veps\to 0}\dtv\prt{X^{\veps}_{t_\veps}(x), \mu^\veps}=
  \dtv\prt{Y_t(\sgn(x)\infty) , \nu}
\end{equation}
must satisfy 
\begin{equation}\label{eq:unique_scale}
\lim_{\veps \to 0} \frac{t_\veps}{a_\veps} =t.
\end{equation}
Indeed, assume that $\limsup_{\veps \to 0} \frac{t_\veps}{a_\veps} > t
  +\delta$ for some $\delta>0$, then by \eqref{Gtsil0} and \eqref{eq:mapa_decrescente}
  \begin{align*}
\liminf_{\veps \to 0}
    \dtv\prt{X^{\veps}_{t_\veps}(x), \mu^\veps}
    &\leq \liminf_{\veps \to 0} \dtv\prt{X^{\veps}_{(t + \delta) a_\veps}(x), \mu^\veps} \\
    &=  \dtv\prt{Y_{t+ \delta}(\sgn(x)\infty) , \nu}< \dtv\prt{Y_{t}(\sgn(x)\infty) , \nu}.
  \end{align*}
 Similarly, if $\liminf_{\veps \to 0} \frac{t_\veps}{a_\veps} < t
 -\delta$ for some  $\delta\in (0,t)$, then
  \begin{align*}
\liminf_{\veps \to 0}
    \dtv\prt{X^{\veps}_{t_\veps}(x), \mu^\veps}
    &\geq \liminf_{\veps \to 0} \dtv\prt{X^{\veps}_{(t - \delta) a_\veps}(x), \mu^\veps} \\
    &=  \dtv\prt{Y_{t- \delta}(\sgn(x)\infty) , \nu}> \dtv\prt{Y_{t}(\sgn(x)\infty) , \nu}.
  \end{align*}

We now fix a (rough) notion of abrupt convergence to equilibrium, the cut-off phenomenon. 
For further details, see Definition~1.8 in~\cite{BHP} and also~\cite{BBF1}.
\begin{definition}[Cut-off phenomenon]\label{def:cutoff}
  For each $\veps \in (0,1]$, let
    $X^\veps = (X^\veps_t, t \geq 0)$ be a stochastic process with
    unique invariant probability measure $\mu^\veps$ and deterministic $t_\veps >0$.  We say
  that the one-parameter family of stochastic processes
  $(X^\veps,\veps\in(0,1])$ \emph{exhibits cut-off  with respect
    to the total variation distance at scale
    $(t_\veps, \veps \in (0,1])$ when $t_\veps\to \infty$ as
  $\veps\to 0$} and
\begin{equation}\label{eq-def:cutoff}
\lim\limits_{\veps\to 0}\dtv\prt{X^{\veps}_{\delta t_\veps}, \mu^\veps}=
\begin{cases}
1  &\quad \textrm{ for }\quad \delta\in (0,1),\\
0  &\quad \textrm{ for }\quad \delta\in (1,\infty).\\
\end{cases}
\end{equation}
We say that the family of stochastic processes
$(X^\veps,\veps\in(0,1])$  \emph{exhibits cut-off with respect to
  the total variation distance} or simply \emph{exhibits cut-off} if there is a
scale $(t_\veps,\veps\in (0,1])$ such that $t_\veps\to \infty$
and \eqref{eq-def:cutoff} holds true.
\end{definition}
Similar arguments leading to \eqref{eq:unique_scale} allows one to obtain that the scale in the cut-off phenomenon is essentially unique, 
see also~\cite{Martinez}.

Roughly speaking, one generally expects that a one-parameter family of
well-mixing stochastic processes will exhibit abrupt convergence of
the marginals to the equilibrium distribution as a function of the
parameter.  This is known in the literature as the \textit{cut-off
  phenomenon} introduced by~\cite{AD} in the context of card
shuffling.  Actually, the notion of cut-off applies to a wide range of
random models.  In the discrete setting, the cut-off phenomenon has
been proved for many different models such as Markovian shuffling
cards and random transpositions~\cite{AD, BD,DI,DSH,Nestoridi}, random walks on the  hypercube~\cite{PDI,Levin}, birth and death chains~\cite{BBF1,BBF}, sparse Markov chains~\cite{Bordenave2019}, Glauber dynamics~\cite{Ding}, SSEP dynamics~\cite{Goncalves},
  SEP in the circle~\cite{Lacoin,Lacoin2017}, random walks in
  random regular graphs~\cite{LS}, Ornstein--Uhlenbeck
processes~\cite{BA,BPH2,JB, Lachaud}, mean field zero-range
process~\cite{Merle}, averaging processes~\cite{Quattropani}, sampling  chains~\cite{BAmax,Ycart1999}, star transpositions~\cite{Nestoridi1}, etc.   There
are relatively few examples of Markov processes, taking values in
continuous state-spaces for which the cut-off phenomenon has been
studied, such processes include linear and nonlinear SDEs driven by
small L\'evy noise~\cite{BPH2, BHP, BPH3, BPH4, BJ,BJ1,BLiu,BP},
Dyson--Ornstein--Uhlenbeck process~\cite{Boursier}, the biased
adjacent walk on the simplex~\cite{labbepetit}, Brownian motion on families of compact
  Riemannian manifolds~\cite{CSC}. 
  
\textbf{No cut-off:} A classical example of a Markov
dynamics that does not exhibit cut-off phenomenon is the random walk on
the circle $\mathbb{Z}_n$, see Example~18.5~in~\cite[Ch.18,
p.253]{Levin} or~\cite[Thm.~2.2.1, p.55]{Ceccherini}. 
Numerical results yields that the cut-off phenomenon does not occur for the entropy in the sense of information theory,
see~\cite{Trefethen}.
It has been also
proved that the ``insect Markov chain'' does not have cut-off,
see~\cite{Dangeli}. Moreover, it has been showed the absence of cut-off
for several classes of trees, including spherically symmetric trees,
Galton--Watson trees of a fixed height, and sequences of random trees
converging to the Brownian CRT, see \cite{Chiclana,Gantert}.  More
recently, it is shown that the TASEP in the coexistence line does not
have cut-off, see~\cite{Elboim}, and that there is no cut-off for sparse chains, see Corollary~4 in~\cite{Munch}.

By the Chapman--Kolmogorov equation, for any $x\in \mathbb{R}$ and $\veps\in (0,1]$ it follows that the map
\begin{equation}\label{eq:monotona}
t\mapsto d^{\veps}_t(x):=\dtv(X^{\veps}_t(x),\mu^{\veps})\quad \textrm{ is non-increasing. }
\end{equation}
By ergodicity, see Lemma \ref{medidainvariante} below, $\lim_{t\to \infty} d^{\veps}_t(x) = 0$.
This allows us to define for any $x\in \mathbb{R}$,
$\veps \in (0,1)$, and
$\eta\in (0,1)$, the
$\eta$-mixing time for the process $(X^{\veps}_t(x), t \geq 0)$ by
\begin{equation}\label{tmix-def}
  \tau^{\veps,x}_{\textsf{mix}}(\eta):=
  \inf\{t\geq 0:  d^{\veps}_t(x)\leq \eta \}.
\end{equation}
That is, one seeks the time required by the process $X^{\veps}(x)$ for
the total variation distance to its invariant measure $\mu^{\veps}$ to
be equal to or smaller than a prescribed error $\eta$.

The phenomenon of cut-off can be detected with the help of the notion
of mixing times, see for instance~\cite{Hunter,LS} or Chapter 18 of~\cite{Levin}. Following 
Equation~(18.3) in~\cite{Levin}, the cut-off
phenomenon
is equivalent to the following relation between mixing times
\begin{equation}
  \label{eq:mix-equiv-cutoff}
 \lim_{\veps \to 0} \frac{ \tau^{\veps,x}_{\textsf{mix}}(\eta)}{
    \tau^{\veps,x}_{\textsf{mix}}(1-\eta)} = 1  \quad \text{ for all }
  \quad 
  \eta \in (0,1).
\end{equation}
Due to its natural relevance, it has been extensively studied in many
stochastic models, see for instance
\cite{Anderson, Avena, Ding, Gantert1, Hermon, Hunter, Levin, Oliveira,
  Pillai, SCO} and the references therein.

The following proposition provides a rather general technique to prove
no cut-off in ergodic systems. 
For convenience we keep the notation and parameters that are proper to our model,
but we emphasize that the result can be seen as a method for proving no-cut-off valid for 
any stochastic process satisfying the hypotheses in Proposition~\ref{idea} below.
In rough terms, it states that if there
exists a non-trivial behavior for a suitable scale, then there is no
cut-off for any scale.

\begin{proposition}[No cut-off and mixing times asymptotics]\label{idea}
  Let $x\in \mathbb{R}$ be given and assume that the scale
  $(a_\veps = a_{\veps}(x), {\veps\in (0,1]})$
  satisfies 
\begin{itemize}
\item[i)] $\lim\limits_{\veps\rightarrow 0}a_{\veps}=\infty$.
\item[ii)] For any $t>0$ 
\begin{equation}\label{eq:away_from_0-1}    0<\liminf\limits_{\veps\rightarrow 0}d^{\veps}_{a_{\veps}t}(x)\leq \limsup\limits_{\veps\rightarrow 0}d^{\veps}_{a_{\veps}t}(x)<1.
\end{equation}
\end{itemize}  
Then there is no cut-off for the family $(X^{\veps}(x),{\veps \in (0,1]})$
as $\veps$ tends to zero.  In addition, if the following limit exists
\begin{equation}
    \label{eq:lim_exists}
    \lim_{\veps\rightarrow 0}d^{\veps}_{a_{\veps}t}(x)=G_x(t)\in (0,1)
\end{equation}
and the map $t\mapsto G_x(t)$ is continuous and strictly decreasing
then, for any $\eta \in (0,1)$,
\begin{equation}\label{tmix-rescaled}
  \lim_{\veps \to 0}\frac{\tau^{\veps,x}_{\textsf{mix}}(\eta)}{a_\veps} =
  \inf\{t>0: G_x(t)\leq \eta\}.
\end{equation}
\end{proposition}
The following result, which establishes the gradual convergence to
equilibrium of the family defined in \eqref{modelo} is a consequence
of Theorem~\ref{T:silhouette} and Proposition~\ref{idea}.
\begin{corollary}[No cut-off phenomenon]\label{C:nocut}
  With the assumptions and notations of Theorem~\ref{T:silhouette},
  for any $x\in \mathbb{R}$, the family of processes
  $(X^\veps(x),\veps\in (0,1])$ does not exhibit cut-off
  as $\veps\to 0$.  In addition,
\begin{equation}\label{ec:limz}
\lim\limits_{\veps \to
  0}\frac{\tau^{\veps,x}_{\textsf{mix}}(\eta)}{a_\veps}
=\inf\{t\geq 0: \dtv\prt{Y_t(\sgn(x)\infty) , \nu}\leq \eta\}.
\end{equation}
\end{corollary}

\begin{proof}
  By Theorem~\ref{T:silhouette} and Proposition~\ref{idea} we obtain
  Corollary~\ref{C:nocut}. Indeed, by~\eqref{eq:paraeps}
  and~\eqref{Gtsil0} we obtain that the scale function
  $(a_{\veps}, {\veps\in (0,1]})$ given in
  Theorem~\ref{T:silhouette} satisfies conditions i) and ii) of
  Proposition~\ref{idea}  with
  \begin{equation}\label{mix-scale-G}
  \begin{split}
    G_x(t) &= \lim_{\veps\rightarrow 0}d^{\veps}_{a_{\veps}t}(x)=
    \lim_{\veps\rightarrow 0} \dtv(X^{\veps}_t(x),\mu^{\veps})\\
    & = 
    \dtv\prt{Y_t(\sgn(x)\infty) , \nu} \in (0,1).
  \end{split}
  \end{equation}
  Therefore, the family of processes $(X^\veps(x),\veps\in (0,1])$ does
  not exhibit cut-off  as $\veps\to 0$.
  To obtain \eqref{ec:limz} it suffices to combine \eqref{tmix-rescaled} with
  \eqref{mix-scale-G} and to note,  by \eqref{eq:mapa_decrescente}, that the map $t\mapsto G_x(t)$ is continuous
  and strictly decreasing.
\end{proof}

\textbf{Structure of the paper.}  In Section~\ref{outline} we explain
the proof of Theorem~\ref{T:silhouette} and in
Section~\ref{proofs} we complete the outline of the proof of
Theorem~\ref{T:silhouette}. The Appendix is divided in four
sections. Appendix~\ref{Zproof} proves that the process defined
in~\eqref{Y0sil} admits a continuous Markovian extension to $\overline{\bb{R}}$.
Appendix~\ref{apmore} is devoted to the proof of uniform bounds for
the entrance on compact sets, a crucial estimate to control the
coupling rate of the process with the equilibrium measure.  In
Appendix \ref{app:convergence-invariant-measure} we give the proofs of
the convergence of the invariant measures for the processes
$X^\veps(x)$ after suitable scaling.  Appendix~\ref{sec:complements}
contains results of technical nature that we collect to make the
presentation more self-contained.

\section{\textbf{Proof of Theorem~\ref{T:silhouette}}}\label{outline}
This section is divided in five parts. Firstly, we give an heuristic
argument for~\eqref{Gtsil0}. Secondly, we examine, for fixed $\veps\in (0,1]$,
convergence to the unique invariant measure of $X^\veps_t(x)$ as
$t \to \infty$. Thirdly, we perform a scale analysis to deduce
$(a_\veps, \veps\in (0,1])$. Fourthly, we introduce a localization
argument which allows us to simplify the potential $V$ under analysis.
Finally, we state the key results used in the proof of~\eqref{Gtsil0}.

\subsection{\textbf{Heuristics}}\label{sec:heuristics}
Assume that $V$ satisfies 
Hypothesis~\ref{hyp1}, Hypothesis~\ref{hyp2} and Hypothesis~\ref{hyp3}. 
Let $\varphi(x):=\left(\varphi_{t}(x),t\geq 0\right)$
be the solution of the  Ordinary Differential Equation (for short ODE)
\begin{equation}
\label{modelod}
\left\{
\begin{array}{r@{\;=\;}l}
\ud \varphi_t & -V^{\prime}(\varphi_t) \ud t\quad \textrm{ for } \quad t\geq 0,\\
\varphi_0 & x.
\end{array}
\right.
\end{equation}
The intuitive reason to consider~\eqref{modelod} is that, with high
probability, at early stages of the random evolution~\eqref{modelo},
the process stays close to the deterministic
evolution~\eqref{modelod}.  By the contracting nature of the random
evolution~\eqref{modelo}, which follows from the convexity of $V$ (Hypothesis \ref{hyp1}) and its growth condition at infinity (Hypothesis \ref{hyp3}), for every
$x \neq 0$ the noise gets dissipated and the process is driven to
zero, falling back into the stream of the deterministic evolution.
For this reason, \eqref{modelod}~is actually a good approximation
of~\eqref{modelo} for a long period of time and for large times,
  what matters is the behavior of $V$ at zero. Moreover, 
Hypothesis~\ref{hyp2} ensures the drift in \eqref{modelod} can be approximated at the origin,
$V^{\prime}(z) \sim C_0 \abs{z}^{1 + \alpha}\sgn(z)$, and the properly
rescaled process should converge to $Y(\sgn(x)\infty)$ as
$\veps\to 0$, where $Y(x)=\left(Y_t(x), t \geq 0\right)$ is the unique
strong solution of the following SDE
\begin{equation}\label{Y00}
\left\{
\begin{array}{r@{\;=\;}l}
\ud Y_t & -C_0|Y_t|^{1+\alpha}\sgn(Y_t)\ud t+ \ud W_t\quad \textrm{ for }\quad t\geq 0,\\
Y_0 & x.
\end{array}
\right.
\end{equation}
The exact scale and validity of the replacement of \eqref{modelo} by  \eqref{Y00} is not immediate and is explained in the remainder this section. 
\subsection{\textbf{The invariant probability measure}}
By the next lemma,~\eqref{modelo} admits a unique invariant probability measure 
$\mu^\veps$.
\begin{lemma}[Exponential ergodicity] 
\label{medidainvariante} 
Assume $V$ satisfies Hypothesis~\ref{hyp1} and Hypothesis~\ref{hyp3}.
Let $\veps\in (0,1]$ be fixed and for each $x \in \bb{R}$ let
$X^\veps(x)$ be the unique strong solution of~\eqref{modelo}. Then
there exists a unique probability measure $\mu^\veps$ such that for
any $c>0$ there are positive constants $C_1=C_1(c,\veps)$ and
$C_2=C_2(c,\veps)$ for which
\begin{equation}\label{zero}
\dtv(X^{\veps}_t(x),\mu^{\veps})\leq 
C_1 e^{-C_2 t}
\left(e^{c |x|}+\int_{\mathbb{R}}e^{c |y|}\mu^{\veps}(\ud y)\right)
\quad \textrm{ for all }\quad x\in \mathbb{R},\,\, t\geq 0.
\end{equation}
Furthermore,  $\mu^\veps$ is absolutely
continuous  with respect to the
Lebesgue measure on $\mathbb{R}$ and its density
$\rho^\veps: \bb{R}\to (0,\infty)$ is given by 
\begin{equation}\label{formula}
  \rho^\veps( x)=\frac{e^{-\frac{2}{\veps}V(x)}}{C_\veps}\qquad \text { with }
  \qquad C_\veps: = \int_{\bb{R}} e^{-\frac{2}{\veps}V(y)}\,\ud y.
\end{equation}
\end{lemma}
The proof of Lemma~\ref{medidainvariante} is given in Appendix~\ref{sec:complements}.

\subsection{\textbf{Scale analysis}}\label{sec:multiscale}
In this section we clarify  the scaling factor in~\eqref{eq:paraeps} and the need to consider $Y_t(\sgn(x)\infty):=\lim\limits_{r\to  \infty}Y_t(\sgn(x)r)$
given in Theorem~\ref{T:silhouette}.
Define, for $t\geq 0$,
\begin{equation}\label{Yrescale}
\mathcal{X}^{\veps,x}_t:=\frac{X^{\veps}_{a_\veps t}(x)}{b_\veps}
\end{equation}
and let us determine the time and space scaling parameters
$a_\veps>0$ and $b_\veps>0$. By It\^o's formula the stochastic process 
$(\mathcal{X}^{\veps,x}_t,t\geq 0)$ has the same law as $(Y^\veps_t(x b_\veps^{-1}),t\geq 0)$,
where $Y^\veps(y):=(Y^\veps_t(y),t\geq 0)$ is the unique strong solution of the following SDE
\begin{equation}
\label{Ysde}
\left\{
\begin{array}{r@{\;=\;}l}
\ud Y^{\veps}_t & -\frac{a_\veps}{b_\veps}V^{\prime}(b_\veps Y^{\veps}_t)\ud t+ \frac{\sqrt{\veps a_\veps}}{b_\veps}\ud B_t\quad \textrm{ for } \quad t\geq 0,\\
Y^{\veps}_0& y.
\end{array}
\right.
\end{equation}
In what follows, we shall refer to the process $Y^\veps(y)$ as the
  rescaled process, since it arises from the rescaling of the process $X^\veps$.
Note that, by Hypothesis~\ref{hyp2}, if $b_\veps \to 0$ as $\veps \to 0$ then for any $z\in \bb{R}$ 
\begin{equation}\label{eq:approximation}
\frac{a_\veps}{b_\veps}V^{\prime}(b_\veps z)\sim 
C_0 \frac{a_\veps}{b_\veps} \abs{b_\veps z}^{1+\alpha}\sgn(z)=C_0 a_\veps b_{\veps}^\alpha \abs{z}^{1+\alpha}\sgn(z).
\end{equation}
Therefore, to obtain a non-trivial limit  for $\mathcal{X}^{\veps,x}_t$ we remove the scaling factors of \eqref{Ysde} by defining the pair $a_\veps$ and $b_\veps$  to be the solution of the system
\begin{equation}\label{scale_sys}
\left\{
\begin{array}{r@{\;=\;}l}
\frac{\sqrt{\veps a_\veps}}{b_\veps} & 1,\\
a_\veps b^{\alpha}_\veps & 1.
\end{array}
\right.
\end{equation}
The solution of~\eqref{scale_sys} is given by
\begin{equation}\label{abscales}
a_\veps=\veps^{-\frac{\alpha}{2+\alpha}}\qquad \text{ and } \qquad
b_\veps=\veps^{\frac{1}{2+\alpha}}.
\end{equation}
Condition~\eqref{scale_sys} sets the scale analysis to a fixed
magnitude of the noise ($\veps a_\veps= b^2_\veps$) and a constant
strength of the velocity field at the origin
($a_\veps b^{\alpha}_\veps = 1$).  By \eqref{eq:approximation} and
\eqref{scale_sys} the dynamics of \eqref{Ysde} converges to the
dynamics \eqref{Y00} on compact intervals as $\veps \to 0$. However,
for any initial condition $x\neq 0$ of \eqref{modelo}, the family of
processes we consider after scaling,
$(\mathcal{X}^{\veps,x}_t, t\geq 0)$, have initial condition
$x b_\veps^{-1}$, which diverges as $\veps \to 0$.  Therefore,
the zero-noise limit of~\eqref{Ysde} requires a rigorous analysis at
infinity.

\subsection{\textbf{Coupling near the origin}}\label{sub:counear}
In this section we show that the problem in Theorem \ref{T:silhouette}
is local.  That is, we prove that one may replace $V'$ in
\eqref{modelo} with the derivative of a suitable function $\widetilde{V}$
that is well behaved at a neighbourhood of the origin and satisfies
mild growth conditions. More precisely, let $x$ be the initial
condition of~\eqref{modelo}.  In the sequel, we consider a convex
potential $\widetilde{V}=\widetilde{V}_{x}$ that satisfies
\begin{equation}\label{Vcoup}
\widetilde{V}(z)=V(z)\quad \textrm{ for any $z$ with}\quad  |z|\leq L,
\end{equation}
where $L>0$ is such that $L^2\geq 1+|x|^2$.
Additionally, we assume the following growth condition. 
\begin{hypothesis}[Polynomial growth at infinity]\label{hyp4}
There exist positive constants $c$, $C$ and $R$  such that
\begin{equation}
\label{GC}
\tag{\textbf{G1}}
\widetilde{V}^{\prime}(z)\geq c z^{1+\alpha}\quad \textrm{ for } \quad z\geq R
\end{equation}
and such that
\begin{equation}\label{Ge}
\tag{\textbf{G2}}
|\widetilde{V}^{\prime}(z)| \leq  C e^{z^2} \quad \textrm{ for } \quad \abs{z} \geq R,
\end{equation}
where $\alpha>0$ is given in Hypothesis~\ref{hyp2}.
\end{hypothesis}
Furthermore, note that $\widetilde{V}$ satisfies Hypothesis~\ref{hyp1} and Hypothesis~\ref{hyp2}.
The existence of $\widetilde{V}$ is guaranteed by Lemma~\ref{lem:Vtilde} in Appendix~\ref{sec:complements}.

For each $\veps\in (0,1]$ and $x\in \mathbb{R}$
we consider the unique strong solution $\widetilde{X}^\veps(x):=(\widetilde{X}^\veps_t(x))_{t\geq 0}$ of the SDE 
\begin{equation}
\label{modelote}
\left\{
\begin{array}{r@{\;=\;}l}
\ud \widetilde{X}^\veps_t & -\widetilde{V}^{\prime}(\widetilde{X}^\veps_t) \ud t + \sqrt{\veps}\ud B _t\quad\textrm{ for }\quad  t\geq  0, \\
\widetilde{X}^\veps_0 & x.
\end{array}
\right.
\end{equation}
Since $\widetilde{V}$ is a convex function, Theorem~3.5
in~\cite[p.58]{Mao} yields that the SDE~\eqref{modelote} has a unique
strong solution.  Furthermore, Lemma~\ref{medidainvariante} implies
that~\eqref{modelote} possesses a unique invariant probability measure
$\widetilde{\mu}^\veps$.  Recall that $\mu^\veps$ is the unique
invariant probability measure for~\eqref{modelo} and for any $t\geq 0$
let
\begin{equation}\label{tvbasic}
  d^{\veps}_t(x):= \dtv\prt{X^{\veps}_t(x),\mu^\veps}\qquad \textrm{ and }\qquad \widetilde{d}^{\veps}_{t}(x):=\dtv(\widetilde{X}^{\veps}_{t}(x),\widetilde{\mu}^\veps).
\end{equation}
The next lemma yields that  $d^{\veps}_{a_{\veps} t}(x)$ and $\widetilde{d}^{\veps}_{a_{\veps} t}(x)$ are asymptotically equivalent in the following precise sense.
\begin{lemma}[Localization and replacement of potentials]\label{lem:replacement}
For all $x\in \mathbb{R}$ and $t\geq 0$ it follows that
\begin{equation}\label{eq:1}
\lim_{\veps \to 0} |d^{\veps}_{a_{\veps}
      t}(x)-\widetilde{d}^{\veps}_{a_{\veps} t}(x)| = 0,
\end{equation}
where $(a_\veps,\veps\in [0,1))$ is defined in~\eqref{eq:paraeps}.
\end{lemma}
The proof of Lemma~\ref{lem:replacement} is given in Subsection~\ref{sub:replacement}.
By Lemma~\ref{lem:replacement} it is enough to show Theorem~\ref{T:silhouette} under Hypothesis~\ref{hyp1}, Hypothesis~\ref{hyp2} and Hypothesis~\ref{hyp4}.

\subsection{\textbf{Bound via limit replacements}}
From this point onwards, we assume that $V$ satisfies  Hypothesis~\ref{hyp1}, Hypothesis~\ref{hyp2} and Hypothesis~\ref{hyp4}.

Due to the scale invariance of the total variation distance
($\dtv(cX,cY)=\dtv(X,Y)$ for any $c\neq 0$ and any pair of random variables $X$ and  $Y$, see for instance Lemma~A.1 in~\cite{BP}) the distance
$d^{\veps}_{a_\veps t}(x)$ in~\eqref{tvbasic} can be expressed in
terms of $Y^\veps_t$ and $\nu^{\veps}$, a ``scalar multiple" of
$\mu^\veps$.  For convenience, we denote by $X^\veps_\infty$ a random
variable with the law $\mu^\veps$ and by $Y_\infty$ a random variable
with law $\nu$ which is the unique invariant probability measure
for~\eqref{Y0sil}. With this notation, we have that $\nu^{\veps}$ is
the law of $Y^\veps_\infty:=b^{-1}_{\veps}X^\veps_\infty$ and
therefore
\begin{equation}\label{tv_scale_inv}
\begin{split}
  d^{\veps}_{a_{\veps} t}(x)&=\dtv(X^{\veps}_{a_\veps t}(x),X^\veps_\infty)\\
  &=
  \dtv\prt{b^{-1}_\veps X^{\veps}_{a_{\veps} t}(x), b^{-1}_{\veps}X^\veps_\infty }=
  \dtv\prt{\mathcal{X}^{\veps,x}_t,Y^\veps_\infty},
\end{split}
\end{equation}
where $\mathcal{X}^{\veps,x}_t$ is given in~\eqref{Yrescale}.
Now, by the triangle inequality  we have
\begin{equation}\label{primera}
\begin{split}
d^{\veps}_{a_{\veps} t}(x)&\leq \dtv\prt{\mathcal{X}^{\veps,x}_t,Y_t(\sgn(x)\infty)}\\
&\qquad+
\dtv\prt{Y_t(\sgn(x)\infty),Y_{\infty}}+
\dtv\prt{Y_{\infty},Y^\veps_\infty},
\end{split}
\end{equation}  
We stress that $Y_t(\sgn(x)\infty)$ is well-defined as we show in
Proposition~\ref{P:limitSDE} below.  Informally, the idea is that the
drift dominates the noise and is strong enough to ensure that the
process comes down from infinity.  The triangle inequality also
implies that
\begin{equation}\label{segunda}
\begin{split}
  \dtv\prt{Y_t(\sgn(x)\infty),Y_{\infty}} &\leq 
  \dtv\prt{Y_t(\sgn(x)\infty),\mathcal{X}^{\veps,x}_t}
  + d^{\veps}_{a_{\veps} t}(x)\\
  &\qquad
  +\dtv\prt{Y^\veps_\infty,Y_{\infty}}.
\end{split}
\end{equation}
Combining~\eqref{primera} and~\eqref{segunda} we obtain the following \textit{key estimate} that we state as a lemma.
\begin{lemma}[Decoupling inequality]
Assume Hypothesis~\ref{hyp1}, Hypothesis~\ref{hyp2} and Hypothesis~\ref{hyp4} hold true. 
Then for any $x\in \mathbb{R}$, $\veps>0$ and $t\geq 0$ it follows that
\begin{equation}\label{key}
|d^{\veps}_{a_{\veps} t}(x)- \dtv\prt{Y_t(\sgn(x)\infty),Y_{\infty}}|\leq 
 \dtv\prt{\mathcal{X}^{\veps,x}_t,Y_t(\sgn(x)\infty)}
+\dtv\prt{Y^\veps_\infty,Y_{\infty}}.
\end{equation}
\end{lemma}
The following proposition
states that the right-hand side of~\eqref{key}
tends to zero as $\veps \to 0$.
\begin{proposition}[]\label{P:limitSDE}
Assume Hypothesis~\ref{hyp1}, Hypothesis~\ref{hyp2} and Hypothesis~\ref{hyp4} hold true.
Then the following holds true:
\begin{itemize}
\item[(1)] \emph{Continuous Markovian extension:} The real valued process defined by~\eqref{Y00}
admits a continuous Markovian extension to
$\overline{\bb{R}}:= \bb{R} \cup\{\pm\infty\}$.
\item[(2)]\emph{Convergence for fixed marginal:}
For all $x \in \bb{R}$ and $t >0$ it follows that
\begin{equation}\label{limitSDE}
\lim\limits_{\veps\to 0}\dtv\prt{Y_t(\sgn(x)\infty),\mathcal{X}^{\veps,x}_t}=0,
\end{equation}
where $(\mathcal{X}^{\veps,x}_t,t\geq 0)$ is defined in~\eqref{Yrescale}.
\item[(3)] \emph{Convergence of invariant measures:} Let $Y_{\infty}$
  and $Y^\veps_\infty$ denote random variables distributed
    according to the unique invariant distributions of the dynamics
  given by~\eqref{Y00} and~\eqref{Ysde}, respectively.  The following
  limit holds true
\begin{equation}\label{limitinvariant}
\lim\limits_{\veps\to 0}\dtv\prt{Y_{\infty},Y^\veps_\infty}=0.
\end{equation}
\end{itemize}
\end{proposition}
The proof of Proposition~\ref{P:limitSDE} is given in
Subsection~\ref{s:asymptotic}.  To complete the proof
of~\eqref{Gtsil0} we rely on the following proposition.
\begin{proposition}\label{prop:non-trivial-distance}
For all $t >0$ and $x \in \bb{R}$
\begin{equation}\label{0-1} 0<\dtv\prt{Y_t(\sgn(x)\infty),Y_{\infty}}<1.
\end{equation}
\end{proposition}
The proof of Proposition \ref{prop:non-trivial-distance} is given in
Subsection~\ref{subsect01bound}.

Now, we are ready to prove Theorem~\ref{T:silhouette}, which is a
consequence of what we have already stated up to here.
\begin{proof}[Proof of Theorem~\ref{T:silhouette}]
  Inequality~\eqref{key} with the help
  of~\eqref{limitSDE}, \eqref{limitinvariant} implies
  that
\[
\lim\limits_{\veps\to 0}d^{\veps}_{a_{\veps} t}(x)=
\dtv\prt{Y_t(\sgn(x)\infty),Y_{\infty}}.
\]
In addition, \eqref{0-1} implies
\eqref{Gtsil0}.  The proof of
\eqref{eq:mapa_decrescente} is given in Lemma~\ref{lem:contdecre} in
Appendix~\ref{sec:complements}.
\end{proof}

\section{\textbf{Proofs: details}}\label{proofs}
In this section, we give the proof of Proposition~\ref{idea} and
complete the proofs of the statements in Section~\ref{outline}.  To be
more precise, in Subsection~\ref{modes} we prove
Proposition~\ref{idea}, the proof of Lemma~\ref{lem:replacement} is
given in Subsection~\ref{sub:replacement}, the proof of
Proposition~\ref{P:limitSDE} is given in Subsection~\ref{s:asymptotic}
and the proof of Proposition~\ref{prop:non-trivial-distance} is given
in Subsection~\ref{subsect01bound}.

\subsection{\textbf{Proof of Proposition~\ref{idea}}}\label{modes}

To prove that there is no cut-off phenomenon, by
Definition~\ref{def:cutoff}, we need to show that for any scale
$(t_\veps, \veps\in [0,1])$ with $\lim_{\veps \to 0} t_\veps = \infty$
the condition \eqref{eq-def:cutoff} does not hold.  Let
$(t_\veps,\veps\in(0,1])$ be such that
$\lim_{\veps \to 0 }t_\veps = \infty$.  First, we assume that
\begin{equation}\label{eq:infinito}
\limsup_{\veps \to 0}\frac{t_\veps}{
a_\veps}<\infty,
\end{equation}
that is, there are constants $C_1>0$
and $\veps_0\in (0,1]$ such that
$t_\veps\leq C_1 a_\veps$ for any $\veps\in (0,\veps_0]$.
By~\eqref{eq:monotona}, for any $\delta>0$ and  $\veps\in (0,\veps_0]$ we have
$d^{\veps}_{\delta  C_1 a_\veps}(x)\leq d^{\veps}_{\delta t_\veps}(x)$.
Therefore, by \eqref{eq:away_from_0-1}, \text{for $\delta>1$}  we have
\[
0<\liminf\limits_{\veps\rightarrow 0} d^{\veps}_{\delta
  C_1 a_\veps}(x) \leq \liminf\limits_{\veps\rightarrow 0} d^{\veps}_{\delta
   t_\veps}(x).
\]
Hence, condition~\eqref{eq-def:cutoff} fails at the scale $(t_\veps,\veps\in (0,1])$
for the family $(X^{\veps}(x),\veps\in (0,1])$.

If~\eqref{eq:infinito} fails, then there exists a sequence
$(\veps_k,k\in \mathbb{N})$ such that $\veps_k\to 0$ as $k\to \infty$
and
\begin{equation}
\limsup_{k \to \infty}\frac{t_{\veps_k}}{
a_{\veps_k}}=\infty.
\end{equation}
In particular, there exists $k_0\in \mathbb{N}$ such that 
${a}_{\veps_k}\leq t_{\veps_k}$ for all $k\geq k_0$.
By~\eqref{eq:monotona} and~\eqref{eq:away_from_0-1}, for $0<\delta<1$ we have
\[
 \limsup\limits_{k\rightarrow \infty} d^{\veps_k}_{\delta t_{\veps_k}}(x)\leq  
\limsup\limits_{k\rightarrow \infty} d^{\veps_k}_{\delta {a}_{\veps_k}}(x) \leq  
\limsup\limits_{\veps\rightarrow 0} d^{\veps}_{\delta {a}_{\veps}}(x)<1.
\]
Hence, by~\eqref{def:cutoff} there is no cut-off at the scale
$(t_\veps,\veps\in (0,1])$ for the family
$(X^{\veps}(x),\veps\in (0,1])$.  Since
$(t_\veps,\veps\in (0,1])$ is any function with
$\lim_{\veps \to 0}t_\veps = \infty$, it follows that there is no
cut-off phenomenon for the family $(X^{\veps}(x),\veps\in (0,1])$.

We now prove \eqref{tmix-rescaled}.
By assumption \eqref{eq:lim_exists} for any $t>0$ we have
\begin{equation}
  \label{Gdef}
  \lim_{\veps\rightarrow 0}d^{\veps}_{a_{\veps}t}(x) = G_x(t) \in (0,1).    
\end{equation}
Moreover, the map $t\mapsto G_x(t)$ is continuous and strictly decreasing.
Now, for each $\eta\in (0,1)$ we define
$H_x(\eta):=\inf\{t\geq 0: G_x(t)\leq \eta\}$. To prove
\eqref{tmix-rescaled} we show that
  \begin{align}
    \label{upper-lower1}
    &\limsup\limits_{\veps \to
      0}\frac{\tau^{\veps,x}_{\textsf{mix}}(\eta)}{a_\veps} \leq
      H_x(\eta) \quad \textrm{ and }\\
    \label{upper-lower2}
    &\liminf\limits_{\veps \to
      0}\frac{\tau^{\veps,x}_{\textsf{mix}}(\eta)}{a_\veps} \geq H_x(\eta).
\end{align}
To prove \eqref{upper-lower1}, let $\gamma^*\in (0,\eta)$ be fixed and
let $t^*:=t^*(\eta-\gamma^*,x)>0$ be such that
$G_x(t^*)=\eta-\gamma^*$, and $G_x(t) >\eta-\gamma^*$ for all $t<t^*$.
By~\eqref{Gdef} there is $\veps^*:=\veps^*(\eta,\gamma^*,x)>0$ such
that
  \begin{equation}
    \label{explain_bound}
-\gamma^*<\dtv\prt{X^{\veps}_{t^*\,a_\veps}(x),
  \mu^\veps}-G_x(t^*)<\gamma^*
\quad \textrm{ for all } \quad \veps\in (0,\veps^*),
  \end{equation}
which implies that $\dtv\prt{X^{\veps}_{t^*\, a_\veps}(x),
  \mu^\veps}<\eta$  for all $\veps\in (0,\veps^*)$.
Therefore, 
$\tau^{\veps,x}_{\textsf{mix}}(\eta)\leq t^*\, a_\veps$ for all $\veps\in (0,\veps^*)$, which yields 
\begin{equation}
\limsup\limits_{\veps\to 0}\frac{\tau^{\veps,x}_{\textsf{mix}}(\eta)}{a_\veps}\leq t^*=t^*(\eta-\gamma^*,x).
\end{equation}
Since $\gamma^* \in (0,\eta)$ is arbitrary and $t \mapsto G_x(t)$ is continuous
and strictly decreasing we obtain that 
\begin{equation}\label{ec:sup}
\limsup\limits_{\veps\to
  0}\frac{\tau^{\veps,x}_{\textsf{mix}}(\eta)}{a_\veps}\leq
\lim_{\gamma^* \to 0} t^*(\eta- \gamma^*,x)=H_x(\eta).
\end{equation}

To prove \eqref{upper-lower2}, let $\gamma_*\in (0,1-\eta)$ be fixed and
let $t_*=t_*(\eta+\gamma_*,x)>0$ be such that $G_x(t_*)=\eta+\gamma_*$
and $G_x(t_*)>\eta+\gamma_*$ for all $t<t_*$.
By~\eqref{Gdef} there is
$\veps_*:=\veps_*(\eta,\gamma_*,x)>0$ such that
\[
  \dtv\prt{X^{\veps}_{t_*\, a_\veps}(x),
  \mu^\veps}>\eta \quad \textrm{ for all } \quad \veps\in (0,\veps_*).
\]
Therefore, 
$\tau^{\veps,x}_{\textsf{mix}}(\eta)\geq  t_*\, a_\veps$ for all $\veps\in (0,\veps_*)$, which implies 
\begin{equation}
\liminf\limits_{\veps\to 0}\frac{\tau^{\veps,x}_{\textsf{mix}}(\eta)}{a_\veps}\geq  t_*=t_*(\eta+\gamma_*,x).
\end{equation}
Therefore
\begin{equation}\label{ec:inf}
\liminf\limits_{\veps\to
  0}\frac{\tau^{\veps,x}_{\textsf{mix}}(\eta)}{a_\veps}\geq
\lim_{\gamma^* \to 0 }t_*(\eta + \gamma_*,x)=H_x(\eta).
\end{equation}
This completes the proof of \eqref{tmix-rescaled}.  \qed

\subsection{\textbf{Proof of Lemma~\ref{lem:replacement}}}\label{sub:replacement}
By the triangle inequality for the total variation distance we have that
\begin{equation}\label{eq:jo1}
d^{\veps}_{a_{\veps} t}(x)\leq 
\dtv(X^\veps_{a_{\veps}t}(x),\widetilde{X}^\veps_{a_{\veps}t}(x))+
\widetilde{d}^{\veps}_{a_{\veps} t}(x)+\dtv(\widetilde{\mu}^\veps,\mu^\veps).
\end{equation}
Similarly, 
\begin{equation}\label{eq:jo2}
\widetilde{d}^{\veps}_{a_{\veps} t}(x)\leq 
\dtv(\widetilde{X}^\veps_{a_{\veps}t}(x),X^\veps_{a_{\veps}t}(x))+
d^{\veps}_{a_{\veps} t}(x)+\dtv(\mu^\veps,\widetilde{\mu}^\veps).
\end{equation}
By~\eqref{eq:jo1} and~\eqref{eq:jo2} we obtain 
\begin{equation}\label{eq:jo}
\left|d^{\veps}_{a_{\veps} t}(x)- \widetilde{d}^{\veps}_{a_{\veps} t}(x)\right|\leq 
\dtv(X^\veps_{a_{\veps}t}(x),\widetilde{X}^\veps_{a_{\veps}t}(x))+\dtv(\mu^\veps,\widetilde{\mu}^\veps)\quad \textrm{ for all }\quad t\geq 0. 
\end{equation}
By Lemma~\ref{lem:couplingX} below and Equation~\eqref{eq:particular} from
Lemma~\ref{lem:couplingINV} in Appendix~\ref{app:convergence-invariant-measure}, we deduce
that the right-hand side of~\eqref{eq:jo} tends to zero as
$\veps \to 0$ and thereby conclude the proof of
Lemma~\ref{lem:replacement}.\qed

\medskip

\begin{lemma}[Convergence of the drift-modified process close to the origin]\label{lem:couplingX}
For any $x\in \mathbb{R}$ and $t\geq 0$ the following limit holds
\[
\lim\limits_{\veps\to 0}\dtv(X^\veps_{a_{\veps}t}(x),\widetilde{X}^\veps_{a_{\veps}t}(x))=0,
\]
where $(a_\veps,\veps\in [0,1))$ is defined in~\eqref{eq:paraeps}.
\end{lemma}
\begin{proof}
The proof follows the steps given in the proof of Proposition~4.1, item (ii), 
of~\cite{BJ}.
Recall the definition of $\widetilde{V}$ given in~\eqref{Vcoup}. In
particular, note that $L=L_x$ is chosen such that $L^2\geq |x|^2+1$.
Let $\veps\in (0,1]$ be fixed. The variational formulation of the total variation distance  yields 
$\dtv(X^\veps_{a_{\veps}t}(x),\widetilde{X}^\veps_{a_{\veps}t}(x))\leq
\mathbb{Q}(X^\veps_{a_{\veps}t}(x)\neq \widetilde{X}^\veps_{a_{\veps}t}(x))$
for any coupling $\mathbb{Q}$ of the random variables $X^\veps_{a_{\veps}t}(x)$ and $\widetilde{X}^\veps_{a_{\veps}t}(x)$.
Moreover, 
as $|x|<L$
for the synchronous coupling  $\mathbb{P}$ (where processes are driven by the same noise),  we have
$\widetilde{X}^\veps_s(x)=X^\veps_s(x)$ for $0\leq s<\widetilde{\tau}^\veps(x)$, 
where $\widetilde{\tau}^\veps(x):=\inf\{s\geq 0: |\widetilde{X}^\veps_s(x)|>L\}$. 
Therefore,
\begin{equation}\label{eq:cotacouplingsup}
\dtv(X^\veps_{a_{\veps}t}(x),\widetilde{X}^\veps_{a_{\veps}t}(x))
\leq 
 \mathbb{P}(\widetilde{\tau}^{\veps}(x)\leq a_{\veps}t) \quad \textrm{ for any }\quad t\geq 0.
\end{equation}
Note that 
\begin{equation}\label{eq:cotacoupligsup1}
\mathbb{P}\left(\widetilde{\tau}^\veps(x)\geq a_{\veps} t \right)=\mathbb{P}\left(\sup\limits_{0\leq s\leq a_\veps t}|\widetilde{X}^\veps_s(x)|\leq L\right).
\end{equation}
Since $\widetilde{V}$ is a smooth, convex, and even function,
It\^o's formula yields $\mathbb{P}$-almost surely that
\begin{equation}\label{eq:martrick}
\begin{split}
|\widetilde{X}^\veps_t(x)|^2&=
|x|^2-2\int_{0}^{t} \widetilde{X}^\veps_s(x) \widetilde{V}^{\prime}(\widetilde{X}^\veps_s(x)) \ud s+\veps t+
\widetilde{M}^\veps_t(x)
\\
&\leq |x|^2+\veps t+\widetilde{M}^\veps_t(x)\quad \textrm{ for all }\quad t\geq 0,
\end{split}
\end{equation}
where $\widetilde{M}^\veps_t(x):=2\sqrt{\veps}\int_{0}^{t}\widetilde{X}^\veps_s(x)\ud B_s$, $t\geq 0$. By a localization procedure, it follows that
\begin{equation}\label{eq:aptrick}
\mathbb{E}[|\widetilde{X}^\veps_t(x)|^2]\leq |x|^2+\veps t\quad \textrm{ for all }\quad t\geq 0
\end{equation}
and hence
 $(\widetilde{M}^\veps_t(x),t\geq 0)$ is a true martingale.
By~\eqref{abscales} we have 
$\veps a_\veps=\veps^{\frac{2}{2+\alpha}}$, which tends to zero as $\veps \to 0$. Then for any $t>0$ fixed there exists $\veps_0=\veps_0(t,\alpha)>0$ such that $1-\veps a_{\veps}t>1/2 $ for all $\veps\in (0, \veps_0)$.
By~\eqref{eq:martrick} for any $\veps\in (0, \veps_0)$ we have
\begin{equation}\label{eq:nuevesita}
\begin{split}
\mathbb{P}\left(\sup\limits_{0\leq s\leq a_\veps t}|\widetilde{X}^\veps_s(x)|\geq  L\right)&=\mathbb{P}\left(\sup\limits_{0\leq s\leq a_\veps t}|\widetilde{X}^\veps_s(x)|^2\geq  L^2\right)\\
&\leq \mathbb{P}\left(\sup\limits_{0\leq s\leq a_\veps t}|\widetilde{M}^\veps_s(x)|\geq  L^2-|x|^2-\veps a_{\veps} t\right)\\
&\leq \mathbb{P}\left(\sup\limits_{0\leq s\leq a_\veps t}|\widetilde{M}^\veps_s(x)|\geq  1/2\right)\\
&=\mathbb{P}\left(\sup\limits_{0\leq s\leq a_\veps t}|\widetilde{M}^\veps_s(x)|^2\geq  1/4\right),
\end{split}
\end{equation}
where for the last inequality  we used that $|x|^2+1< L^2$ and $1-\veps a_{\veps}t>1/2$.
Now, by Doob's submartingale inequality, It\^o's isometry
and~\eqref{eq:aptrick} we have for all  $\veps\in (0, \veps_0)$
\begin{equation}\label{eq:nuevesita76}
\begin{split}
\mathbb{P}\left(\sup\limits_{0\leq s\leq a_\veps t}|\widetilde{X}^\veps_s(x)|\geq  L\right)
&\leq \mathbb{P}\left(\sup\limits_{0\leq s\leq a_\veps t}|\widetilde{M}^\veps_s(x)|^2\geq  1/4\right)\leq 4{\mathbb{E}[|\widetilde{M}^\veps_{a_{\veps}t}(x)|^2]}\\
&=
{16\veps\int_{0}^{a_\veps t}\mathbb{E}[|\widetilde{X}^\veps_s(x)|^2]\ud s }\leq 16|x|^2\veps a_\veps t+8\veps^2a^2_\veps t^2.
\end{split}
\end{equation}
By~\eqref{eq:cotacouplingsup},~\eqref{eq:cotacoupligsup1} 
and~\eqref{eq:nuevesita76}
we deduce
\begin{equation}
\dtv(X^\veps_{a_{\veps}t}(x),\widetilde{X}^\veps_{a_{\veps}t}(x))\leq  16|x|^2\veps a_\veps t+8\veps^2a^2_\veps t^2\quad \textrm{ for any }\quad t\geq 0,
\end{equation}
which implies the statement as $\veps \to 0$.
\end{proof}

\subsection{\textbf{Proof of Proposition~\ref{P:limitSDE}}}
\label{s:asymptotic}
The proof of Proposition~\ref{P:limitSDE} is divided in three parts,
one for each claim of the proposition. To ease the exposition, we only
give here the main steps of the
proofs and postpone the technical details to the Appendix.
\subsubsection{\textbf{Continuous Markovian extension}}
\label{cme}
The continuous Markovian extension of the SDE~\eqref{Y00} is done in
three steps.  Their proofs are given in detail in
Appendix~\ref{Zproof} and here we only outline the main steps.  First,
based on a monotonic coupling and uniform moment bounds for
$x \in \bb{R}$, SDE~\eqref{Y00} can be extended to
$\overline{\bb{R}}$, see Section~\ref{cdown}.  Second, because
$\pm \infty$ are exit boundaries for the dynamics in $\bb{R}$, the
extended family $(Y(x), x \in \overline{\bb{R}})$ is Markovian, see
Section~\ref{mprop}. Finally, in Section~\ref{cprop} we show that the
extension is continuous in the sense that
\begin{equation}\label{catinfty}
\lim_{x \to \pm\infty} \dtv\prt{Y_t(x) ,\yyzz{t}{\pm \infty}} = 0.
\end{equation}

\subsubsection{\textbf{Convergence for fixed marginal}}
\label{s:tvart}
In this section we show the limit~\eqref{limitSDE}.  For simplicity, we
consider only the case when the initial condition $x$
in~\eqref{modelo} is positive, the case when $x$ is negative can be
treated by an analogous argument, while the case $x=0$ is easier
as no scaling of the initial condition is required and~\eqref{limitSDE}  follows from the uniform convergence of the velocity fields, see~\eqref{fieldcv} below.  To ease
notation and clarify the limit procedures, we denote by $F_0, F_\veps$
the velocity fields of~\eqref{Y00} and~\eqref{Ysde}, respectively.  That
is, for any $\veps\in [0,1]$ and $z \in \bb{R}$ we define
\begin{equation}\label{fieldnota}
  F_0 (z):=- C_0\abs{z}^{1 + \alpha} \sgn(z)
 \qquad  \text{and} \qquad
  F_\veps (z):=- \frac{a_\veps}{b_\veps}V^{\prime}(b_\veps z) = - \frac{V^{\prime}(b_\veps z)}{b^{1 + \alpha}_\veps}.
\end{equation}
To ease notation, denote by $Y^0(x)$ the solution of~\eqref{Y00}. With this, 
$(Y^{\veps}(x), \veps \in [0,1])$ solves
\begin{equation}\label{gensde}
\left\{  \begin{array}{r@{\;=\;}l}
   \ud Y_t & F_\veps(Y_t) \ud t + \ud B_t \quad\textrm{ for } \quad t\geq  0, \\
    Y_0& x.
  \end{array}\right.
\end{equation}
In what follows, we consider uniform bounds for $Y^{\veps}(x)$ with
$\veps\in [0,1]$ and we will take the limit of such processes as
$\veps \to 0$.  First, since $b_\veps \to 0$ as $\veps \to 0$, Hypothesis~\ref{hyp2} yields for all $K >0$
\begin{equation}\label{fieldcv}
\lim_{\veps \to 0} \sup_{\abs{z} \leq K}\abs{F_\veps(z) - F_0(z)} =0.
\end{equation}
Also, by Proposition~\ref{Pcdfi} it follows that, almost surely, for
any $\veps \in [0,1]$, the limit
\[
Y^{\veps}_t(\infty) := \lim_{x \to \infty} Y^{\veps}_t(x)
\] exists and
is finite for $t>0$.  Next, by Lemma~\ref{Lem:Tvconv} for all $t >0$ and
$\veps \in [0,1]$,
\begin{equation}\label{continft}
  \lim_{x \to \infty}
  \dtv\prt{Y^{\veps}_t(x) , Y^{\veps}_t(\infty)} =0.
\end{equation}
Now, we fix $\eta>0$. By the uniform behavior at infinity, see
Proposition~\ref{P:compactin}, it follows that for any $a>0$, there
are $b>0$ and $\delta\in (0,t)$ such that
\begin{equation}\label{compactin}
  \sup_{\veps \in  [0,1]}
  \bb{P}\prt{\yyyz{\veps}{\delta}{\infty} \notin [a,b]}\leq   \eta/8.
\end{equation}
By~\eqref{continft}, we may choose $a>0$ large enough so that
\begin{equation}\label{continft0}
  \sup_{x\geq a}
  \dtv\big(Y^{0}_t(x) , \yyy{0}{t}{x}\big) \leq  \eta/4.
\end{equation}
Now, given $a,b$ and $\delta \in (0,t)$ we claim that there is
$\veps_0 = \veps(\eta)>0$ for which 
\begin{equation}\label{fdelta}
 \sup_{0\leq \veps\leq  \veps_0}\sup_{x \in [a,b]}\dtv\big(Y^0_{t-\delta}(x), Y^{\veps}_{t-\delta}(x)\big)\leq  \eta/4.
\end{equation}
The proof of~\eqref{fdelta} is given in Appendix~\ref{p:fdelta}.  

Now,
let $x_\veps := xb_\veps^{-1}$ and define $\mu_\delta^{\veps}$ to be
the synchronous coupling (both SDEs are driven with the same noise) of $\yyy{0}{\delta}{1}$ and
$Y^{\veps}_{\delta}(x_\veps)$. We write
$\mu_\delta^{\veps}(A,B):= \bb{P}(\yyy{0}{\delta}{1}\in A ,
Y^{\veps}_{\delta}(x_\veps)\in B)$ for any $A,B$ Borelian subsets of
$\bb{R}$. We may choose $a>0$ for which~\eqref{continft0} holds, then
we choose $b>a$ and $\delta\in (0,t)$ such that~\eqref{compactin}
and~\eqref{fdelta} also hold true.  With these choices, it follows
that for any $\veps \in [0,1]$
\begin{equation}\label{massout}
\mu^\veps_\delta(\bb{R}^2 \setminus [a,b]^2)
\leq \bb{P}\big(\yyyz{0}{\delta}{\infty} \notin [a,b]\big)
+\bb{P}\big(\yyyz{\veps}{\delta}{\infty} \notin [a,b]\big)\leq\eta/4. 
\end{equation}
The disintegration
inequality, see~Proposition~\ref{L:disintegration}, and the triangle
inequality for the total variation distance imply that
for each $x>0$ and $t>0$ there is $\veps_0>0$ such that for
$\veps \in [0, \veps_0]$ we choose $a>0$, $b>a$ and $\delta\in (0,t)$ for which~\eqref{compactin},~\eqref{continft0}, and~\eqref{fdelta} hold true and therefore
\begin{equation}\label{conclusion}
\begin{aligned}
\dtv\prt{Y^0_{t}(\infty), Y^{\veps}_t(x_\veps) }
 & \leq  \int_{\bb{R}^2} \dtv\prt{Y^0_{t-\delta}(x) , Y^{\veps}_{t-\delta}(y)}\mu_\delta^{\veps}(\ud x,\ud y)\\
  &  \leq  \mu^\veps_\delta(\bb{R}^2 \setminus [a,b]^2) +\int_{[a,b]^2} \dtv\prt{Y^0_{t-\delta}(x) , Y^\veps_{t-\delta}(y)}  \mu_\delta^{\veps}(\ud x,\ud y)\\ 
  &  \leq\qquad \quad \eta/4 \qquad+\int_{[a,b]^2} \dtv\prt{Y^0_{t-\delta}(x) , Y^0_{t-\delta}(\infty)} \mu_\delta^{\veps}(\ud x,\ud y)\\
    &\qquad \quad + \int_{[a,b]^2} \dtv\prt{Y^0_{t-\delta}(\infty) , Y^0_{t-\delta}(y)} \mu_\delta^{\veps}(\ud x,\ud y)\\
    &\qquad \quad+\int_{[a,b]^2} \dtv\prt{Y^0_{t-\delta}(y) , Y^{\veps}_{t-\delta}(y)}  \mu_\delta^{\veps}(\ud x,\ud y)\\
  & \leq \eta/4 +\eta/4 +\eta/4 +\eta/4  = \eta.
\end{aligned}
\end{equation}
Recall~\eqref{Yrescale} and observe that
$\dtv\prt{\yyzz{t}{\sgn(x)\infty},\mathcal{X}^{\veps,x}_t}
=\dtv\prt{Y^0_{t}(\infty), Y^{\veps}_t(x_\veps) }$.  Since $\eta>0$ is
arbitrary, the proof of~\eqref{limitSDE} is complete.
\qed

\subsubsection{\textbf{Convergence of invariant measures}}\label{item-3-proof-2.5}
Recall the notation introduced above \eqref{tv_scale_inv}, that is,
$Y_\infty\stackrel{d}= \nu$
and $\widetilde{X}^\veps_\infty \stackrel{d}=\widetilde{\mu}^\veps$,
where $\stackrel{d}=$ denotes equality in the distribution sense. By
Lemma~\ref{medidainvariante} it follows that
\begin{equation}\label{densityY}
\nu(\ud z)=C^{-1}\exp\big(-2 V_0(z)\big)\ud z,
\end{equation}
where $C$ is a normalization constant, $V_0(z):=(2 +
\alpha)^{-1}C_0|z|^{2+\alpha}$ with $\alpha$ and $C_0$ defined in
Hypothesis~\ref{hyp2}. 
Similarly,
$\widetilde{\mu}^\veps$ is the density of 
$\widetilde{X}^{\veps}_\infty$ and it is
given by
\begin{equation}
\widetilde{\mu}^\veps(\ud z)={\widetilde{C}_{\veps}^{-1}}\exp\left(-2\frac{\widetilde{V}(z)}{\veps}\right)\ud z.
\end{equation} 
By the change of variable theorem, with $(b_\veps,\veps\in [0,1))$ as
defined in~\eqref{abscales}, the density of
 $Y^\veps_\infty=\frac{\widetilde{X}^{\veps}_\infty}{b_{\veps}}$ is given by 
\begin{equation}\label{eq:cambio}
b_{\veps}{\widetilde{C}_{\veps}^{-1}}\exp\left(-2\frac{\widetilde{V}(b_{\veps}z)}{\veps}\right)\ud z.
\end{equation} 
By~\eqref{densityY},~\eqref{eq:cambio}, and Scheff\'e's lemma~(\cite[Lemma~3.3.2, p.95]{Reiss}), 
to conclude the proof of~\eqref{limitinvariant},
it suffices to show
\begin{equation}\label{densityc23new}
  \lim\limits_{\veps\to 0}\frac{b_{\veps}}{\widetilde{C}_{\veps}}
  e^{-2\frac{\widetilde{V}(b_{\veps}z)}{\veps}}=\frac{1}{C}{e^{-2 V_0(z)}}\quad
  \textrm{ for any }\quad z\in \mathbb{R}.
\end{equation}
The proof of~\eqref{densityc23new} is given in Lemma~\ref{claim1in} and
Lemma~\ref{lem:constantes} in Appendix \ref{app:convergence-invariant-measure}.

\subsection{\textbf{Strict inequalities for the rescaled process}}\label{subsect01bound}
In this section we show 
\[0<\dtv\prt{Y_t(\sgn(x)\infty),Y_{\infty}}<1\quad \textrm{ for any }\quad  t>0.
\]
First, we prove the upper bound
and then we show the lower bound.

\noindent\textbf{The upper bound.}
We first note that for any $t>0$, $x \in \bb{R}$, the marginal
$Y_t(x)$ has full support in $\bb{R}$, see
Proposition~\ref{Supporttheorem} in Appendix~\ref{sec:complements} for
a proof.  By Proposition~\ref{P:limitSDE} the family
$(Y(x), x \in \overline{\bb{R}})$ is Markovian, and hence, by
semigroup property, $Y_t(\infty)$ is equal in law to
$Y_{t/2}(Y_{t/2}(\infty))$ for any $t>0$.  Since
$\bb{P}(Y_{t/2}(\infty) \in \bb{R}) = 1$ it follows by our previous
discussion that $Y_t(\infty)$  with law $\nu_t$ possesses a
continuous density $\rho_t:\bb{R} \to (0, \infty)$, that is
  $\nu_t(\ud x) =\rho_t(x) \ud x$.  Furthermore, the invariant
distribution of $Y$ corresponding to the random variable $Y_\infty$ has explicit density
function $\rho:\bb{R} \to (0, \infty)$, which is given
in~\eqref{densityY}.  To conclude that $\dtv (Y_t(\infty),Y_\infty)<1$
we note that
\begin{equation}\label{eq:menorqueuno}
  \dtv(Y_t(\infty),Y_\infty)  = 1 -\int_{\bb{R}}  \min \{\rho_t(z), \rho(z) \}\, \ud z < 1.
\end{equation}

\noindent\textbf{The lower bound: injective evolution map.}
To prove the lower bound, we first define the evolution map on the
space of measures. Let $\mc{P}$ be the space of probability measures
on $\bb{R}$ that are absolutely continuous with respect to the
Lebesgue measure on $\bb{R}$
and let $C_b(\bb{R})$ be the set of bounded
continuous functions on $\bb{R}$.
For each $\mu \in  \mc{P}$ and $t \geq 0$ let $\varphi = \varphi(\mu,t)$ be the measure  such that for every $f\in C_b({\bb{R}})$
\begin{equation}
  \label{evolution-map}
 \int f(x) \ud \varphi(x) := \int
 \bb{E}[f(Y_t(x))] \ud \mu(x). 
\end{equation}
By Proposition~\ref{Supporttheorem} we have that $\varphi(\mu,t)\in \mc{P}$ for all $\mu \in \mc{P}$ and $t>0$.
For fixed time $t>0$, the evolution map is injective in the sense that
\begin{equation}
  \label{injective}
 \text{if }\quad  \mu, \mu' \in \mc{P}\quad\text{ and }\quad\mu\neq \mu'\quad\text{ then for all }\quad t \geq 0\quad
 \varphi(\mu,t)\neq \varphi(\mu',t).
\end{equation}
Moreover, since the dynamics is uniquely ergodic, see
Lemma~\ref{medidainvariante}, for all $t>0$ the map $\mu\mapsto
\varphi(\mu,t)$ admits a unique fixed point, that is, there is a unique
$\nu\in \mc{P}$ such that
\begin{equation}\label{ergodic}
\varphi(\nu,t) = \nu \quad \textrm{ for any }\quad t\geq 0.
\end{equation}
Recall that we denote the law of $Y_\infty$ by $\nu$.
By Proposition~\ref{P:limitSDE} and Proposition~\ref{Supporttheorem}
 it follows that for all
$\delta>0$ the law of $Y_\delta(\sgn(x)\infty)$ denoted by $\mu_\delta^\infty$ belongs to
$\mc{P}$. By the Markov property of the extended process, see Proposition~\ref{P:limitSDE},
we have for any $\delta\in (0,t)$
\begin{equation}
  \label{markov-flow}
  \mu^\infty_t= \varphi(\mu_\delta^\infty,t-\delta).
\end{equation}
Let $t>0$ be fixed. We observe that there exists $\delta \in (0,t)$ such that $\mu_\delta^\infty\neq \nu$.
By~\eqref{injective},~\eqref{ergodic} and~\eqref{markov-flow} it follows that 
$
  \mu^\infty_t = \varphi(\mu_\delta^\infty,t-\delta) \neq
  \varphi(\nu,t-\delta) = \nu
$
and hence
\begin{equation}
  \label{eq:12}  
 0<\dtv\prt{ \mu^\infty_t,\nu}= \dtv\prt{Y_t(\sgn(x)\infty),Y_{\infty}}. 
\end{equation}

\begin{appendix}
\section{\textbf{The Continuous Markovian extension: details}}
\label{Zproof}
In this section we prove that the SDEs defined in~\eqref{Y00}
and~\eqref{Ysde}, or equivalently in~\eqref{gensde}, with state space
$\bb{R}$ may be extended to $\overline{\bb{R}}: = \bb{R}\cup \{\pm
\infty\}$. Furthermore, we show that this extension is Markovian and
that the family of transition kernels associated to it is continuous
with respect to the initial condition, in the sense 
of~\eqref{catinfty}.

The main reason for this appendix is to provide a full proof of the continuous Markovian extension of SDE~\eqref{gensde}.  
The methods we employ here are of a probabilistic and  path-wise nature offering an alternative to the classical analytical techniques of generators and resolvents presented in~\cite[p.366 ff]{EK},~\cite{Feller54},\cite[Chap.~17]{Kallenberg}, \cite{Feller52}. More specifically, we apply martingale convergence methods and $L^2$-bounds, which can be found in \cite[Chap.~5]{KarSch98} together with standard methods for ODEs and SDEs which can be found in~\cite[Thm.~1]{Pontryagin} and \cite[Thm.~1.1]{ikedawatanabe}.  

For the extension, we consider $\overline{\bb{R}}$ endowed with the
Borel $\sigma$-algebra associated to the metric $d_\infty
:\overline{\bb{R}}\times\overline{\bb{R}}\to \bb{R}_+$ defined by
\begin{equation}
  \label{compact-metric}
  d_\infty(x_1,x_2): =    \abs{\arctan(x_1) - \arctan(x_2)},
\end{equation}
where $\arctan:\overline{\bb{R}}\to \bb{R}$ is the continuous function
defined by 
\begin{equation}
  \label{arctan}
  \arctan(v) :=
  \begin{cases}
    -\pi/2  & \text{ for } v = -\infty,    \\
    \int_0^v \frac{1}{1 + u^2} \ud u  & \text{ for } v \in \bb{R},    \\
    \pi/2  & \text{ for } v  = \infty.
    \end{cases}
\end{equation}
Let $Y^\veps(x) = (Y^\veps_t(x), t \geq 0)$ be the
unique strong solution of~\eqref{gensde}.
For $x \in \bb{R}$ let $P^\veps_x$ be the law induced by
$Y^\veps(x)$ on the space of real valued continuous functions $(C, \mc{C})$  and let
$\Pe^\veps_x$ be its law on the space of $\overline{\bb{R}}$-valued
continuous functions $(\Ce,\Cce)$. To complete the extension we define
$\Pe^\veps_x$ for $x \in\{-\infty, \infty \}$ as the law on
$(\Ce,\Cce)$ induced by $Y^\veps(\infty)$ and $Y^\veps(-\infty) $
where for all $t \geq 0$
\begin{equation}\label{extmon}
Y^\veps_t(\infty):=\lim_{x \to \infty}Y^\veps_t(x)
\qquad \text{ and }\qquad
Y^\veps_t(-\infty) := \lim_{x \to -\infty}Y^\veps_t(x).
\end{equation}
The above extension is well-defined since, by the comparison lemma for
SDEs, see \cite[Thm. 1.1]{ikedawatanabe},
\begin{equation}\label{Ymon}
x\leq x' \Rightarrow \bb{P}(Y^\veps_t(x) \leq Y^\veps_t(x') \quad \forall t\geq 0)=1.
\end{equation}
This section is divided into three subsections. In
Subsection~\ref{cdown} we prove trajectory properties
of the above extension. In Subsection~\ref{mprop} we
prove that the extension is Markovian. Finally, in
Subsection~\ref{cprop} we prove that the extension is
continuous with respect to the initial condition.
\subsection{\textbf{Coming down from infinity}}
\label{cdown}
We next explain when a solution to an SDE \emph{comes down from
  infinity}. This is based on entrance conditions at the boundary for
ODEs. In fact, as we shall see in Lemma~\ref{L2bound},  for all $t>0$, $\veps \in[0,1]$, the family $(Y^\veps_t(x), x \in \bb{R} )$
satisfies a uniform $L^2$ bound and so are a.s. finite for all positive
times. This section is organized as follows: First, we prove an
entrance condition for ODEs. Then we show the uniform bounds in
$L^2$. Finally, we define what is meant by the integral form of the
solution when the initial condition is $\pm\infty$.  We include an
explanation of these standard techniques for completeness and to
prepare for specific results we will need.
\subsubsection{\textbf{Entrance condition for ODEs}}
Let 
\begin{equation}\label{entrance}
  \begin{split}
    \mathfrak{L}:=\Big\{G:\mathbb{R}\to \mathbb{R}|& \textrm{ $G$ is locally
      Lipschitz and } \\
      &\hspace{3cm} -\infty< \int_R^\infty \frac{1}{G(u)}\ud u<0
    \textrm{ for some $R>0$} \Big\}
  \end{split}
\end{equation}\
be the space of velocity fields in which we are interested in.
\begin{lemma}[\textbf{Descent from infinity}]
\label{comedowsl}  
Given any fixed $G \in \mathfrak{L}$ and any $x \in \bb{R}$,
let $\psi(x):=(\psi_{t}(x),t\geq 0)$ be
the unique solution of the differential equation
\begin{equation}
\label{modeloinftycd}
\left\{
\begin{array}{r@{\;=\;}l}
\frac{\ud}{\ud t} \psi_t &G(\psi_t)\quad \textrm{ for }\quad  t\geq  0, \\
\psi_0 & x.
\end{array}
\right.
\end{equation}
Then for all $t>0$, the limit
$ \psi_t(\infty):=\lim_{x \to \infty}\psi_t(x)$
is well-defined and finite.
\end{lemma}
\begin{proof}
By the comparison lemma for ODEs we have
\begin{equation}\label{comparison}
 \forall \, t \geq 0,\; x_1, x_2\in \bb{R},\; x_1\leq x_2\Rightarrow \psi_t(x_1) \leq \psi_t(x_2). 
\end{equation}
Therefore, the limit $ \psi_t(\infty) := \lim_{x \to \infty}\psi_t(x)$
is well-defined but may be infinite.  In the sequel, we show that
$\psi_t(\infty)<\infty$ for any $t>0$.  Since $G$ is locally Lipschitz
and satisfies~\eqref{entrance}, $G(x)<0$ for all $x \geq R$.  Fix
$T>0$ and let $L : = \psi_T(R)$.  By uniqueness of solutions, the map
$t\mapsto \psi_t(R)$ is decreasing and for all $x \geq L$
$G(x)<0$. Let $F_{L,R}:[L, R] \to [0,T]$ be such that
$F_{L,R} (\psi_t(R)) = t$ for all $t \in [0,T]$ and note that
$F^{\prime}_{L,R}(u) = 1/G(u)$.  Since $F_{L,R}(R) = 0$ and $F_{L,R}(L) = T$,
we obtain
\begin{equation}\label{Lentrance}
  -T = \int_{L}^R F^{\prime}_{L,R}(u) \ud u =
  \int_{L}^R \frac{1}{G(u)} \ud u  .
\end{equation}
Now let $F_L:[L,\infty]\rightarrow [F_L(\infty),0]$ be given by 
$    F_L(x):=\int_{L}^{x}\frac{1}{G(u)}\ud u$
for  $x \in [L, \infty)$,
and set $F_L(\infty):= \lim_{x \to \infty}F_L(x)$.
By~\eqref{entrance} and~\eqref{Lentrance},
$F_L(\infty)\in(-\infty,0)$.
By~\eqref{comparison}, for any
$x \in [ R, \infty)$, $t \leq T$, we have
$\psi_t(x) \geq L $ and
$\frac{\ud}{\ud t}F_L(\psi_t(x)) = 1$. We thereby, conclude that
\begin{equation}\label{implicit}
F_L(\psi_{t}(x))
=t+F_L(x),\quad t \in[0,T].
\end{equation}
Again by~\eqref{comparison} and the continuity of $F_L$,
we can take the limit
$x \to \infty$ in~\eqref{implicit}  to obtain
\begin{equation}\label{wisejordan}
F_L(\psi_{t}(\infty))
= t + F_L(\infty),\quad  t \in[0,T].
\end{equation}  
Therefore
$\psi_{t}(\infty)< \infty$
for any
$t >0.$
\end{proof}
We may now define the \emph{extended ODE}. By Lemma~\eqref{comedowsl}
$\psi(\infty): = (\psi_t(\infty), t \geq 0)$ solves 
\begin{equation}
\label{modeloinftycdi}
\left\{
\begin{array}{r@{\;=\;}l}
\frac{\ud}{\ud t} \psi_t &G(\psi_t)\quad \textrm{ for }\quad t> 0, \\
\psi_0 & \infty,
\end{array}
\right.
\end{equation}
in the sense that $\psi_{0}(\infty) = \infty$,
and for any $t_0>0$, the following integral relation holds
\begin{equation}\label{integralform}
\psi_{t}(\infty)= \psi_{t_0}(\infty) +\int_{t_0}^t G(\psi_{s}(\infty)) \ud s \quad \textrm{ for all }\quad t \geq t_0.
\end{equation}
Equation~\eqref{integralform} is a consequence of the Fundamental Theorem of Calculus. 
Indeed, by~\eqref{wisejordan}, for any $s \geq t_0$,
$\psi_{s}(\infty) := F_{L}^{-1}(s + F_L(\infty)) \in \bb{R}$
and taking derivatives on both sides of~\eqref{wisejordan}, 
we obtain
\begin{equation}\label{chainrule}
\frac{\ud}{\ud s} \psi_{s}(\infty) = \frac{1}{F^{\prime}_L(\psi_{s}(\infty))} = G(\psi_{s}(\infty)).
\end{equation}
\subsubsection{\textbf{Uniform $L^2$ bounds}}
\label{vbound}
In what follows, we prove a second moment bound for any fixed $t>0$
for $(Y^\veps_t(x),x \in \bb{R})$.
\begin{lemma}[$L^2$-bound]\label{L2bound}  
If $\veps \in [0,1]$, then the process $Y^\veps(x)$ defined
by~\eqref{gensde} satisfies for any $t>0$, 
\begin{equation}\label{2mombound}
  \sup_{x \in \bb{R}}  \mathbb{E}[|Y^\veps_{t}(x)|^2]<\infty.
\end{equation}
\end{lemma}
\begin{proof} 
  Fix $\veps\in [0,1]$ and let $G_\veps: \bb{R} \to \bb{R}$ be given
  by $G_\veps(y):= yF_\veps(y)$ for all $y\in \mathbb{R}$ with
  $F_\veps$ defined in~\eqref{fieldnota}. Recall that $V$ satisfies
  Hypothesis~\ref{hyp1}, Hypothesis~\ref{hyp2} and
  Hypothesis~\ref{hyp4}.  Since $V^\prime$ is an odd function, it
  follows that $G_\veps(y) = G_\veps(\abs{y})\leq 0$ for all
  $y\in \mathbb{R}$. By Hypothesis~\ref{hyp2} and
  Hypothesis~\ref{hyp4} (Condition~\eqref{GC}) there is $c_*>0$ for
  which
\begin{equation}\label{ec:Ger}
G_\veps(y)\leq 
  -c_* |y|^{2+\alpha}\quad \textrm{ for all }\quad y\in \mathbb{R}.
\end{equation}
By It\^o's formula, for $t>0$, we have
\begin{equation}\label{ito2y}
\begin{aligned}
|Y^\veps_{t}(x)|^2= x^2 + 2
\int_0^t G_\veps(\yyyz{\veps}{s}{x}) \ud s+
t+
M^{x}_t,
\end{aligned}
\end{equation}
where $M^x=(M^{x}_t, t \geq 0)$ is a local martingale given by 
\begin{equation}\label{local}
 M^{x}_t=2\int_{0}^{t}\yyyz{\veps}{s}{x}\ud B_s.
\end{equation}
Since $G_\veps(y)\leq 0$ for all $y\in \bb{R}$,
a localization argument
yields that,
$\mathbb{E} [|Y^\veps_{t}(x)|^2] 
\leq  x^2 +t$ for any $t\geq 0$.
As a consequence, we have that $M^x$
is a true mean-zero  martingale.
Now, if we take expectation
on both sides of equality~\eqref{ito2y},
apply Fubini's theorem and use~\eqref{ec:Ger}
we obtain for all $t>0$ 
\begin{equation}\label{itoey}
\begin{split}
\mathbb{E}[|Y^\veps_{t}(x)|^2]&= x^2+2\int_0^t\mathbb{E}[G_\veps\prt{Y^\veps_{t}(x)}]  \ud s+t\\
&
\leq  x^2-2c_*\int_0^t\mathbb{E}[|Y^\veps_{s}(x)|^{2+\alpha}]\ud s +t.
\end{split}
\end{equation}
By Jensen's inequality we obtain
\begin{equation}\label{GCeps2}
  \begin{aligned}
   \mathbb{E}[|Y^\veps_{t}(x)|^{2+\alpha}]\geq 
   (\mathbb{E}[|Y^\veps_{t}(x)|^{2}])^{1+\alpha/2} \quad
   \textrm{for all }\quad t\geq 0.
    \end{aligned}
\end{equation}
By~\eqref{itoey} and~\eqref{GCeps2} if we denote 
$\psi^\veps_t(x):=\mathbb{E}[|Y^\veps_{t}(x)|^2]$ 
and let $\widetilde{G}(y) := -2c_*|y|^{1+\alpha/2} + 1$ for all $y\in \mathbb{R}$
then we have that
\begin{equation}\label{inedey}
  \begin{aligned}
    \frac{\ud}{ \ud t}\psi^\veps_t(x)
    &\leq  \widetilde{G}(\psi^\veps_t(x))\quad \textrm{ for } \quad t\geq 0.
  \end{aligned}
\end{equation}
Now, we let $(\widetilde{\psi}_t(x),t\geq 0)$ be the solution
of~\eqref{modeloinftycdi} for $G = \widetilde{G}$ and with initial
condition $\widetilde{\psi}_0(x) = \psi^\veps_0(x) = x^2$.  Observe
that $\widetilde{G} \in \mathfrak{L}$, where $\mathfrak{L}$ is defined
in~\eqref{entrance}.  To conclude~\eqref{2mombound}, we rely on
monotonicity and Lemma~\ref{comedowsl}. Indeed, for any $x \in \bb{R}$
and $t>0$
\begin{equation}\label{varboundt}
\bb{E}\crt{|Y^\veps_{t}(x)|^2}\leq \widetilde{\psi}_t(x)
  \leq   \lim_{z \to \infty}\widetilde{\psi}_t(z)
  =    \widetilde{\psi}_t(\infty)< \infty.
\end{equation}  
\end{proof}
\subsubsection{\textbf{Integral expression}}
\label{esol}
Now, we examine the integral form of the limit process.

\begin{proposition}[Integral form]
\label{Pcdfi}
For any fixed $\veps\in [0,1]$, let $F_\veps$ be as defined 
in~\eqref{fieldnota}. Then, the limit process
$Y^\veps_t(\sgn(x)\infty):=\lim_{r\to\infty}Y^\veps(\sgn(x)\cdot r)$ 
solves
\begin{equation}\label{Y0*}
\left\{
\begin{array}{r@{\;=\;}l}
  \ud Y_t & F_\veps(Y_t)  \ud t + \ud B_t\quad \textrm{ for  } \quad t> 0,\\[0.2cm]
   Y_0 & \sgn (x) \infty,
\end{array}
\right.
\end{equation}
in the sense that almost surely $\lim_{t \to 0}{Y}_t=\sgn(x)\infty$  
and for any $0<t_0< t$
\begin{equation}\label{zintegralsense0}
{Y}_t = {Y}_{t_0} + \int_{t_0}^t F_\veps ({Y}_s) \, \ud s + B_t - B_{t_0}.
\end{equation}
\end{proposition}
\begin{proof}
  Assume without loss of generality that $x>0$. By \eqref{Ymon}
  $Y^\veps_t(r)$ increases with $r$ and therefore the limit
  $\yyyz{\veps}{t}{\infty}$ exists. By~\eqref{2mombound} it follows that $\bb{P}(\yyyz{\veps}{t}{\infty}<\infty) = 1$ for
  any $t>0$.  Given $T> t_0>0$, we claim that, for every $\delta>0$
\begin{equation}\label{uniflim}
\lim_{r \to \infty}\bb{P}\Big(\sup_{t \in [t_0,T]}
\abs{F_\veps(\yyyz{\veps}{t}{\infty}) - F_\veps(Y^\veps_t(r))}
>\delta\Big)= 0. 
\end{equation}
The proof of~\eqref{uniflim} is postponed to Lemma~\ref{tightell}
below. Now,
note that, almost surely 
\begin{equation}\label{intx}
 Y^\veps_t(r) = Y^\veps_{t_0}(r)+ \int_{t_0}^ t
 F_\veps(Y^\veps_s(r))\,\ud s + B_t - B_{t_0}.  
\end{equation} 
By~\eqref{uniflim} we may take the limit inside the above integral and
therefore 
\begin{equation}\label{ieq}
 \bb{P}\Big( Y^\veps_{t}(\infty) = \yyyz{\veps}{t_0}{\infty}+ \int_{t_0}^t
  F_\veps(\yyyz{\veps}{s}{\infty}) \,\ud s + B_t - B_{t_0}  \quad \forall \, t>t_0\Big) =1. 
\end{equation}
\end{proof}
\begin{lemma}\label{tightell}
  For any $T>t_0>0$ and $\delta>0$ the equality in~\eqref{uniflim}
  holds true.
\end{lemma}
\begin{proof}
We first note that~\eqref{uniflim} is a consequence of 
\begin{equation}\label{locald1}
 \lim_{A \to \infty} \bb{P}\prt{\sup_{t \in [t_0,T]}\abs{Y^\veps_t(\infty)} >A } = 0
\end{equation}
and
\begin{equation}\label{locald2}
  \forall \delta>0 \quad
  \limsup_{r \to \infty}
  \bb{P}\prt{\sup_{t \in [t_0,T]}
    \abs{Y^\veps_{t}(\infty) - Y^\veps_t(r)} >\delta } = 0.
\end{equation}
Indeed, as $F_\veps$ is locally Lipschitz, for any $\delta>0$ and $A>0$
there is $\delta' = \delta'(\delta,A, \veps)>0$ for which
\begin{equation}
\begin{split}
\bb{P}\Big(\sup_{t \in [t_0,T]} \abs{F_\veps(Y^\veps_{t}(\infty)) - F_\veps(Y^\veps_t(r))}> \delta \Big)&\leq
\bb{P}\prt{\sup_{t \in [t_0,T]}\abs{Y^\veps_{t}(\infty) - Y^\veps_t(r)} >\delta' }\\
&\qquad +
\bb{P}\prt{\sup_{t \in [t_0,T]}\abs{Y^\veps_t(\infty)} >A }.
\end{split}
\end{equation}
By monotonicity and Lemma~\ref{L2bound},
$\lim_{r\to \infty} Y^\veps_t(r) = Y^\veps_t(\infty) \in \bb{R}$, for
any $t > 0$.  The pointwise limit, does not guarantee~\eqref{locald1}
and~\eqref{locald2}.  In order to obtain the above uniform bounds we
will show that the family $(Y^\veps(r), r\geq 0)$ is tight in the
space of continuous paths $C$.  Tightness in $C$ and pointwise
convergence imply uniform convergence of the family and the
bounds~\eqref{locald1} and~\eqref{locald2}.  By Aldous' tightness
criterion, see~\cite[Thm.~16.10, p.178]{Bil99} or
\cite[Thm.~4.1.3, p.51]{kipnislandim} we only need to show that
\begin{equation}\label{ald1}
  \forall \, t \in [t_0,T]\quad   \lim_{A \to \infty}
  \sup_{r \in \bb{R}}
  \bb{P}\prt{\abs{Y^\veps_t(r)} >A } = 0,
\end{equation}
and that
\begin{equation}\label{ald2}
  \forall \eta>0 \quad
  \lim_{\delta \to 0}
  \sup_{r \in \bb{R}}
  \bb{P}\prt{\sup_{\abs{t - s}<\delta}
    \abs{Y^\veps_{t}(r) - Y^\veps_s(r)} >\eta } = 0.
\end{equation}

\noindent{\bf Proof of~\eqref{ald1}.} 
By Lemma~\ref{L2bound} we have that
\begin{equation}\label{Cepst} 
C^\veps_ t := \sup_{r\in \bb{R}} \bb{E}[\abs{Y^\veps_t(r)}^2] < \infty.
\end{equation}
Therefore, by Chebyshev's inequality, for any $t \in [t_0,T]$ it follows that
\begin{equation}\label{supt}
  \sup_{r \in \bb{R}}\bb{P}\prt{\abs{Y^\veps_t(r)} >A }
  \leq \sup_{r \in \bb{R}}\frac{\bb{E}[\abs{Y^\veps_t(r)}^2]}{A^2}
  \leq \frac{C^\veps_t}{A^2} \to 0
  \quad \textrm{ as }\quad
  A \to \infty.
\end{equation}

\noindent{\bf Proof of~\eqref{ald2}.} 
We first write
$  Y^\veps_t(r) - Y^\veps_s(r) = \int_s^t F_\veps(Y^\veps_u(r)) \,\ud u + B_t - B_s.$
By the triangle inequality
and the continuity of Brownian motion,
to verify~\eqref{ald2}
it suffices to prove that for any $\eta>0$
\begin{equation}\label{intto0}
  \begin{aligned}
    \lim_{\delta \to 0}
    \sup_{r\in \bb{R}}
    \bb{P}\Big(\sup_{\abs{t-s}\leq \delta }
      \int_s^tF_\veps(Y^\veps_u(r)) \,\ud u> \eta\Big) =0.
\end{aligned}
\end{equation}
Fix  $K>0$, $\eta>0$, and let
$A^r_K: = \{\sup_{u \in [t_0,T]} |Y^\veps_u(r)|>K\}$. Now note that
there is $\delta = \delta(K,\eta,\veps)$ such that
$\delta \sup_{\abs{y} \leq K } \abs{F_\veps(y)} \leq \eta$ and
therefore for any $K>0$, $\eta>0$
\begin{equation}\label{totalprob0}
  \lim_{\delta \to 0}  \sup_{r \in \bb{R}} \bb{P}\Big(\sup_{\abs{t-s}\leq \delta } \int_s^tF_\veps(Y^\veps_u(r))\,\ud u> \eta \Big)  \leq
\sup_{r \in \bb{R}} \bb{P}\prt{A^r_K}.
\end{equation}
Since the left-hand side of the  above inequality does not depend on $K$,
we obtain that
\begin{equation}\label{supestimate}
  \lim_{\delta \to 0}
  \sup_{r \in \bb{R}}  \bb{P}\Big(\sup_{\abs{t-s}\leq \delta } \int_s^tF_\veps(Y^\veps_u(r))\,\ud u> \eta \Big)
  \leq \lim_{K \to \infty}  \sup_{r \in \bb{R}} \bb{P}(A^r_{K}).
\end{equation}
Hence, to obtain~\eqref{ald2} it is enough to prove that
$\lim_{K \to \infty}  \sup_{r \in \bb{R}} \bb{P}(A^r_{K})= 0$.
By~\eqref{ito2y} and~\eqref{local}, since
$G_\veps(y) = y F_{\veps}(y)\leq 0$ for all $y \in \bb{R}$
we have that
\begin{equation}\label{negeliminate}
\sup_{t\in [t_0,T]}   |Y^\veps_t(r)|^2 \leq
|Y^\veps_{t_0}(r)|^2+T+\sup_{t\in [t_0,T]}\abs{\int_{t_0}^t 2Y^\veps_s(r)\ud B_s}.
\end{equation}
The estimate on $A^r_K$ then becomes
\begin{equation}\label{split2mom}
  \begin{aligned}
    \bb{P} \prt{A^r_K} &\leq    \bb{P} \prt{|Y^\veps_{t_0}(r)|^2+T>K^2/2}
    +    \bb{P} \Big( \sup_{t\in [t_0,T]}\abs{\int_{t_0}^t2 Y^\veps_s(r)\ud B_s}>K^2/2\Big).
  \end{aligned}
\end{equation}
First, note that for $K^2/2 > 2T$ and by~\eqref{Cepst} we have
\begin{equation}\label{split1Z}
  \bb{P} \prt{|Y^\veps_{t_0}(r)|^2+T>K^2/2}
  \leq \bb{P} \prt{|Y^\veps_{t_0}(r)|^2>K^2/4}
  \leq 4 \frac{C^\veps_{t_0}}{K^2}. 
\end{equation}
To conclude, note that by Doob's submartingale inequality, It\^o's
isometry and~\eqref{Cepst} it follows that
\begin{equation}\label{split2Z}
  \begin{aligned}
    \bb{P} \Big(\sup_{t\in [t_0,T]}\abs{\int_{t_0}^t2 Y^\veps_s(r) \ud B_s}>K^2/2\Big) &\leq  \frac{ \int_{t_0}^T4 \bb{E} [ |Y^\veps_s(r)|^2] \ud s}{K^4/4} \leq  \frac{\int_{t_0}^T 16 C^{\veps}_s \, \ud s}{K^4}\\
      &\leq\frac{16 T \sup_{s \in [t_0,T]}C^{\veps}_s}{K^4}\to 0 \quad \text{ as }\quad K \to \infty,
    \end{aligned}
\end{equation}
where for the last passage we note that by~\eqref{varboundt},
$s \mapsto C^{\veps}_s$ is bounded by a continuous function in $(0,\infty)$.
\end{proof}

\subsection{\textbf{Markov property of the extended family}}
\label{mprop}
To prove that the extended family obtained by~\eqref{extmon} is Markovian,
one needs to verify the conditions
stated in Theorem~5.16 in~\cite[Ch~2, p.78]{KarSch98}.
These are the (i) compatible initial values,
(ii) the measurability of the transition laws,
and (iii) Markov property.

Note that for all $x \in \bb{R}$, $\veps>0$, and $t>0$,  $\bb{P}(Y^\veps_t(x) \in \{\pm\infty\}) = 0$ so
conditions (i)--(iii) are satisfied for all finite initial values.
It only remains to see that (i)--(iii) is satisfied for initial values
$x   \in  \{-\infty,\infty\} $. 
We only consider $x = \infty$, the case $x = -\infty$ is analogous.
For (i), note that
\begin{equation}
  \label{eq:28}
\bb{P}(Y^\veps_0(\infty)=  \infty)= \lim_{R\to \infty}\bb{P}(Y^\veps_0(\infty) >R)=
    \lim_{R\to \infty}\lim_{x \to \infty}\bb{P}(Y^\veps_0(x) >R)=1.
\end{equation}
For (ii), Proposition~\ref{Pcdfi} implies that for any $t_0>0$ and
$T>0$, the process $Y^\veps(n)$ converges
uniformly on the interval $[t_0,T]$ to $Y^\veps(\infty)$ as $n \to \infty$. 
Therefore by the monotonicity in~\eqref{Ymon} and the continuity of probability,
for all $k \in \bb{N}$, $t_1 <\ldots < t_k$, and $a_1, \ldots, a_k \in \bb{R}$ one has
\[
\liminf_{n \to \infty}\bb{P}(Y^\veps_{t_1}(n) >a_1, \ldots, Y^\veps_{t_k}(n) >a_k)
=\bb{P}(Y^\veps_{t_1}(\infty) >a_1, \ldots, Y^\veps_{t_k}(\infty) >a_k).
\]
The measurability follows by extending the above using Dynkin's $\pi$-$\lambda$ theorem.

Condition (iii) is a consequence of~\eqref{zintegralsense0} in Proposition~\ref{Pcdfi}. Indeed,
for any fixed $s>0$, if we let $(W_t: = B_{t + s} - B_{s}, t \geq 0)$
we have that almost surely  for any $t>0$
\[
\yyyz{\veps}{t + s}{\infty} = \yyyz{\veps}{s}{\infty}+ \int_{0}^{t}
  F_\veps(\yyyz{\veps}{u+s}{\infty}) \,\ud u + W_t.
\] 
Now, if we let $\mathbb{Y}_t: = \yyyz{\veps}{t + s}{\infty}$,
it follows that $(\mathbb{Y}_t, t \geq 0)$
solves the SDE~\eqref{gensde} with initial condition $\yyyz{\veps}{s}{\infty}$.
Furthermore, by Theorem~3.5 in~\cite[p.58]{Mao}, Equation~\eqref{gensde} is well-posed in $\bb{R}$ and,
by Lemma~\ref{L2bound}, $\yyyz{\veps}{s}{\infty} \in \bb{R}$ almost surely for any $s>0$.
Therefore, with the help of Theorem~9.1 in~\cite[p.86]{Mao} we conclude that
\[
\bb{P} (\yyyz{\veps}{t + s}{\infty}  \in A\vert Y^\veps_s(\infty) = y) = \bb{P}(\yyyz{\veps}{t }{y} \in A),
\]
which yields (iii) and concludes that the extended family is Markovian.

\subsection{\textbf{Continuity at infinity}}\label{cprop}
Let $\veps \in [0,1]$ be fixed. In this section we prove that the map
$x \mapsto Y_t^\veps(x)$ is continuous with respect to the total
variation distance for any $t>0$.
We first note that the map above is continuous in $\bb{R}$, see Theorem~1.3
in~\cite{donpensonzha16} and Theorem~1.1 in~\cite{Peng15} for a proof. Therefore, it only remains to verify the continuity at infinity. This is the content of the following lemma.
\begin{lemma}[\textbf{Continuity in total variation at infinity}]\label{Lem:Tvconv}
For any $\veps \in [0,1]$ and any $t>0$ it follows that
\begin{equation}\label{tvconv}
\lim_{x\to \infty} \dtv\prt{Y^\veps_{t}(x),
Y^\veps_{t}(\infty)}=0\qquad \textrm{ and }\qquad
\lim_{x\to -\infty}
\dtv\prt{Y^\veps_{t}(x), Y^\veps_{t}(-\infty)}=0.
\end{equation}
\end{lemma}
\begin{proof}
  We only prove the case for which $x \to + \infty$. The case when
  $x \to -\infty$ follows from the symmetry of $V^\veps$.
  Let $\mu^{\veps,x}_t$ be the measure in $\bb{R}^2$ defined by
\[
  \mu^{\veps,x}_t(\ud z_1,\ud z_2) = \bb{P}(Y^\veps_{t}(x)\in \ud z_1,
  Y^\veps_t(\infty) \in \ud z_2).
\]
For $s\in(0,t)$ we define $f: \bb{R}^2 \to [0,1]$ by
$f(z_1,z_2): = \dtv(Y^\veps_{t-s}(z_1), Y^\veps_{t-s}(z_2)) $. By the
Markovian property of the extended family
$\prt{\yyyz{\veps}{\cdot}{x}, x\in \overline{\bb{R}}}$ and Proposition~\ref{L:disintegration} for any $K>0$ we have that
\begin{equation}\label{desint}
  \begin{aligned}
     \dtv\prt{Y^\veps_{t}(x), Y^\veps_{t}(\infty)\,} &\leq
    \int_{\mathbb{R}^2} f(z_1,z_2)
    \mu^{\veps,x}_s(\ud z_1, \ud z_2)\\
    &\leq \int\limits_{|z_1|,|z_2|\leq K} f(z_1,z_2)
    \mu^{\veps, x}_s(\ud z_1, \ud z_2)\\
    &\qquad
    +\mathbb{P}(|Y^\veps_{t}(x)|>K) +
    \mathbb{P}(|Y^\veps_{t}(\infty)|> K).
  \end{aligned}
\end{equation}
By Lemma~\ref{L2bound} and Chebyshev's inequality it follows that
\begin{equation}\label{secmlim}
\limsup_{K \to \infty}\limsup_{x \to \infty}
\prt{\mathbb{P}(|Y^\veps_{t}(x)|>K) + \mathbb{P}(|Y^\veps_{t}(\infty)|> K)}=0.
\end{equation}
It suffices to show that for any $K>0$
\begin{equation}\label{TVr}
  \limsup_{x \to \infty}
  \int\limits_{|z_1|,|z_2|\leq K}
  f(z_1,z_2)
  \mu^{\veps,x}_s(\ud z_1, \ud z_2) =0.
\end{equation}
We now define for any $\delta>0$ 
\[\omega_{f,K}(\delta): = 
\max\{f(z_1,z_2)\colon \abs{z_1},\abs{z_2} \leq K,
\abs{z_1-z_2}\leq\delta \}.
\] 
Since $f$ is continuous and $f(z,z) = 0$, it follows that
\begin{equation}
  \label{eq:3}
  \omega_{f,K}(\delta)< \infty \qquad \text{ and } \qquad \lim_{\delta
    \to 0}\omega_{f,K}(\delta) = 0. 
\end{equation}
Given $\delta>0$, consider the following split of the integral in~\eqref{TVr},
\begin{equation}
  \label{eq:4}
  \begin{aligned}
     \int\limits_{|z_1|,|z_2|\leq K} f(z_1,z_2) \mu^{\veps,x}_s(\ud z_1, \ud
    z_2)&= \int\limits_{|z_1|,|z_2|\leq K} f(z_1,z_2)\;\Ind{\abs{z_1 -
        z_2}\leq\delta}\;
    \mu^{\veps,x}_s(\ud z_1, \ud z_2)\\
    &\qquad+ \int\limits_{|z_1|,|z_2|\leq K} f(z_1,z_2)\; \Ind{\abs{z_1
        - z_2}>\delta}\; \mu^{\veps,x}_s(\ud z_1, \ud z_2)\\
        &
    \leq \omega_{f,K}(\delta ) + \eta(x,\delta),
  \end{aligned}
\end{equation}
where
$\eta(x,\delta) : = \mu^{\veps,x}_s(\abs{z_1 - z_2} >\delta) =
\bb{P}(\abs{Y^\veps_s(x)- Y^{\veps}_s(\infty)} > \delta)$.  By~\eqref{extmon}, 
\[
\lim_{x \to \infty} \eta(x,\delta) = 0\quad  \textrm{ for any } \delta>0.
\]
To conclude the proof of Lemma~\ref{Lem:Tvconv} we note that, by~\eqref{eq:3}
\begin{equation}  \label{eq:5}
  \limsup_{x \to \infty}  \int_{|z_1|,|z_2|\leq K} f(z_1,z_2) 
  \mu^{\veps,x}_s(\ud z_1, \ud z_2) \leq \inf_{\delta>0} \omega_{f,K}(\delta) = 0.
\end{equation}
\end{proof}
  
\section{\textbf{Uniform bounds}}\label{apmore}
In this section we prove the bounds~\eqref{compactin}
and~\eqref{fdelta}.

\subsection{\textbf{Uniform entrance in a compact}}\label{AppUEC}
The bound~\eqref{compactin} is a consequence of the following proposition.
\begin{proposition}
  \label{P:compactin}
  For any  $\eta>0$ and  $a>0$ there are $b>0$ and $\delta \in (0, \eta)$ such that
  \begin{equation}\label{compactin2}
    \sup_{\veps \in  [0,1]} \bb{P}\prt{\yyyz{\veps}{\delta}{\infty} \notin [a,b]} \leq  \eta.
  \end{equation}
\end{proposition}
The proof of Proposition~\ref{P:compactin} is based on the two following statements, whose proofs are given afterwards.
  \begin{align}\label{smallerthana}
    \text{  For any $a>0$, }\quad &
    \lim_{\delta \to 0} \sup_{\veps \in [0,1]}
    \mathbb{P}(\abs{Y^\veps_\delta(\infty)} \leq a) = 0.\\
    \label{largerthanb}
    \text{  For any $\delta>0$, }\quad &
    \lim_{b \to \infty} \sup_{\veps \in [0,1]}
    \mathbb{P}(\abs{Y^\veps_\delta(\infty)} >b) = 0. 
  \end{align}  
\begin{proof}[Proof of Proposition~\ref{P:compactin}]
Given $\eta>0$ and $a>0$, by~\eqref{smallerthana}
there is $\delta\in(0, \eta)$ such that
$\mathbb{P}(\abs{Y^\veps_\delta(\infty)} < a) < \eta/2$ for any
$\veps \in[0,1] $.
Next, by~\eqref{largerthanb}, we choose $b>0$ such that
$ \mathbb{P}(\abs{Y^\veps_\delta(\infty)} > b) < \eta/2$ for any
$\veps \in[0,1] $.
With this, we conclude that for every $\veps \in[0,1]$
\begin{equation}\label{conclude_ab}
\bb{P}(\abs{Y^\veps_\delta(\infty)} \notin [a,b]) =
\bb{P}(\abs{Y^\veps_\delta(\infty)} > b) +
\bb{P}(\abs{Y^\veps_\delta(\infty)} < a) < \eta.
\end{equation}
\end{proof}

In what follows we prove~\eqref{smallerthana} and~\eqref{largerthanb}.
\begin{proof}[Proof of~\eqref{smallerthana}.]
Fix any $a>0$. For $D>0$, let $\Omega(\delta, D) := \{\sup_{t \leq
  \delta} \abs{B_t} \leq D\}$. 
Fix $x := 2(a + D)$ and choose $K := 2x$. Now let $\sigma:
= \tau(a)\wedge \tau(K)$ where $\tau(v) = \tau(a,x,\veps ): =
\inf\{t>0:Y^\veps_t(x) = v\}$ for  $v \in \bb{R}$. 
Note that
$Y^\veps_\delta(\infty) \geq Y^\veps_\delta(x)$, and that almost surely
\begin{equation}\label{intesY}
  Y^\veps_{\delta \wedge\sigma}(x)  = x - \int_0^{\delta \wedge \sigma} \abs{F_\veps (Y^\veps_s(x))}\ud s  + B_{\delta\wedge  \sigma}.
\end{equation}
By~Hypothesis~\eqref{hyp2}  and~\eqref{fieldnota} it follows that
\begin{equation}\label{uupperFe}
C(K):=  \sup_{\veps \in [0,1]}\sup_{|y| \leq K}\abs{F_\veps(y)}=\sup_{\veps \in [0,1]}\sup_{|y| \leq K}\abs{\frac{V^{\prime}(b_\veps y)}{b^{1 + \alpha}_\veps}}  < \infty.
\end{equation} 
Given $D>0$, there is $\eta>0$ such that $\delta \in (0,\eta)$, implies
$\delta <\sigma$ on $\Omega(\delta, D)$. Indeed, by~\eqref{uupperFe}
there is $C=C(D)>0$ which allows~\eqref{intesY} to be bounded by
\begin{equation}\label{Y-downD-upD}
Y^\veps_{\delta \wedge \sigma}(x)  \geq x  -
  C(\delta\wedge \sigma) - D \geq x - D  -   C\delta 
  \qquad  \text{ and }  \qquad
  Y^\veps_{\delta \wedge \sigma}(x)  \leq x + D< K.
\end{equation}
Since $x - D>2a$, for any $ \delta\in (0, C^{-1}a)$ it follows
that $Y^\veps_{\delta \wedge \sigma}(x) \in (a,K)$. 
In conclusion, on  the event $\Omega(\delta,D)$,
we have that $ Y^\veps_{\delta}(\infty) \geq
Y^\veps_{\delta}(x)=Y^\veps_{\delta \wedge \sigma}(x) > a$. 
Since $\bb{P}(\Omega(\delta, D)) \to 1$ as $\delta \to 0$, the proof 
of~\eqref{smallerthana} is complete.
\end{proof}
 
\begin{proof}[Proof of~\eqref{largerthanb}]
We start the proof with uniform  $L^2$ bounds, that is, we
prove that for any $t>0$
\begin{equation}\label{uniformgepsba}
  \sup_{\veps \in[0,1]}  \mathbb{E}[|Y^\veps_{t}(x)|^2]<\infty.
\end{equation}
For any $x\in \mathbb{R}$ and $t\geq 0$, inequality~\eqref{varboundt} yields
\begin{equation}
\bb{E}\crt{|Y^\veps_{t}(x)|^2}\leq \widetilde{\psi}_t(x),
\end{equation} 
where $(\widetilde{\psi}_t(x),t\geq 0)$ is the solution of
\begin{equation}
\left\{
\begin{array}{r@{\;=\;}l}
    \frac{\ud}{ \ud t}\psi_t(x)
    &  \widetilde{G}(\psi_t(x)),\\
\psi_0(x) & x
\end{array}
\right.
\end{equation}
with $\widetilde{G}(y) := -2c_*|y|^{1+\alpha/2} + 1$ for all $y\in \mathbb{R}$ for some $c_*>0$.
The monotone convergence theorem  with the help of Lemma~\ref{comedowsl} implies
\begin{equation}\label{ecu:abt}
\bb{E}\crt{|Y^\veps_{t}(\infty)|^2}\leq \widetilde{\psi}_t(\infty)<\infty.
\end{equation}
To conclude~\eqref{largerthanb} note that \eqref{ecu:abt} yields,
\begin{equation}
\sup_{\veps \in [0,1]} \mathbb{P}(\abs{Y^\veps_\delta(\infty)} >b)\leq 
\frac{\widetilde{\psi}_t(\infty)}{b^2} \to 0
\quad \textrm{ as } \quad b \to \infty.
 \end{equation}
\end{proof}

\subsection{\textbf{Uniform convergence in total variation distance}}
\label{p:fdelta}
The bound~\eqref{fdelta} is a consequence of the following proposition.

\begin{proposition}\label{P:fdelta}
For any $t>0$ , $a>0$, $b>a$ and  $\eta>0$  there is  $\veps_0 >0$ for which
\begin{equation}\label{eq:6}
\sup_{0\leq \veps< \veps_0}\sup_{x \in [a,b]}\dtv\prt{Y^0_{t}(x), Y^{\veps}_{t}(x)}< \eta.
\end{equation}
\end{proposition} 
\begin{proof}
By Theorem~5.1 in~\cite{KabLipShi86}, we have
\begin{equation}\label{eq:27}
\abs{\dtv\prt{Y^0_{t}(x), Y^{\veps}_{t}(x)}}^2
\leq 2 \int_0^{t}\bb{E}\crt{\abs{F_0(Y^0_{s}(x)) -F_\veps(Y^\veps_{s}(x))}^2 }\ud s. 
\end{equation}
To conclude~\eqref{eq:6} we show that
\begin{equation}
  \label{eq:10}
  \lim_{ \veps_0 \to 0}
  \sup_{0\leq \veps< \veps_0}
  \sup_{x \in [a,b]}\sup_{s \in [0,t]}
  \bb{E}\crt{\abs{F_0(Y^0_{s}(x))-F_\veps(Y^\veps_{s}(x))}^2 }
  =0.
\end{equation}
First, we define the event 
\begin{equation}\label{AM}
A_{M,\veps} = A_{M,\veps}(x,t) := \Big\{\sup_{s \in [0,t]} \vert Y^0_{s}(x)\vert
\vee \abs{Y^\veps_{s}(x)} \leq M\Big\}.
\end{equation}
Now,  for any $M>0$, we may write the expectation term in~\eqref{eq:10} as
\begin{equation}
  \label{eq:53}
  \begin{aligned}
    \bb{E}\crt{\abs{F_0(Y^0_{s}(x))-F_\veps(Y^\veps_{s}(x))}^2
      \Ind{A_{M,\veps}}}
    +\bb{E}\crt{\abs{F_0(Y^0_{s}(x))-F_\veps(Y^\veps_{s}(x))}^2
      \Ind{A^c_{M,\veps}}} .
  \end{aligned}
\end{equation}
Since $M>0$ is arbitrary, to prove~\eqref{eq:10} it suffices to show
that for any $M>0$
\begin{equation}
  \label{eq:54}
  \lim_{ \veps_0 \to 0}
  \sup_{0\leq \veps< \veps_0}
  \sup_{x \in [a,b]}\sup_{s \in [0,t]}
  \bb{E}\crt{\abs{F_0(Y^0_{s}(x))-F_\veps(Y^\veps_{s}(x))}^2
  \Ind{A_{M,\veps}}} =0
\end{equation}
and that 
\begin{equation}
  \label{eq:55}
  \lim_{M \to \infty}
  \sup_{ \veps \in [0, 1]}
  \sup_{x \in [a,b]}\sup_{s \in [0,t]}
  \bb{E}\crt{\abs{F_0(Y^0_{s}(x))-F_\veps(Y^\veps_{s}(x))}^2\Ind{A^c_{M,\veps}}} =0.
\end{equation}

\noindent\textbf{Proof of~\eqref{eq:54}.}
Let $\Delta^\veps_s(x):=Y^\veps_s(x) - Y^0_s(x)$.
We recall that for each $\veps \in [0,1]$ the process
$Y^\veps(x) = (Y_t^\veps(x), t \geq 0)$ 
solves~\eqref{gensde}. Moreover, the processes $(Y^\veps(x), \veps \in
[0,1])$ are coupled with the same noise,
i.e., there is $(B_t, t \geq 0)$ a Brownian motion under $\bb{P}$ such
that for all $\veps \in [0,1]$ we have almost surely that
\begin{equation}
  \label{couple}
  Y^\veps_{t}(x)  = x + \int_0^{t}F_\veps
    (Y^\veps_s(x))\ud s  + B_t \quad \text{ for all }\quad t >0.
\end{equation}
Now, we may use~\eqref{couple}
to write 
\[
  \Delta^\veps_s(x) = Y^\veps_s(x) - Y^0_s(x) =
  \int_0^s [F_\veps(Y^\veps_u(x)) - F_0(Y^0_u(x))]\, \ud u.
\]
By the mean value theorem we have 
\begin{equation}\label{difprocas}
  \begin{aligned}
    \Delta^\veps_s(x) &= \int_0^s [F_\veps(Y^\veps_u(x)) -F_\veps( Y^0_u(x))- (F_0(Y^0_u(x))-F_\veps( Y^0_u(x))] \, \ud u\\
    &  = \int_0^s [F^{\prime}_\veps(\Theta^\veps_u)\Delta^\veps_u(x) - (F_0(Y^0_u(x))-F_\veps( Y^0_u(x))] \, \ud u,\\
  \end{aligned}
\end{equation}
where
$\Theta^\veps_u \in \big(Y^\veps_u(x)\wedge Y^0_u(x) , Y^\veps_u(x)\vee Y^0_u(x)
\big)$ for all $u \geq 0$.  By the convexity of $V$, it follows that
$F^{\prime}_\veps(\Theta^\veps_u) \leq 0$. By the chain rule and the fact that
$\abs{x} \leq 1 + x^2 $ for all $x \in \bb{R}$ we have that
\begin{equation}
  \label{eq:7}
  \begin{aligned}
    |\Delta^\veps_s(x)|^2 &= |\Delta^\veps_0(x)|^2
    + \int_0^s 2 \Delta^\veps_u(x) \ud \Delta^\veps_u(x)\\  
    &= 2\int_0^s [F^{\prime}_\veps(\Theta^\veps_u)|\Delta^\veps_u(x)|^2 - \Delta^\veps_u(x)(F_0(Y^0_u(x))-F_\veps( Y^0_u(x))] \, \ud u\\
    & \leq 2 \int_0^s (1 + |\Delta^\veps_u(x)|^2 )
    \abs{F_0(Y^0_u(x))-F_\veps( Y^0_u(x))}\, \ud u.
  \end{aligned}
\end{equation}
If we let
$\psi_s^\veps(x): = \sup_{u \in [0,s]} |\Delta^\veps_s(x)|^2\Ind{A_{M,\veps}}$
and 
$K(M, \veps): = \sup_{\abs{z} \leq M} \abs{F_0(z) -
  F_\veps(z)} $
 we obtain that $ \psi_s^\veps(x) \leq  2 K(M,\veps)\int_0^s (1 +
 \psi_u^\veps(x)) \ud u $ for any $x \in \bb{R}$.
Now, for any fixed $M>0$, by Hypothesis~\ref{hyp2} we have $K(M,\veps) \to 0$ as $\veps \to 0$.
This implies that for any $\eta>0$ and $M>0$
there is $\veps_0$ such that 
$\sup_{\veps \in [0,\veps_0]}  \sup_{x \in \bb{R}}\psi_t^\veps(x) \leq \eta$.
This completes the proof of~\eqref{eq:54}.

\noindent\textbf{Proof of~\eqref{eq:55}.}
Since $(r_1 + r_2)^2 \leq 2 \prt{r_1^2 + r_2^2}$ for any
$r_1, r_2 \in \bb{R}$, the expectation in~\eqref{eq:55} can be bounded
by
\[
2 \bb{E}\crt{\abs{F_0(Y^0_{s}(x))}^4\Ind{A^c_{M,\veps}(x,s)}} +
  2\bb{E}\crt{\abs{F_\veps(Y^\veps_{s}(x))}^4 \Ind{A^c_{M,\veps}(x,s)}},
\]
it remains to show that
\begin{equation}
  \label{eq:56}
\lim_{M \to \infty}   \sup_{\veps \in [0, 1]}
   \sup_{x \in [a,b]}\sup_{s \in [0,t]}
\bb{E}\crt{\abs{F_\veps(Y^\veps_{s}(x))}^4 \Ind{A^c_{M,\veps}(x,s)}} =0.
\end{equation}
By Cauchy--Schwarz inequality, we have
\begin{equation}
  \label{eq:59}
  \bb{E}\crt{\abs{F_\veps(Y^\veps_{s}(x))}^4 \Ind{A^c_{M,\veps}}} \leq \prt{\bb{E}\crt{\abs{F_\veps(Y^\veps_{s}(x))}^8}}^{1/2} \cdot \prt{\bb{P}\prt{A^c_{M,\veps}}}^{1/2}.
\end{equation}
We now claim that
\begin{equation}
  \label{F-Y-eps-8-bound}
    \sup_{\veps \in [0, 1]}\sup_{x \in [a,b]}\sup_{s \in [0,t]}
  \bb{E}\crt{\abs{F_\veps(Y^\veps_{s}(x))}^8} < \infty,
\end{equation}
and that, recall~\eqref{AM}, for any fixed $t>0$
\begin{equation}\label{eq:9}
 \sup_{\veps \in [0, 1]}\sup_{x \in [a,b]}
  \bb{P}(A^c_{M,\veps}(x,t)) \to 0\quad  \text{ as } \quad  M \to \infty.
\end{equation}
From~\eqref{eq:59},~\eqref{F-Y-eps-8-bound} and~\eqref{eq:9} we conclude~\eqref{eq:55}.
It remains to prove~\eqref{F-Y-eps-8-bound} and~\eqref{eq:9}.

\noindent\textbf{Proof of~\eqref{F-Y-eps-8-bound}.}
We first note that Hypothesis~\ref{hyp2} and 
Hypothesis~\ref{hyp4} (Condition~\eqref{Ge}) imply that there is
$\tilde{c}>0 $ such that 
 $\abs{V^{\prime}(z)}
\leq \tilde{c} |z|^{1+\alpha}\exp(z^2)$
for any $z\in \bb{R}$.
Therefore
\begin{equation}
  \label{eq:62}
  \begin{aligned}
 \bb{E}\crt{\abs{F_\veps(Y^\veps_{s}(x))}^8 }&=
  \bb{E}\crt{\abs{V^{\prime}(b_\veps Y^\veps_s(x))/b^{1+\alpha}_\veps}^8}\\
  &\leq  \tilde{c}^8\bb{E}\crt{\abs{Y^\veps_s(x)}^{8(1 + \alpha)}\abs{1 +
    \exp(8 b_\veps \abs{Y^\veps_s(x)}^2)}}.   
  \end{aligned} 
\end{equation}
Now, since $b_\veps \to 0$ as $\veps \to 0$ and $|z|^{1 + \alpha} \leq \exp(z^2)$ for all  $z\in \mathbb{R}$, there is $\tilde{C}>0$ for
which
\begin{equation}
  \label{bound-exp}
   \bb{E}\crt{\abs{F_\veps(Y^\veps_{s}(x))}^8 }   \leq  \tilde{C}\bb{E}\crt{\exp(\abs{Y^\veps_s(x)}^2)}.   
\end{equation}
To conclude, we now show that
\begin{equation}
  \label{eq:52}
\sup_{\veps \in [0, 1]}
\sup_{x \in [a,b]}\sup_{s \in [0,t]}
\bb{E}\crt{\exp(\abs{Y^\veps_s(x)}^2)}
= C(t,a,b)< \infty.
\end{equation}
Indeed, by It\^o's formula for $H(z) = \exp(z^2)$, $z\in \mathbb{R}$, we have that
\begin{equation}
  \label{eq:61uno}
  H(Y^\veps_{s}(x)) =  H(x)+ \int_0^s
  H(Y^\veps_{u}(x))\big(2Y^\veps_{u}(x)F_\veps(Y^\veps_{u}(x)) +
  2(Y^\veps_{u}(x))^2 +  1\big) \ud u + M_s,
\end{equation}
where $(M_s, s \geq 0)$ is a local martingale. Recall that $G_\veps (z)
=zF_\veps(z)$ for all $z \in \bb{R}$. 
By~\eqref{ec:Ger} 
we deduce that 
\begin{equation}\label{eq:63}
  \sup_{\veps \in [0, 1]}
 \sup_{z \in \bb{R}} \prt{2zF_\veps(z) + 2 z^2 + 1}  = C < \infty.
\end{equation}
Therefore, if we let $\tau_K = \tau(K, \veps, x) : = \inf\{s>0\colon
\abs{Y^\veps_s(x)} > K \}$ we obtain 
\begin{equation}
  \label{eq:61}
  \bb{E}\crt{H(Y^\veps_{s\wedge\tau_K}(x))} \leq  H(x)+ C\int_0^{s}
 \bb{E}\crt{H(Y^\veps_{u\wedge\tau_K}(x))} \ud u.
\end{equation}
Now, by Gr\"onwall's inequality we obtain for $x \in [a,b]$ and $s \in
[0,t]$ that
\begin{equation}
  \label{eq:17}
  \bb{E}\crt{H(Y^\veps_{s\wedge\tau_K}(x))} \leq  H(x) \exp(Cs)
  \leq\big( H(a) + H(b) \big) \exp(Ct).
\end{equation}
Since the constant $C$ in~\eqref{eq:61} does not depend on $\veps$,  Fatou's lemma implies
\begin{equation}
  \label{eq:68}
  \begin{aligned}
    \sup_{\veps \in [0, 1]}\sup_{x \in [a,b]}\sup_{s\in[0,t]}
    \bb{E}\crt{\exp(\abs{Y^\veps_s(x)}^2)}&=\sup_{\veps \in [0,
      1]}\sup_{x \in [a,b]}\sup_{s\in[0,t]} 
    \bb{E}\crt{\liminf_{K \to \infty}H\big(Y^\veps_{s\wedge
        \tau_K}(x)\big)}\\
 & \leq \sup_{\veps \in [0, 1]}\sup_{x \in [a,b]}\sup_{s\in[0,t]}\liminf_{K \to \infty}     
    \bb{E}\crt{H(Y^\veps_{s\wedge\tau_K}(x))}\\
& \leq\big( H(a) + H(b)\big)\exp( Ct),
  \end{aligned}
\end{equation}
which yields~\eqref{eq:52}.
This completes the proof of~\eqref{F-Y-eps-8-bound}.

\noindent\textbf{Proof of~\eqref{eq:9}.}
Since $zF_\veps(z) = G_\veps(z)\leq 0$ for all $z \in \bb{R}$, it
follows from~\eqref{ito2y} that 
\begin{equation}
  \label{eq:70}
  \sup_{ s \in [0,t] }|Y^\veps_{s}(x)|^2 \leq x^2 + t 
 + \sup_{ s \in [0,t] } \abs{M^{\veps, x}_s},
\end{equation}
where $M^{\veps, x} =(M^{\veps,x}_s, s\geq 0)$ is the local martingale
given by~\eqref{local}.
Therefore, for any $x \in [a,b]$ and any $M$ such that $M> a^2 + b^2 + t$ 
\begin{equation}
  \label{eq:71}
  \begin{aligned}
    \bb{P}\prt{\sup_{ s \in [0,t] }|Y^\veps_{s}(x)|^2>M}
    &\leq
    \bb{P}\prt{x^2 + t + \sup_{ s \in [0,t] } \abs{M^{\veps, x}_s}> M}\\
    &=
    \bb{P}\prt{\sup_{ s \in [0,t]}|M^{\veps, x}_s|^2>(M-x^2 -t)^2}\\  
    &\leq \frac{\bb{E}\crt{|M^{\veps, x}_t|^2}}{(M-x^2 -t)^2} 
    = \frac{C_{\veps,t,x}}{(M-x^2 -t)^2},
  \end{aligned}
\end{equation}
where the second inequality follows from Doob's $L^2$ submartingale
inequality and by It\^o's isometry,  $C_{\veps,t,x} := 4\int_0^t
\bb{E}\crt{|\yyyz{\veps}{s}{x}|^2} \ud s$. Now, since $z^2 \leq
\exp(z^2)$ for all $z \in \bb{R}$, 
by~\eqref{eq:68} it follows that
\begin{equation}
  \label{eq:73}
  \sup_{\veps \in [0, 1]}\sup_{x \in [a,b]}\sup_{s \in [0,t]}
  C_{\veps,t,x} = C'(t,a,b) < \infty.
\end{equation}
Now, recall the definition in~\eqref{AM}. Since $x^2 \leq a^2 + b^2$ for any $x \in [a,b]$, it follows that for any $M$ such that $M> a^2 + b^2 + t$
\begin{equation}
  \label{eq:72}
\begin{aligned}
 \sup_{\veps \in [0, 1]}\sup_{x \in [a,b]}
 \bb{P}(A^c_{M,\veps}(x,t) ) \leq
  \frac{C'(t,a,b)}{(M-a^2 - b^2 -t)^2} 
\end{aligned}
\end{equation}
and so, we obtain~\eqref{eq:9} and thereby conclude the proof of
Proposition~\ref{P:fdelta}.
\end{proof} 
\section{\textbf{Convergence of invariant measures: details}}
\label{app:convergence-invariant-measure}
Recall the notation introduced above~\eqref{tv_scale_inv}, that is,
$Y_\infty\stackrel{d}= \nu$, $X^\veps_\infty \stackrel{d}=\mu^\veps$
and $\widetilde{X}^\veps_\infty \stackrel{d}=\widetilde{\mu}^\veps$.
By Lemma~\ref{medidainvariante} it follows that
\begin{equation}\label{densityY-app}
\nu(\ud z)=C^{-1}\exp\big(-2 V_0(z)\big)\ud z,
\end{equation}
where $C$ is a normalization constant, $V_0(z):=(2 +
\alpha)^{-1}C_0|z|^{2+\alpha}$ with $\alpha$ and $C_0$ defined in
Hypothesis~\ref{hyp2}. 
Similarly,
 $\mu^\veps$ and
$\widetilde{\mu}^\veps$ is the density of $X^{\veps}_\infty$ and
$\widetilde{X}^{\veps}_\infty$
given by,
 and they are given by 
\begin{equation}\label{invariant-formula}
\mu^\veps(\ud z)=C_{\veps}^{-1}
\exp\left(-2\frac{V(z)}{\veps}\right)\ud z\qquad \textrm{ and }\qquad
\widetilde{\mu}^\veps(\ud z)={\widetilde{C}_{\veps}^{-1}}\exp\left(-2\frac{\widetilde{V}(z)}{\veps}\right)\ud z.
\end{equation} 
respectively. For $(b_\veps,\veps\in [0,1))$ as defined
in~\eqref{abscales}, the same argument leading to \eqref{eq:cambio}
yields 
that the density of
$\frac{X^{\veps}_\infty}{b_{\veps}}$
and $Y^\veps_\infty=\frac{\widetilde{X}^{\veps}_\infty}{b_{\veps}}$
is given by 
\begin{equation}\label{eq:cambio-app}
b_{\veps}{C_{\veps}^{-1}}\exp\left(-2\frac{V(b_{\veps}z)}{\veps}\right)\ud
z
\qquad \textrm{ and } \qquad
b_{\veps}{\widetilde{C}_{\veps}^{-1}}\exp\left(-2\frac{\widetilde{V}(b_{\veps}z)}{\veps}\right)\ud z.
\end{equation} 

\begin{lemma}[Asymptotic coupling of the invariant measures]\label{lem:couplingINV}
  For each $\veps>0$, let $X^\veps_\infty$ and
  $\widetilde{X}^\veps_\infty$ be the random variables whose
    distributions are the invariant measures
  of the SDE given by \eqref{modelo} and \eqref{modelote}, 
  respectively.
  Assume that the potential $V$ of \eqref{modelo} satisfies
  Hypotheses~\ref{hyp1}, \ref{hyp2}, and \ref{hyp3} and assume that
  the potential $\widetilde{V}$ of \eqref{modelote} satisfies
  Hypotheses~\ref{hyp1}, \ref{hyp2}, and \ref{hyp4}.  Under those
  assumptions,
it follows that
\begin{equation}\label{eq:limitesm}
\lim\limits_{\veps\to 0}\dtv\left(\frac{X^\veps_\infty}{b_\veps},Y_\infty\right)=0\qquad \textrm{ and }\qquad
\lim\limits_{\veps\to 0}\dtv\left(\frac{\widetilde{X}^\veps_\infty}{b_\veps},Y_\infty\right)=0,
\end{equation}
where $(b_\veps,\veps\in [0,1))$ is defined in~\eqref{abscales}.
In particular,
  \begin{equation}
    \label{eq:particular}
    \lim\limits_{\veps\to
      0}\dtv(X^\veps_\infty,\widetilde{X}^\veps_\infty)
    =\lim\limits_{\veps\to 0}\dtv(\mu^\veps,\widetilde{\mu}^\veps)=0.
\end{equation}
\end{lemma}
The following two lemmas will be instrumental for the proof of Lemma \ref{lem:couplingINV}.\begin{lemma}[Uniform convergence of the potentials]\label{claim1in}
Under the same hypotheses of Lemma \ref{lem:couplingINV}, for any $K>0$ it follows that
\begin{equation}\label{approachV}
\lim\limits_{\veps \to 0}
\sup\limits_{|z|\leq K}\left|\frac{V(b_{\veps} z)}{\veps}-V_0(z) \right|=0
\end{equation}
and 
\begin{equation}\label{approachVtilde}
\lim\limits_{\veps \to 0}
\sup\limits_{|z|\leq K}\left|\frac{\widetilde{V}(b_{\veps} z)}{\veps}-V_0(z) \right|=0.
\end{equation}
where $V_0(z) =(2 + \alpha)^{-1}C_0|z|^{2+\alpha}$ for any $z\in \mathbb{R}$ with $\alpha$ and $C_0$ defined in Hypothesis~\ref{hyp2}.
\end{lemma}

\begin{proof}
In the sequel, we show~\eqref{approachV}.
Let $K>0$ and $\eta>0$ be fixed and
define $\widetilde{\eta}:= \eta K^{-1}>0$.
By~\eqref{abscales}, $b_\veps\to 0$, as $\veps\to 0$.
Now, by Hypothesis~\ref{hyp2} there exists
$\veps_0=\veps_0(K,\widetilde{\eta})>0$
such that for any $|z|\leq K$ and $\veps<\veps_0$,
\begin{equation}\label{luvb}
C_0|z|^{1+\alpha}\sgn(z)-\widetilde{\eta}<\frac{V^{\prime}(b_\veps z)}{b^{1+\alpha}_\veps}<C_0|z|^{1+\alpha}\sgn(z)+\widetilde{\eta}.
\end{equation}
If we integrate each term from $0$ to $x$
in the above inequality,
use Hypothesis~\ref{hyp1}
 and note that, by~\eqref{abscales},
$b^{2+\alpha}_\veps=\veps$
we obtain that for any
$|z|\leq K$ and $\veps<\veps_0$
\begin{equation}
\sup\limits_{|z|\leq K}\left|\frac{V(b_\veps z)}
  {\veps}-V_0(z)\right|\leq  \widetilde{\eta} K=\eta.
\end{equation}
Since $\eta>0$ is arbitrary,
the proof of~\eqref{approachV} is complete. 

By construction, we stress that $\widetilde{V}$ also satisfies Hypothesis~\ref{hyp1} and Hypothesis~\ref{hyp2}. Hence the proof of~\eqref{approachVtilde} is analogous.
\end{proof}

\begin{lemma}[Convergence of  the normalizing constants]\label{lem:constantes}
Under the same hypotheses of Lemma \ref{lem:couplingINV} it follows that
\begin{equation}
  \label{eq:14}
\lim_{\veps \to 0} b_\veps C^{-1}_\veps =  C
\end{equation}
and
\begin{equation}
  \label{eq:1456}
\lim_{\veps \to 0} b_\veps \widetilde{C}^{-1}_\veps = C.
\end{equation}
\end{lemma}
\begin{proof}
By the change of variables $z \mapsto b_\veps z$ we obtain that $\frac{C_\veps}{b_\veps}=\int_{\mathbb{R}}e^{-\frac{2}{\veps}V(b_\veps z)}\ud z$.
Now,  by Lemma~\ref{claim1in} and Fatou's lemma we have
\begin{equation}\label{desi1}
C=\int_{\mathbb{R}}e^{-2V_0(z)}\ud z\leq \liminf_{\veps\to 0}\int_{\mathbb{R}}e^{-\frac{2}{\veps}V(b_\veps z)}\ud z = \liminf_{\veps\to 0} \frac{C_\veps}{b_\veps}.
\end{equation}
Similarly, for $\widetilde{V}$ we obtain
\begin{equation}\label{desi1one}
C\leq  \liminf_{\veps\to 0} \frac{\widetilde{C}_\veps}{b_\veps}.
\end{equation}
We next show that
$\limsup\limits_{\veps\to 0}\frac{C_\veps}{b_\veps} \leq C$. 
Note first that
\begin{equation}\label{ineq66}
\begin{aligned}
\limsup_{\veps \to 0}\frac{C_\veps}{b_\veps}
&\leq
\lim_{K \to \infty}\limsup_{\veps \to 0}
\int_{|z|\leq K}e^{-\frac{2}{\veps}V(b_\veps z)}\ud z
+\lim_{K \to \infty}\limsup_{\veps \to 0}
\int_{|z|>K}e^{-\frac{2}{\veps}V(b_\veps z)}\ud z.
\end{aligned}
\end{equation}
By  Lemma~\ref{claim1in}, the dominated convergence theorem,
and the monotone convergence theorem we obtain that
\begin{equation}\label{limite20}
\lim_{K \to \infty}\limsup_{\veps\to 0}\int_{|z|\leq K}e^{-\frac{2}{\veps}V(b_\veps z)}\ud z=\lim_{K \to \infty}\int_{|z|\leq K}e^{-2V_0(z)}\ud z = C.
\end{equation}
Similarly, for $\widetilde{V}$ we obtain
\begin{equation}\label{limite20one}
\lim_{K \to \infty}\limsup_{\veps\to 0}\int_{|z|\leq K}e^{-\frac{2}{\veps}\widetilde{V}(b_\veps z)}\ud z = C.
\end{equation}
It remains to show that
\begin{equation}\label{limite201}
\lim_{K \to \infty}\limsup_{\veps\to 0}\int_{|z|> K}e^{-\frac{2}{\veps}V(b_\veps z)}\ud z=0
\end{equation}
and 
\begin{equation}\label{limite201one}
\lim_{K \to \infty}\limsup_{\veps\to 0}\int_{|z|> K}e^{-\frac{2}{\veps}\widetilde{V}(b_\veps z)}\ud z=0.
\end{equation}
In the sequel, we give the proof of~\eqref{limite201}, which
we divide in two cases, depending on whether $\beta \geq \alpha$ or
$\beta<\alpha$, where $\alpha$ and $\beta$ are defined in Hypothesis~\ref{hyp2} and Hypothesis~\ref{hyp3}, respectively.
Note first that by Hypothesis~\ref{hyp2}
for any $\delta>0$ there is $c_0(\delta)>0$
such that for any $z$ with $\abs{z}\leq \delta$
\begin{equation}\label{inequality1}
\abs{V^{\prime}(z)}\geq c_0(\delta)\abs{z}^{1+\alpha}.
\end{equation} 

Assume that $\beta\geq \alpha$.
By Hypothesis~\ref{hyp1} and Hypothesis~\ref{hyp3} there is an $R>1$ and $c>0$ such that for any $z$ with $\abs{z}\geq R_0$,
\begin{equation}\label{inequality2}
  \abs{V^{\prime}(z)}\geq c\abs{z}^{1+\beta}\geq c\abs{z}^{1+\alpha}.
\end{equation}
By~\eqref{inequality1} and~\eqref{inequality2}
there is a $c_1(\delta)>0$
such that for any  $z\in \bb{R}$,
$\abs{V^{\prime}(z)}\geq c_1(\delta)\abs{z}^{1+\alpha}$.
Since $V^\prime(z) = V^\prime(\abs{z})\sgn(z)$ and $V(0)=0$,
if we compute the integral from $0$ to $z$
of both sides of~\eqref{inequality2}
we obtain that there is a $c(\delta)>0$ such that
$V(z)\geq c(\delta)\abs{z}^{2+\alpha}$ for any  $z\in \bb{R}$. 
The preceding inequality implies
that $-V(b_\veps z) \leq - c(\delta)b^{2+\alpha}_\veps |z|^{2+\alpha}$
which together with
$b^{2+\alpha}_\veps=\veps$ yields that
\begin{equation}\label{inebuena}
\lim_{K \to \infty}\limsup\limits_{\veps\to 0} \int_{|z|>K}e^{-\frac{2}{\veps}V(b_\veps z)}\ud z\leq 
\lim_{K \to \infty}\limsup\limits_{\veps\to 0}
\int_{|z|>K}e^{-2c(\delta)|z|^{2+\alpha}}\ud z = 0.
\end{equation}
This completes the case $\beta \geq \alpha$.

Now, we assume that $-1<\beta<\alpha$. 
Since $V$ satisfies Hypothesis~\ref{hyp3}, we may now  take  $R_0 \geq [(2+\alpha)(2 +\beta)^{-1}]^{\frac{1}{\alpha-\beta}}$
and let 
$\kappa_0(\delta):=\min\{V^{\prime}(z) |z|^{-1-\alpha}\colon \delta\leq z\leq R_0 \}$, where $\delta>0$ is given in~\eqref{inequality1}.
Note that $\kappa_0(\delta)>0$ and that 
$k_1(\delta) := \min \{c_0(\delta), \kappa_0(\delta)\} > 0$ is such that
$  V^{\prime}(z) \geq k_1(\delta) z^{1+\alpha}$ for $z\in [0,R_0]$.
Now, by Hypothesis~\ref{hyp3} there is $c>0$ such that $V^{\prime}(z)\geq c z^{1+\beta}$
for any $z\geq R_0$ and therefore, for $\widetilde{c}  = \widetilde{c}(\delta):= \min\{k_1(\delta), c\}>0$
\begin{equation}
V^{\prime}(z)\geq
\begin{cases}
  \widetilde{c}\, z^{1+\alpha}&\quad \textrm{ for }\quad z\in [0,R_0],\\
  \widetilde{c}\, z^{1+\beta}&\quad \textrm{ for }\quad z > R_0.
\end{cases}
\end{equation}
As $V(0) = 0$, integrating from $0$ to $z$
in the both sides of the above inequality we obtain
\begin{equation}
V(z)\geq
\begin{cases} 
  \widetilde{c}\, z^{2+\alpha}(2+\alpha)^{-1} & \quad \textrm{ for }\quad z\in[0,R_0],\\
  \frac{\widetilde{c}\, z^{2+\beta}}{(2 + \beta)} + \widetilde{c}\,\Big[\frac{ R_0^{2+\alpha}}{2+\alpha}
    - \frac{R_0^{2+\beta}}{(2 + \beta)}\Big]& \quad\text{ for }\quad z> R_0.
\end{cases}
\end{equation}
Since $R_0\geq[(2+\alpha)(2 +\beta)^{-1}]^{\frac{1}{\alpha-\beta}}$, it follows that
$\frac{ R_0^{2+\alpha}}{2+\alpha}-\frac{ R_0^{2+\beta}}{2+\beta}\geq 0$,
and because $V$ is an even function we deduce
the existence of $\kappa = \kappa(\delta)>0$ such that
\begin{equation}\label{ine:buena1}
V(z)\geq 
\begin{cases}
\kappa |z|^{2+\alpha} &\quad \textrm{ for }\quad |z|\leq R_0,\\
\kappa |z|^{2+\beta} &\quad \textrm{ for }\quad |z|\geq R_0.
\end{cases}
\end{equation}
Since $b_\veps^{2 + \alpha} = \veps $
and $b_\veps^{2 + \beta}/\veps  = b_\veps^{\beta - \alpha} = b_\veps^{- \abs{\beta - \alpha}} \to \infty$
as $\veps \to 0$, the dominated convergence theorem yields
\begin{equation}\label{intV}
\begin{split}
\lim_{K \to \infty}\limsup_{\veps\to 0}\int_{{|z|>K}}&e^{-\frac{2}{\veps}V(b_\veps z)}\ud z\\
&=
\lim_{K \to \infty}\limsup_{\veps\to 0}\int_{{|z|>K}}e^{-\frac{2}{\veps}V(b_\veps z)}\prt{\Ind{\abs{b_\veps z}\leq R_0} + \Ind{\abs{b_\veps z}>R_0}}\ud z \\
&\leq
\lim_{K \to \infty} \limsup_{\veps\to 0}\Big(\;\int_{|z|> K}e^{-2\kappa|z|^{2+\alpha}}\ud z
+\int_{|z|> K}e^{-2\kappa b^{\beta-\alpha}_\veps |z|^{2+\beta}}\ud z\; \Big)\\
&=
\lim_{K \to \infty}\int_{|z|> K}e^{-2\kappa|z|^{2+\alpha}}\ud z = 0.
\end{split}
\end{equation}
This completes the case $-1<\beta < \alpha$.

Combining~\eqref{inebuena} and~\eqref{intV} we obtain~\eqref{limite201}. This finishes the proof of~\eqref{eq:14}.

In the sequel, we stress that~\eqref{eq:1456} is just a consequence from above  case $\beta\geq \alpha$.
Indeed, by~\eqref{desi1one} and~\eqref{limite20one}, it is enough to 
show~\eqref{limite201one}.
Since $\widetilde{V}$ satisfies Hypothesis~\ref{hyp3} with $\beta=\alpha$, the proof is already covered in~\eqref{inebuena}.

The proof of Lemma~\ref{lem:constantes} is finished.
\end{proof}
\begin{proof}[Proof of Lemma \ref{lem:couplingINV}]
    By~\eqref{densityY},~\eqref{eq:cambio} and Scheff\'e's lemma~(\cite[Lemma~3.3.2, p.95]{Reiss}), 
to obtain  \eqref{eq:limitesm},
it suffices to note that Lemma~\ref{claim1in} and Lemma~\ref{lem:constantes} imply that
\begin{equation}\label{densityc}
  \lim\limits_{\veps\to 0}\frac{b_{\veps}}{C_{\veps}}
  e^{-2\frac{V(b_{\veps}z)}{\veps}}=\frac{1}{C}{e^{-2 V_0(z)}}\quad
  \textrm{ for any }\quad z\in \mathbb{R}
\end{equation}
and
\begin{equation}\label{densityc23}
  \lim\limits_{\veps\to 0}\frac{b_{\veps}}{\widetilde{C}_{\veps}}
  e^{-2\frac{\widetilde{V}(b_{\veps}z)}{\veps}}=\frac{1}{C}{e^{-2 V_0(z)}}\quad
  \textrm{ for any }\quad z\in \mathbb{R}.
\end{equation}
\end{proof}
\section{\textbf{Complements}}\label{sec:complements}
In this section we include, for completeness of the exposition, a few
results that have been used throughout the text with a brief explanation.

\begin{proof}[Proof of Lemma~\ref{medidainvariante}]
We apply Theorem~3.3.4 of~\cite[Ch.~3, p.91]{KU}.
By Hypothesis~\ref{hyp3} for all $|z|\geq R$ we have 
$-V^{\prime}(z)\frac{z}{|z|^{1+\kappa}}\leq  -c|z|^{\rho-\kappa}$ for any $\kappa\in (0, \rho)$,
and therefore
\[\lim\limits_{|z|\to
  \infty}\left(-V^{\prime}(z)z|z|^{-1-\kappa}\right)=-\infty<0.
  \]
  Hence, the field $-V^{\prime}$ satisfies the drift condition
  eq.~(3.3.4) in~\cite[Ch.~3, p.86]{KU}. By Theorem~3.3.4 of~\cite{KU}
  we have the existence and uniqueness of the invariant measure
  $\mu^\veps$. In addition, for any $c>0$ there are
  $C_1=C_1(c,\kappa,\veps)>0$ and $C_2=C_2(c,\kappa,\veps)>0$ such
  that
\[
\dtv(X^{\veps}_t(x),X^{\veps}_t(y))\leq  C_1 e^{-C_2 t}(e^{c|x|}
+e^{c|y|})\quad \textrm{for any }\quad x,y \in \mathbb{R},~ t\geq 0.
\]   
By
Hypothesis~\ref{hyp3}, we have $\int_{\mathbb{R}} e^{c |z|}\mu^{\veps}(\ud z)<\infty$. 
Therefore, Theorem~3.3.4 in~\cite{KU}
yields~\eqref{zero}. Moreover, formula~\eqref{formula} follows
from Proposition~4.2 in~\cite[p.110]{pavliotis}.
\end{proof}
Let $\mathcal{C}^2$ represent the set of twice continuously
differentiable functions $f: \bb{R} \to \bb{R}$.
\begin{lemma}[Existence of a regular potential]\label{lem:Vtilde}
Assume that $V$ satisfies Hypothesis~\ref{hyp1} and
Hypothesis~\ref{hyp2} with $\alpha>0$.
For each  $M>0$, there exist an even $\mathcal{C}^2$ convex function
$V_M = V_{M,\alpha}:\mathbb{R} \to [0,\infty)$ and positive constants $c=c_{M,\alpha}$,
$C=C_{M,\alpha}$ and $R=R_{M,\alpha}$ such that
\begin{equation}\label{eq:moli1}
V_{M}(z)=V(z)\quad \textrm{ for }\quad  |z|\leq M
\end{equation}
and 
\begin{equation}\label{eq:moli2}
V^\prime_{M}(z)\geq c z^{1+\alpha}
\qquad \textrm{ and }\qquad
|V^\prime_{M}(z)|\leq C e^{z^2}\quad \textrm{ for all  }\quad z\geq R.
\end{equation}
In particular, the potential $V_{M}$ satisfies
Hypothesis~\ref{hyp1}, Hypothesis~\ref{hyp2} and
Hypothesis~\ref{hyp4}.
\end{lemma}
\begin{proof}
The proof follows by a standard mollifier procedure.
We mimic the lines given in Proposition~4.10 of~\cite{BJ}.
Let $g:\mathbb{R}\to [0,1]$ be an increasing 
$\mathcal{C}^\infty$-function such that $g(u)=0$ for $u\leq 1/2$,
$g(u)=1$ for $u\geq 1$, and $g(u)\in (0,1)$ for all $u\in (1/2,1)$. 
Let $M>0$ be fixed and 
define 
\begin{equation}\label{G_M}
G_M(u)=\left(1-g\left(\frac{u^2}{2M^2}\right)\right)V^{\prime\prime}(u)+g\left(\frac{u^2}{2M^2}\right)
|u|^{\alpha}\quad \textrm{ for all }\quad u \in \mathbb{R}.
\end{equation}
Observe that $G_M(u)=V^{\prime\prime}(u)$ for all $|u|\leq M$, and
$G_M(u)=|u|^{\alpha}$ for all $|u|\geq \sqrt{2}M$.  We note that $G_M$
is a non-negative continuous function and then we set
$H_M(u):=\int_{0}^{u}G_M(y)\ud y$ for all $u\in \mathbb{R}$.  Finally
we define $V_M(z):=\int_{0}^{\abs{z}}H_M(u)\ud u$ for all $z $. Since
$G_M$ is an even function, it follows that $H_M$ is odd and $V_M$ is again
even. Now, since $V_M(0) = V(0)=0$, $V^{\prime}_M(0) = V^{\prime}(0) = 0$ and
$V^{\prime\prime}_M(z) = V^{\prime\prime}(z)$ for $z \leq M$ it follows that  $V_M$ 
satisfies~\eqref{eq:moli1}. Moreover, since there is $C>0$ for which
$\abs{u}^\alpha \leq C \exp(u^2)$ for all $u \in \bb{R}$ it follows
that $V_M$ satisfies~\eqref{eq:moli2}.
\end{proof}
 
\begin{proposition}[Disintegration inequality]\label{L:disintegration}
  Suppose that $ \{X(x) = (X_t(x), t \geq 0), x \in S\}$ and
  $Y = \{Y(y) = (Y_t(y), t \geq 0), y \in S\}$ are Markov families on
  the measurable space $(S,\mathcal{S})$ and defined on the same
  probability space $(\Omega, \mc{F}, \bb{P})$.  Then, for all
  $r,s>0$, $a,b \in S$, the following disintegration inequality for
  the total variation distance holds:
  \begin{equation}
    \label{eq:15}
    \dtv\big(X_{r + s}(a),Y_{r+s}(b)\big) \leq \int_{S^2}
    \dtv\big(X_s(x),Y_s(y)\big)\, \bb{P}(X_r(a) \in \ud x, Y_r(b) \in
    \ud y).
  \end{equation}
\end{proposition}
\begin{proof}
Write $t = r + s$. Since the families $\{X(x), x \in S\}$ and $\{Y(y), y \in
  S\}$ are Markovian and are defined on the same probability space,
  for any $a \in S$ and $B \in \mc{S}$ we have that
  \begin{equation}
    \label{Markov-X-disintegration}
    \begin{aligned}
      \bb{P}(X_t(a) \in B) &= \int_S \bb{P}(X_s(x) \in B)
      \bb{P}(X_r(a) \in \ud x)\\
      &= \int_{S^2} \bb{P}(X_s(x) \in B)
      \bb{P}(X_r(a) \in \ud x, Y_r(a) \in \ud y).
    \end{aligned}
  \end{equation}
  Similarly, we have that
  \begin{equation}
    \label{Markov-Y-disintegration}
    \bb{P}(Y_t(a) \in B) = \int_{S^2} \bb{P}(Y_s(y) \in B)
    \bb{P}(X_r(a) \in \ud x, Y_r(a) \in \ud y).
  \end{equation}
  Therefore, from the definition of total variation distance, together
  with~\eqref{Markov-X-disintegration} and~\eqref{Markov-Y-disintegration}
  we obtain that
  \begin{equation}  
  \begin{split}
    \dtv\big(X_t(a), Y_t(b)\big)
    &= \sup_{B \in \mc{S}} \abs{\bb{P}\big(X_t(a) \in B\big) - \bb{P}\big(Y_t(b\big)
      \in B)}\\ 
    &\hspace{-1.5cm}= \sup_{B \in \mc{S}} \abs{\int_{S^2}\Big(\bb{P}\big(X_s(x) \in B\big) -\bb{P} \big(Y_s(y)
      \in B\big)\Big)\, \bb{P}\big(X_r(a) \in \ud x , Y_r(b) \in \ud y\big)}\\
    &\hspace{-1.5cm}\leq \int_{S^2} \dtv\big(X_s(x),Y_s(y)\big)\bb{P}
      \big(X_r(a) \in \ud x , Y_r(b) \in \ud y\big). 
  \end{split}
  \end{equation}
\end{proof}
\begin{proposition}[Support theorem for diffusions]\label{Supporttheorem}
For any $x \in \bb{R}$ and $\veps\in [0,1]$ let $Y^\veps(x) =
(Y^\veps_t(x), t \geq 0)$ be the solution of~\eqref{gensde}. For
each fixed $t>0$, the law of $Y^\veps_t(x)$ is absolutely continuous with respect to the Lebesgue measure and it has full support on
$\bb{R}$. 
\end{proposition}
\begin{proof}
  Fix $\veps \in [0,1]$. Now, write for simplicity
  $Y_t = Y^\veps_t(x)$, $F =F_\veps $ and note that, almost surely,
  for every $t\geq 0$
  \begin{equation}\label{Yt}
    Y_t = x + \int_0^t F (Y_s)\, \ud s + B_t.
  \end{equation}
  The proof is done in two steps. On the first step, following the
  ideas in~\cite{Fournier}, we prove that for any $t>0$ the law of
  $Y_t$ denoted by $\mu_t$ is absolutely continuous with respect to
  the Lebesgue measure on $\bb{R}$. Let $\rho_t$ represent a
  density of $\mu_t$, i.e. for any $a,b \in \bb{R}$ with $a<b$
  \begin{equation}
    \label{density}
    \bb{P}(Y_t \in [a,b]) = \mu_t([a,b]) = \int_a^b \rho_t(z)\, \ud z.
  \end{equation}
  On the second step, we prove, with the help of the maximum principle
  in~\cite{Nirenberg}, that $\rho_{t + s}(z) >0$ for all $z \in \bb{R}$.
  Since $t>0$ and $s>0$ are arbitrary, this completes the proof that $\mu_t(\ud
  z) = \rho_t(z) \ud z$ with   $\rho_t(z)>0$ for all $t>0$, i.e. the law
  of $Y_t$ has full support. 

  We remark that a standard localization argument is not
  straightforward with the methods in~\cite{Fournier}. Indeed, as the
  authors themselves say
  \begin{quote}
    ``Our result might be deduced from [Aronson-1968] by a localization
    argument, however, we did not succeed in this direction.''
  \end{quote}
  \noindent
  \textbf{Step 1.}
  We adapt to our case the proof of Theorem~2.1 in~\cite{Fournier}. 
  This means that $\ud Y_t = b(Y_t) \ud t + \sigma(Y_t) \ud B_t$ with
  $b(z) = F(z)$ and $\sigma(z) = 1$ for all $z \in \bb{R}$. Since
  $b$ is not bounded by a linear function we cannot apply Theorem~2.1 
  in~\cite{Fournier} directly. However, the field $F$ is convex and
  drives the trajectories towards the origin which allows us to
  obtain $L^2$ bounds and replicate the main steps in the proof.
  Moreover, the noise term is simpler
  and this allows us to ignore the auxiliary function $f_\delta$
  defined in Lemma~1.2 in~\cite{Fournier}. 
  
  Now, for $\delta \in (0,t)$, consider the random variable
$Z_\delta: = Y_{t - \delta} + B_{t} - B_{t - \delta}$.
  Note that for any $b\in \bb{R}$
  \begin{equation}
    \label{Z-transform}
    \abs{\bb{E}[\exp(\ii b Z_\delta) \vert \mc{F}_{t- \delta}]} =
    \abs{\exp(\ii bY_{t-\delta} - \delta b^2/2) }  = \exp(-\delta b^2/2),
  \end{equation}
where $(\mc{F}_{t},t\geq 0)$ is the natural filtration of the Brownian motion $B=(B_t,t\geq 0)$ and $\ii$ is the unit imaginary.
  By~\eqref{Yt} it follows that 
  \begin{equation}
    \label{Y-Z-dif}
    Y_t - Z_\delta = \int_{t-\delta}^t F (Y_s)\, \ud s.
  \end{equation}
  By~\eqref{F-Y-eps-8-bound}, there is $C = C_t$ such that
  $\sup_{s \in [0,t]}\bb{E}[|F(Y_s)|^2] \leq C^2$ and therefore by
  Jensen's and  Cauchy--Schwarz's inequalities
  \begin{equation}
    \label{abs-Y-Z}
    \begin{aligned}
      \big(\bb{E}[\abs{Y_t - Z_\delta}]\big)^2 &\leq \bb{E}[\abs{Y_t -
        Z_\delta}^2]\\
        &
      = \bb{E} \bigg[\Big(\int_{0}^\delta F(Y_{t-\delta +
        s})\,  \ud s\Big)^2\bigg]\\
        & \leq \delta \int_0^\delta \bb{E}[\abs{F(Y_{t-\delta +
          s})}^2] \leq C^2 \delta^2.
    \end{aligned}
  \end{equation}
  Let $\mu_t$ be the law of $Y_t$ and let $\widehat{\mu}_t$ be the
  characteristic function of $\mu_t$ defined by
  $\widehat{\mu}_t(b) : = \bb{E}[\exp(\ii bY_t)]$. We note that for any
  $\delta \in (0,t)$ and $b \in \bb{R}$, by~\eqref{Z-transform} 
  and~\eqref{abs-Y-Z}, we have
  \begin{equation}
    \label{mu-hat-t}
    \begin{split}
      \abs{\widehat{\mu}_t(b)} &= \abs{\bb{E}[\exp(\ii b Y_t)]}
      \leq \abs{\bb{E}[\exp(\ii b Z_\delta)]}  +\abs{b} \bb{E}[\abs{Y_t
        - Z_\delta}]\\
        &\leq  \exp(-\delta b^2/2) + C \abs{b}  \delta.
    \end{split}
  \end{equation}
  Let $R_t>0$ be such that $ (\log\abs{b})^2/b^2 < t$ when
  $\abs{b} > R_t$.
  For each $b$ with $\abs{b}\geq R_t$ we choose
  $\delta_b: = (\log\abs{b})^2/b^2$ and so the bound 
  in~\eqref{mu-hat-t} implies that
  \begin{equation}
    \label{Yt-Zdelta_b}
    \begin{aligned}
      \abs{\widehat{\mu}_t(b)} \leq \exp(-\delta_b b^2/2) + C \abs{b}
      \delta_b &= \exp(-(\log \abs{b})^2/2) + C (\log\abs{b})^2/b.
    \end{aligned}
  \end{equation}
   Since $\abs{\widehat{\mu}_t(b)} \leq 1$ for all $b \in \bb{R}$ it
   follows that $\int_{-\infty}^\infty\abs{\widehat{\mu}_t(b)}^2 \ud b< \infty$ and so, by
   Lemma~1.1 in~\cite{Fournier} it follows that $\mu_t$
   has density in $\bb{R}$.

  \noindent
  \textbf{Step 2.} Note that $\rho_{t+s}$ is the solution of
  \begin{equation}\label{pde}
    L u + F^{\prime} u =0
  \end{equation}
  at time $s$ with initial condition $\rho_t$ where
  $Lu : = \partial_t u -\frac{1}{2}(\partial_x)^2 u + F \partial_x u $.
  By Theorem~3 in~\cite{Nirenberg}, $(\rho_{t + h}, h \geq 0)$ is a
  non-negative solution of~\eqref{pde} and therefore either
  $\rho_{t + s}(z)>0$ for all $z\in \bb{R}$ or $\rho_{t + s}(z) = 0$ for all
  $z \in \bb{R}$, and since $\int_\bb{R} \rho_{t + s}(z) \ud z = 1$ it follows that $\rho_{t+s}(z)>0$ for all $z\in \bb{R}$. This concludes the proof.  
\end{proof}
We conclude this section with the following  result.
\begin{lemma}[Monotonicity and continuity]\label{lem:contdecre}
  For all $x \in \bb{R}$ the function $t \mapsto G_t(x) $ defined
  in~\eqref{Gdef} is continuous and strictly decreasing in $t$.
\end{lemma}
\begin{proof}
  Let $x \in \bb{R}$ be fixed. By the triangle inequality  for all $t>0$ and $s>0$ we have 
  \begin{equation}
    \label{GboundY}
    \abs{G_x(t) - G_x(s) } \leq \dtv  (Y_t(\sgn(x)\infty), Y_s(\sgn(x)\infty)).
  \end{equation}
Then it is enough to show that the right-hand side of the preceding inequality tends to zero as $t\to s$.
For short, we write $Y_u = Y_u(\sgn(x)\infty)$, $u\geq 0$.
By Proposition~\ref{Supporttheorem} it follows that for every $t>0$, the law of $Y_t$ is absolutely continuous with respect to the Lebesgue measure and has a full support density $\rho_t(y)$. 
Moreover, $(\rho_t(y))_{t\geq 0}$ solves the so-called Fokker--Planck equation 
 \begin{equation}
    \label{Fokker-PlankUNO}
    \begin{aligned}
      \partial_t \rho_t(y)= \frac{1}{2}\partial^2_y \rho_t -
      \partial_y (F_0(y) \rho_t(y)),
    \end{aligned}
  \end{equation}
where $F_0$ is as defined in \eqref{fieldnota}, see  for instance~\cite[Section~2.2]{SI}.
Then $\lim_{t \to s} \rho_t(y)
  = \rho_s(y)$ for all $y \in \mathbb{R}$ and therefore by Scheff\'e's lemma, see~\cite[Lemma~3.3.2, p.95]{Reiss}, we have $\lim_{t \to s} \dtv(Y_t,Y_s) = 0$.
This completes the proof that $t\mapsto G_t(x)$ is continuous. 
  
We now turn to the proof that $G_t(x)$ is strictly decreasing in
$t$. Recall that $G_x(t)=\dtv\prt{Y_t(\sgn(x)\infty),\nu}$ for
$t\geq 0$.  Let $x\in \mathbb{R}$ and $t>0$ be
fixed. By~\eqref{eq:menorqueuno} we have $G_x(t)<1$. For short let
$\theta_{x,t}:=G_x(t)$ and denote the law of $Y_t(\sgn(x)\infty)$ by
$\mu_{x,t}$.  Let $(P_s)_{s\geq 0}$ be the semigroup associated to the
Markov process $(Y_s(z),s\geq 0,z\in \mathbb{R})$ and note the
invariance $P_s(\nu)=\nu$, $s\geq 0$.  Since $\theta_{x,t}$ is the
total variation distance between $Y_t(\sgn(x)\infty)$ and $\nu$, there
exists a coupling between $\mu_{x,t}$ and $\nu$ such that
$\mu_{x,t}=(1-\theta_{x,t})\nu+\theta_{x,t}\eta_{x,t}$, where
$\eta_{x,t}$ is a probability measure on $\mathbb{R}$.  By the
semigroup property we have for any $s>0$
\[
\begin{split}
\dtv(\mu_{x,t+s},\nu)&=\dtv(P_s(\mu_{x,t}),\nu)
=\dtv((1-\theta_{x,t})P_s(\nu)+\theta_{x,t}P_s(\eta_{x,t}),\nu)\\
&=\dtv((1-\theta_{x,t})\nu+\theta_{x,t}P_s(\eta_{x,t}),\nu)=\theta_{x,t}\dtv(P_s(\eta_{x,t}),\nu).
\end{split}
\]
Now, we claim that $\dtv(P_s(\eta_{x,t}),\nu)<1$. Indeed, by disintegration we have 
\begin{equation}
\begin{split}
\dtv(P_s(\eta_{x,t}),\nu)\leq \int_{\mathbb{R}}
\dtv(P_s(z),\nu) \eta_{x,t}(\ud z)=
\int_{\mathbb{R}}
G_{z}(s) \eta_{x,t}(\ud z).
\end{split}
\end{equation}
Hence, $\dtv(P_s(\eta_{x,t}),\nu)=1$ if and only if $G_{z}(s)=1$ for $z$-almost surely with respect to the measure $\eta_{x,t}$. This yields a contradiction with~\eqref{eq:menorqueuno} and hence the proof that the function $G_x$ is strictly decreasing is finished.
\end{proof}
\end{appendix}

\section*{\textbf{Statements and declarations}}
\subsection*{\textbf{Acknowledgments}}
G.~Barrera is greatly indebted to S.~Olla (Universit\'e
Paris-Dauphine, CNRS CEREMADE) for bringing out to our attention the
problem in the degenerate setting.  Part of this work was done during
G.~Barrera stay at the Institut Henri Poincar\'e (Centre Emile Borel)
during the trimester ``Stochastic Dynamics Out of Equilibrium
2017''. He thanks this institution for hospitality and support.  He
also is in debt with FORDECyT-CONACyT-M\'exico for all the support to
attend to the trimester Stochastic Dynamics Out of Equilibrium.
G.~Barrera is indebted to professors 
G.~Giacomin (Universit\'e de
Paris UFR de Math\'ematiques and LPSM) and  J.~Beltr\'an (Pontificia
Universidad Cat\'olica del Per\'u, PUCP) for rich talks at the
beginning of this project and with professor J.~Lukkarinen (University of Helsinki) for rich conversations along the project.
We would like to thank the anonymous referee for their insightful comments and constructive suggestions, which have significantly improved this paper.

\subsection*{\textbf{Funding}}
The research of G.~Barrera has been supported by the Academy of Finland, via the Matter and Materials Profi4 University Profiling Action, an Academy project (project No.~339228) and the Academy of Finland via Finnish Centre of Excellence in Randomness and STructures (projects 
No.~346306 and No.~346308).

\noindent
C. da Costa was supported by the Engineering and Physical Sciences Research Council [EP/W00657X/1].

\subsection*{\textbf{Competing interests}} The authors declare that they have no conflict of interest.

\subsection*{\textbf{Authors' contributions}}
All authors have contributed equally to the paper.

\subsection*{\textbf{Availability of data and material}}
Data sharing not applicable to this article as no datasets were generated or analyzed during the current study.

\subsection*{\textbf{Ethical approval}} Not applicable.

\begin{thebibliography}{99}
\bibitem{ALD} 
Aldous, D.:
\textit{Random walks on finite groups and rapidly mixing Markov chains}.
Seminar on Probability XVII, Lecture Notes in Mathematics \textbf{986}, Springer, Berlin, (1983), 243--297.
\url{https://doi.org/10.1007/BFb0068322}
\MR{0770418}

\bibitem{AD} Aldous, D. \& Diaconis, P.: 
Shuffling cards and stopping times. 
\textit{Amer. Math. Monthly} \textbf{93}, no. 5, (1986), 333--348.
\url{https://doi.org/10.2307/2323590}
\MR{0841111}

\bibitem{Anderson}
Anderson, R., Duanmu, H. \& Smith, A.:
Mixing and average mixing times for general Markov processes.
\textit{Canad. Math. Bull.} \textbf{64}, no. 3,
(2021), 541--552.
\url{https://doi.org/10.4153/S0008439520000636}
\MR{4313548}
 
\bibitem{Arapostathis}
Arapostathis, A., Biswas, A. \& Borkar, V.:
Controlled equilibrium selection in stochastically perturbed dynamics. 
\textit{Ann. Probab.} \textbf{46}, no. 5, (2018), 2749--2799.
\url{https://doi.org/10.1214/17-AOP1238}
\MR{3846838}

\bibitem{Avena}
Avena, L., G\"ulda\c{s}, H., van der Hofstad, R.,
den Hollander, F. \& Nagy, O.:
Linking the mixing times of random walks on static and dynamic random graphs.
\textit{Stochastic Process. Appl.} \textbf{153}, (2022), 145--182.
\url{https://doi.org/10.1016/j.spa.2022.07.009}
\MR{4474678}

\bibitem{BAKHTIN} Bakhtin, Y.: 
Small noise limit for diffusions near heteroclinic networks.
\textit{Dyn. Syst.} \textbf{25}, no. 3, (2010), 413--431.
\url{https://doi.org/10.1080/14689367.2010.482520}
\MR{2731621} 

\bibitem{BA} Barrera, G.:
Abrupt convergence for a family of Ornstein--Uhlenbeck processes.
\textit{Braz. J. Probab. Stat.} \textbf{32}, no. 1,  (2018), 188--199.
\url{https://doi.org/10.1214/16-BJPS337}
\MR{3770869}

\bibitem{BAmax} Barrera, G.:
Cutoff phenomenon for the maximum of a sampling of Ornstein-Uhlenbeck processes.
\textit{Statist. Probab. Lett.} \textbf{168}, no. 108954, (2021), 7 pp.
\url{https://doi.org/10.1016/j.spl.2020.108954}
\MR{4160702}

\bibitem{BAinv} Barrera, G.:
Limit behavior of the invariant measure for Langevin dynamics.
\textit{Probab. Math. Statist.} \textbf{42}, no. 1, (2022), 143--162.
\url{https://doi.org/10.37190/0208-4147.00020}
\MR{4490674}

\bibitem{BPH2} Barrera, G., H\"ogele, M. \& Pardo, J.:
Cutoff thermalization for Ornstein--Uhlenbeck systems with small L\'evy noise in the Wasserstein distance.
\textit{J. Stat. Phys.} \textbf{184}, no. 27, (2021), 54 pp.
\url{https://doi.org/10.1007/s10955-021-02815-0}
\MR{4307706}

\bibitem{BHP} Barrera, G., H\"ogele, M. \& Pardo, J.:
The cutoff phenomenon in total variation for nonlinear Langevin systems with small layered stable noise.
\textit{Electron. J. Probab.} \textbf{26}, no. 119, (2021), 76 pp.
\url{https://doi.org/10.1214/21-EJP685}
\MR{4315514}

\bibitem{BPH3} Barrera, G., H\"ogele, M. \& Pardo, J.:
The cutoff phenomenon in Wasserstein distance for nonlinear stable Langevin systems with small L\'evy noise. 
\textit{J. Dynam. Differential Equations} \textbf{36}, no. 1, (2024), 251--278.
\url{https://doi.org/10.1007/s10884-022-10138-1}
\MR{4710779}


\bibitem{BPH4} Barrera, G., H\"ogele, M. \& Pardo, J.:
The cutoff phenomenon for the stochastic heat and wave equation subject to small L\'evy noise. 
\textit{Stoch. Partial Differ. Equ. Anal. Comput.} \textbf{11}, no. 3, (2023), 1164--1202.
\url{https://doi.org/10.1007/s40072-022-00257-7}
\MR{4624136}

\bibitem{BJ} Barrera, G. \& Jara, M.: 
Abrupt convergence of stochastic small perturbations of one dimensional dynamical systems. 
\textit{J. Stat. Phys.} \textbf{163}, no. 1, (2016),    113--138.
\url{https://doi.org/10.1007/s10955-016-1468-1}
\MR{3472096}

\bibitem{BJ1} Barrera, G. \& Jara, M.:  
Thermalisation for small random perturbations of dynamical systems. 
\textit{Ann. Appl. Probab.} \textbf{30}, no. 3, (2020), 1164--1208.
\url{https://doi.org/10.1214/19-AAP1526}
\MR{4133371}

\bibitem{BLiu} Barrera, G. \& Liu, S.:  
A switch convergence for a small perturbation of a linear recurrence equation. 
\textit{Braz. J. Probab. Stat.} \textbf{35}, no. 2, (2021), 224--241.
\url{https://doi.org/10.1214/20-BJPS474}
\MR{4255156}

\bibitem{BP} Barrera, G. \& Pardo, J.:
Cut-off phenomenon for Ornstein--Uhlenbeck processes driven by L\'evy processes.
\textit{Electron. J. Probab.} \textbf{25}, no. 15, (2020), 33 pp.
\url{https://doi.org/10.1214/20-EJP417}
\MR{4073676}

\bibitem{BBF1} Barrera, J., Bertoncini, O. \& Fern\'andez, R.: 
Abrupt convergence and escape behavior for birth and death chains. 
\textit{J. Stat. Phys.} \textbf{137}, no. 4, (2009), 
595--623. 
\url{https://doi.org/10.1007/s10955-009-9861-7}
\MR{2565098}

\bibitem{BBF}
Barrera, J., Bertoncini, O. \& Fern\'andez, R.:  Cut-off and exit from metastability: two sides of
the same coin. 
\textit{C. R. Math. Acad. Sci. Paris} \textbf{346}, no. 11-12, (2008), 691--696. 
\url{https://doi.org/10.1016/j.crma.2008.04.007}
\MR{2423280}

\bibitem{JB} Barrera, J., Lachaud, B. \& Ycart, B.: Cut-off for $n$-tuples of exponentially converging processes. 
\textit{Stochastic Process. Appl.} \textbf{116}, no. 10, (2006), 1433--1446.
\url{https://doi.org/10.1016/j.spa.2006.03.003}
\MR{2260742}

\bibitem{BY} Barrera, J. \& Ycart, B.: 
Bounds for left and right window cutoffs.
\textit{ALEA Lat. Am. J. Probab. Math. Stat.} \textbf{11}, no. 2, (2014), 445--458.
\url{https://alea.impa.br/articles/v11/11-19.pdf}
\MR{3265085} 

\bibitem{BD} Bayer, D. \& Diaconis, P.:
Trailing the dovetail shuffle to its lair. 
\textit{Ann. Appl. Probab.} \textbf{2}, no. 2, (1992), 294--313.
\url{https://doi.org/10.1214/aoap/1177005705}
\MR{1161056}

\bibitem{Beltran} Beltr\'an, J. \& Landim, C.:
A martingale approach to metastability.
\textit{Probab. Theory Related Fields} \textbf{161}, no. 1-2, (2015), 267--307.
\url{https://doi.org/10.1007/s00440-014-0549-9}
\MR{3304753}

\bibitem{Berglund} Berglund, N. \& Gentz, B.:
\textit{Noise-induced phenomena in slow-fast dynamical systems. A sample-paths approach}. Probability and its Applications, Springer-Verlag, London, (2006).
\url{https://doi.org/10.1007/1-84628-186-5}
\MR{2197663} 

\bibitem{Bil99} Billingsley, P.: 
\textit{Convergence of Probability measures}. Second edition. Wiley Series in Probability and Statistics. John Wiley \& Sons, Inc., New York, (1999).
\url{https://doi.org/10.1002/9780470316962}
\MR{1700749}

\bibitem{Biswas} Biswas, A. \& Borkar, V.: 
Small noise asymptotics for invariant densities for a class of diffusions: a control theoretic view.
\textit{J. Math. Anal. Appl.} \textbf{360}, no. 2, (2009), 476--484.
\url{https://doi.org/10.1016/j.jmaa.2009.06.070}
\MR{2561245}

\bibitem{Bogachev} Bogachev, V., Krylov, N., 
R\"ockner, M. \& Shaposhnikov, S.:
\textit{Fokker--Planck--Kolmogorov equations}. 
Mathematical Surveys and Monographs \textbf{207}. American Mathematical Society, Providence, RI, (2015).
\url{https://doi.org/10.1090/surv/207}
\MR{3443169}
 
\bibitem{Bolley} Bolley, F., Gentil, I. \& Guillin, A.:
Convergence to equilibrium in Wasserstein distance for Fokker--Planck equations. 
\textit{J. Funct. Anal.} \textbf{263}, no. 8, (2012), 2430--2457.
\url{https://doi.org/10.1016/j.jfa.2012.07.007}
\MR{2964689} 

\bibitem{Bordenave2019}
Bordenave, C., Caputo, P. \&  Salez, J.:
Cutoff at the ``entropic time" for sparse Markov chains. 
\textit{Probab. Theory Relat. Fields} \textbf{173}, no. 1-2, (2019),  261--292.
\url{https://doi.org/10.1007/s00440-018-0834-0}
\MR{3916108}

\bibitem{Boursier}
Boursier, J., Chafa\"i, D. \& Labb\'e, C.:
Universal cutoff for Dyson Ornstein Uhlenbeck process.
\textit{Probab. Theory Relat. Fields} \textbf{185}, no.1-2, (2022).
\url{https://doi.org/10.1007/s00440-022-01158-5} 
\MR{4528974}

\bibitem{Bovier} Bovier, A. \& den Hollander, F.:
\textit{Metastability. A potential-theoretic approach}. 
MGrundlehren der mathematischen Wissenschaften \textbf{351}. Springer, Cham, (2015).
\url{https://doi.org/10.1007/978-3-319-24777-9}
\MR{3445787}

\bibitem{CapQua20}
Caputo, P. \& Quattropani, M.:
Mixing time trichotomy in regenerating dynamic digraphs.
\textit{Stochastic Process. Appl.} \textbf{137}, (2021), 222-251.
\url{https://doi.org/10.1016/j.spa.2021.03.003}
\MR{4244192}

\bibitem{Ceccherini}
Ceccherini-Silberstein, T., Scarabotti, F. \& Tolli, F.:
\textit{Harmonic analysis on finite groups.
Representation theory, Gelfand pairs and Markov chains}.
Cambridge Studies in Advanced Mathematics \textbf{108}. Cambridge University Press, Cambridge, (2008).
\url{https://doi.org/10.1017/CBO9780511619823}
\MR{2389056}

\bibitem{CSC} Chen, G. \& Saloff-Coste, L.:
The cutoff phenomenon for ergodic Markov processes.
\textit{Electron. J. Probab.} \textbf{13}, no. 3, (2008), 26--78.
\url{https://doi.org/10.1214/EJP.v13-474}
\MR{2375599} 

\bibitem{Chiclana} Chiclana, R. \& Peres, Y.:
No cutoff in spherically symmetric trees.
\textit{Electron. Commun. Probab.} \textbf{27}, no. 27, (2022), 11 pp.
\url{https://doi.org/10.1214/22-ECP468}
\MR{4424033} 

\bibitem{COFFEY}
Coffey, W. \& Kalmykov, Y.: 
\textit{The Langevin equation: with applications in physics, chemistry and electrical engineering}.
Third edition. World Scientific Series in Contemporary Chemical Physics \textbf{27}, (2012).
\url{https://doi.org/10.1142/8195}
\MR{3236656}

\bibitem{Dangeli}
D'Angeli, D. \& Donno, A.:
No cut-off phenomenon for the ``insect Markov chain"
\textit{Monatsh. Math.} \textbf{156}, no. 3, (2009), 201--210.
\url{https://doi.org/10.1007/s00605-008-0014-x}
\MR{2481088}. 

\bibitem{PDI} Diaconis, P.:
\textit{Group representations in probability and statistics}. 
Institute of Mathematical Statistics Lecture Notes
Monograph Series \textbf{11}. Institute of Mathematical Statistics, Hayward, CA, (1988).
\url{http://www.jstor.org/stable/4355560}
\MR{0964069}

\bibitem{DI} Diaconis, P.: 
The cutoff phenomenon in finite Markov chains.
\textit{Proc. Nat. Acad. Sci. U.S.A.} \textbf{93}, no. 4, (1996), 1659--1664.
\url{https://doi.org/10.1073/pnas.93.4.1659}
\MR{1374011} 

\bibitem{DSH} Diaconis, P. \& Shahshahani, M.:
Generating a random permutation with random transpositions. 
\textit{Z. Wahrsch. Verw. Gebiete} \textbf{57}, no. 2, (1981), 159--179.
\url{https://doi.org/10.1007/BF00535487} 
\MR{0626813} 

\bibitem{Ding} Ding, J., Lubetzky, E. \& Peres, Y.:
The mixing time evolution of Glauber dynamics for the mean-field Ising model.
\textit{Comm. Math. Phys.} \textbf{289}, no. 2, (2009), 725--764.
\url{https://doi.org/10.1007/s00220-009-0781-9} 
\MR{2506768}  
 
\bibitem{donpensonzha16} Dong, Z., Peng, X., Song, Y. \& Zhang, X.:
Strong Feller properties for degenerate SDEs with jumps. 
\textit{Ann. Inst. Henri Poincar\'e Probab. Stat.} \textbf{52}, no. 2, (2016), 888--897. 
\url{https://doi.org/10.1214/14-AIHP658}
\MR{3498014}

\bibitem{Eberle} Eberle, A., Guillin, A. \& Zimmer, R.:
Couplings and quantitative contraction rates for Langevin dynamics.
\textit{Ann. Probab.} \textbf{47}, no. 4, (2019), 1982--2010.
\url{https://doi.org/10.1214/18-AOP1299}
\MR{3980913}

\bibitem{Elboim} Elboim, D. \& Schmid, D.:
Mixing times and cutoff for the TASEP in the high and low density phase.
\textit{Probab. Math. Phys.} \textbf{5}, no. 2, (2024), 413--459.
\url{https://doi.org/10.2140/pmp.2024.5.413}
\MR{4749811}

\bibitem{EK} Ethier, S. \& Kurzt, T.:
\textit{Markov processes.
Characterization and convergence}.
John Wiley \& Sons, Inc., New York, (1986).
\url{https://doi.org/10.1002/9780470316658}
\MR{0838085}


\bibitem{Feller52}
Feller, W.:
The parabolic differential equations and the associated semi-groups of transformations.
\textit{Ann. of Math.} \textbf{55}, no. 2, (1952), 468--519.
\url{https://doi.org/10.2307/1969644}
\MR{0047886}

\bibitem{Feller54}
Feller, W.:
The general diffusion operator and positivity preserving semi-groups in one dimension.
\textit{Ann. of Math.} \textbf{60}, no, 2, (1954), 417--436.
\url{https://doi.org/10.2307/1969842}
\MR{0065809}

\bibitem{Fournier} Fournier, N. \& Printems, J.:
Absolute continuity for some one-dimensional processes.
\textit{Bernoulli} \textbf{16}, no. 2, (2010), 343--360.
\url{https://doi.org/10.3150/09-BEJ215}
\MR{2668905} 

\bibitem{FW} Freidlin, M. \& Wentzell, A.:
\textit{Random perturbations of dynamical systems}. Third edition. Grundlehren der mathematischen Wissenschaften \textbf{260}. Springer, Heidelberg, (2012).
\url{https://doi.org/10.1007/978-3-642-25847-3}
\MR{2953753}

\bibitem{Galves} Galves, A., Olivieri, E. \& Vares, M. E.:
Metastability for a class of dynamical systems subject to small random perturbations.
\textit{Ann. Probab.} \textbf{15}, no. 4, (1987), 1288--1305.
\url{https://doi.org/10.1214/aop/1176991977}
\MR{0905332} 

\bibitem{Gantert} Gantert, N., Nestoridi, E. \& Schmid, D.:
Cutoff on trees is rare.
\textit{J. Theoret. Probab.} \textbf{37}, no. 2, (2024), 1417--1444.
\url{https://doi.org/10.1007/s10959-023-01274-5}
\MR{4751297} 

\bibitem{Gantert1} Gantert, N., Nestoridi, E. \& Schmid, D.:
Mixing times for the simple exclusion process with open boundaries.
\textit{Ann. Appl. Probab.} \textbf{33}, no. 2, (2023), 1172--1212.
\url{https://doi.org/10.1214/22-AAP1839}
\MR{4564424} 


\bibitem{Giacomin} Giacomin, G. \& Merle, M.:
Weak noise and non-hyperbolic unstable fixed points: sharp estimates on transit and exit times. 
\textit{Bernoulli} \textbf{21}, no. 4, (2015), 2242--2288.
\url{https://doi.org/10.3150/14-BEJ643}
\MR{3378466}  

\bibitem{Goncalves}
Gon\c{c}alves, P., Jara, M., Menezes, O. \& Marinho, R.: 
Sharp convergence to equilibrium for the SSEP with reservoirs. 
To appear in  \textit{	
Ann. Inst. Henri Poincar\'e Probab. Stat.} (2024+).

\bibitem{Hermon} Hermon, J. \& Sousi, P.:
A comparison principle for random walk on dynamical percolation. 
\textit{Ann. Probab.} \textbf{48}, no. 6, (2020), 2952--2987.
\url{10.1214/22-AAP183910.1214/20-AOP1441}
\MR{4164458}  

\bibitem{Hunter} Hunter, J.:
Coupling and mixing times in a Markov chain.
\textit{Linear Algebra Appl.} \textbf{430}, no. 10,  (2009), 2607--2621.
\url{https://doi.org/10.1016/j.laa.2008.09.017}
\MR{2509844}

\bibitem{Iacobucci} Iacobucci, A., Olla, S. \&  Stoltz, G.:
Convergence rates for nonequilibrium Langevin dynamics.
\textit{Ann. Math. Qu\'ebec} \textbf{43}, no. 1, (2019), 73--98.
\url{https://doi.org/10.1007/s40316-017-0091-0}
\MR{3925138} 

\bibitem{ikedawatanabe} Ikeda, N. \& Watanabe, S.:
A comparison theorem for solutions of stochastic differential equations and its applications.
\textit{Osaka Math. J.} \textbf{14}, no. 3, (1977), 619--633.
\url{https://doi.org/10.18910/7664}
\MR{0471082}

\bibitem{Ji} Ji, M., Shen, Z. \& Yi, Y.:
Convergence to equilibrium in Fokker--Planck equations.
\textit{J. Dyn. Diff. Equat.} \textbf{31}, no. 3, (2019), 1591--1615.
\url{https://doi.org/10.1007/s10884-018-9705-8}
\MR{3992083} 

\bibitem{KabLipShi86} 
Kabanov, Y., Liptser, R. \& Shiryaev, A.:
On the variation distance for probability measures defined on a filtered space.
\textit{Probab. Theory Relat. Fields} \textbf{71}, no. 1, (1986), 19--35.  
\url{https://doi.org/10.1007/BF00366270}
\MR{0814659} 

\bibitem{Kallenberg} Kallenberg, O.:
\textit{Foundations of modern probability}.
Springer-Verlag, New York, (1997). 
\MR{1464694}

\bibitem{KarSch98} Karatzas, I. \& Shreve, S.: 
\textit{Brownian motion and stochastic calculus}. 
Second edition. Graduate Texts in Mathematics \textbf{113}. Springer-Verlag, New York, (1991).
\url{https://doi.org/10.1007/978-1-4684-0302-2}
\MR{1121940} 

\bibitem{kipnislandim} Kipnis, C. \& Landim, C.: 
\textit{Scaling limits of interacting particle systems}. 
Grundlehren der mathematischen Wissenschaften 
\textbf{320}. Springer-Verlag, Berlin, (1999). 
\url{https://doi.org/10.1007/978-3-662-03752-2}
\MR{1707314} 

\bibitem{Kuehn} Kuehn, C., Neam\c{t}u, A. \& Pein, A.:
Random attractors for stochastic partly dissipative systems.
\textit{Nonlinear Differ. Equ. Appl.} \textbf{27}, no. 35, (2020), 37 pp.
\url{https://doi.org/10.1007/s00030-020-00638-8}
\MR{4110683} 

\bibitem{KU} Kulik, A.: 
\textit{Ergodic behavior of Markov processes.
With applications to limit theorems}. De Gruyter Studies in Mathematics \textbf{67}. De Gruyter, Berlin, (2018). 
\url{https://doi.org/10.1515/9783110458930}
\MR{3791835} 

\bibitem{labbepetit}
Labb\'e C. \& Petit, E.: 
Hydrodynamic limit and cutoff for the biased adjacent walk on the simplex.
To appear in  \textit{	
Ann. Inst. Henri Poincar\'e Probab. Stat.} (2024+).

\bibitem{Lachaud}
Lachaud, B.: Cut-off and hitting times of a sample of Ornstein--Uhlenbeck processes and its
average. 
\textit{J. Appl. Probab.} \textbf{42}, no. 4, (2005), 1069--1080. 
\url{https://doi.org/10.1239/jap/1134587817}
\MR{2203823}

\bibitem{Lacoin} Lacoin, H.: 
The cutoff profile for the simple exclusion process on the circle. 
\textit{Ann. Probab.} \textbf{44}, no. 5, 
(2016), 3399--3430. 
\url{https://doi.org/10.1214/15-AOP1053}
\MR{3551201}

\bibitem{Lacoin2017} Lacoin, H.: 
The simple exclusion process on the circle has a diffusive cutoff window. 
\textit{Ann. Inst. Henri Poincar\'e Probab. Stat.} \textbf{53}, no. 3, (2017), 1402--1437.
\url{https://doi.org/10.1214/16-AIHP759}
\MR{3689972}

\bibitem{Landim} Landim, C.:
Metastable Markov chains.
\textit{Probab. Surv.} \textbf{16}, (2019), 143--227.
\url{https://doi.org/10.1214/18-PS310}
\MR{3960293}


\bibitem{Langevin1908} Langevin, P.: 
Sur la theorie du mouvement brownien. 
\textit{C. R. Acad. Sci. Paris} \textbf{146}, (1908), 530--533.

\bibitem{LELI} Leli\`evre, T., Nier, F. \& Pavliotis, G.:
Optimal non-reversible linear drift for the convergence to equilibrium of a diffusion.
\textit{J. Stat. Phys.} \textbf{152}, no. 2, (2013), 237--274.
\url{https://doi.org/10.1007/s10955-013-0769-x}
\MR{3082649}

\bibitem{Levin}  Levin, D., Peres, Y. \& Wilmer, E.:
\textit{Markov chains and mixing times}. 
With a chapter by James G. Propp and D. Wilson. 
Amer. Math. Soc., Providence, RI, (2009).
\url{https://doi.org/10.1090/mbk/058}
\MR{2466937}

\bibitem{LiWang} Li, Y. \& Wang, S.:
Numerical computations of geometric ergodicity for stochastic dynamics. 
\textit{Nonlinearity} \textbf{33}, no. 12, (2020), 6935--6970.
\url{https://doi.org/10.1088/1361-6544/aba93f}
\MR{4173565}

\bibitem{LS} Lubetzky, E. \& Sly, A.: 
Cutoff phenomena for random walks on random regular graphs.
\textit{Duke Math. J.} \textbf{153}, no. 3, (2010), 475--510. 
\url{https://doi.org/10.1215/00127094-2010-029}
\MR{2667423}

\bibitem{Mao} Mao, X.: 
\textit{Stochastic differential equations and applications}. Second edition. Horwood Publishing Limited, Chichester (2008).
\MR{2380366}

\bibitem{Martinelli} 
Martinelli, F., Sbano, L. \& Scoppola, E.:
Small random perturbation of dynamical systems: recursive multiscale analysis.
\textit{Stochastics} \textbf{49}, no. 3-4, (1994), 253--272.
\url{https://doi.org/10.1080/17442509408833923}
\MR{1785008}

\bibitem{Martinelli1} 
Martinelli, F. \& Scoppola, E.:
Small random perturbations of dynamical systems: exponential loss of memory of the initial condition.
\textit{Comm. Math. Phys.} \textbf{120}, no. 1, (1988), 25--69.
\url{https://doi.org/10.1007/BF01223205}
\MR{0972542}

\bibitem{Martinez} Mart\'{\i}nez, S. \& Ycart, B.:
Decay rates and cut-off for convergence and hitting times of Markov chains with countably infinite state space. 
\textit{Adv. in Appl. Probab.} \textbf{33}, no. 1, (2001), 188--205.
\url{https://doi.org/10.1017/S0001867800010697}
\MR{1825322}

\bibitem{Merle} Merle, M. \& Salez, J.:
Cutoff for the mean-field zero-range process.
\textit{Ann. Probab.} \textbf{47}, no. 5, (2019), 3170--3201.
\url{https://doi.org/10.1214/19-AOP1336}
\MR{4021248}

\bibitem{Monmarche} Monmarch\'e, P. \& Ramil, M.:
Overdamped limit at stationarity for non-equilibrium Langevin diffusions.
\textit{Electron. Commun. Probab.} \textbf{27}, no. 3, (2022), 8 pp.
\url{https://doi.org/10.1214/22-ECP447}
\MR{4368697}

\bibitem{Munch} M\"unch, F. \& Salez, J.:
Mixing time and expansion of non-negatively curved Markov chains.
\textit{J. \'Ec. polytech. Math.} \textbf{10}, (2023), 575--590.
\url{https://doi.org/10.5802/jep.226}
\MR{4567745}

\bibitem{Nestoridi} Nestoridi, E. \& Olesker-Taylor, S.: 
Limit profiles for reversible Markov chains. 
\textit{Probab. Theory Relat. Fields} \textbf{182}, no. 1-2, (2022), 157--188.
\url{https://doi.org/10.1007/s00440-021-01061-5}
\MR{4367947}

\bibitem{Nestoridi1}
Nestoridi, E.:
Comparing limit profiles of reversible Markov chains.
\textit{Electron. J. Probab.} \textbf{29}, no. 58, (2024), 14 pp.
\url{https://doi.org/10.1214/24-EJP1110}
\MR{4728694}

\bibitem{Nirenberg} Nirenberg, L.:
A strong maximum principle for parabolic equations. \textit{Comm. Pure Appl. Math.} \textbf{6}, (1953), 167--177. 
\url{https://doi.org/10.1002/cpa.3160060202}
\MR{0055544}

\bibitem{Oliveira} Oliveira, R.:
Mixing and hitting times for finite Markov chains. 
\textit{Electron. J. Probab.} \textbf{17} no. 70, (2012), 12 pp. 
\url{https://doi.org/10.1214/EJP.v17-2274}
\MR{2968677}

\bibitem{VARES} Olivieri, E. \& Vares, M.: \textit{Large deviations and metastability}.
Encyclopedia of Mathematics and its Applications \textbf{100}. Cambridge Univ. Press, Cambridge, (2005).
\url{https://doi.org/10.1017/CBO9780511543272}
\MR{2123364}
 
\bibitem{pages} Pag\`es, G. \& Panloup, F.:
Ergodic approximation of the distribution of a stationary diffusion: rate of convergence. 
\textit{Ann. Appl. Probab.} \textbf{22} no. 3, (2012), 1059--1100. 
\url{https://doi.org/10.1214/11-AAP779}
\MR{2977986}

\bibitem{pavliotis} Pavliotis, G.: 
\textit{Stochastic processes and applications. Diffusion processes, the Fokker--Planck and Langevin
equations}. 
Texts in Applied Mathematics \textbf{60}. Springer, New York, (2014).
\url{https://doi.org/10.1007/978-1-4939-1323-7}
\MR{3288096} 

\bibitem{Peng15} Peng, X.:
The continuity of SDE with respect to initial value in the total variation.
\textit{Adv. Math. (China)} \textbf{44}, no. 5,  (2015), 783--788.
\MR{3440472}
 
\bibitem{Peng} Peng, X. \& Zhang, R.:
Exponential ergodicity for SDEs under the total variation.
\textit{J. Evol. Equ.} \textbf{18}, no. 3, (2018), 1051--1067.
\url{https://doi.org/10.1007/s00028-018-0429-3}
\MR{3859440}

\bibitem{Pillai} Pillai, N. \& Smith, A.:
Kac's walk on $n$-sphere mixes in $n\log(n)$ steps.
\textit{Ann. Appl. Probab.} \textbf{27}, no. 1, (2017), 631--650.
\url{https://doi.org/10.1214/16-AAP1214}
\MR{3619797} 

\bibitem{Pomeau} Pomeau, Y. \& Piasecki, J.: 
The Langevin equation. 
\textit{C. R. Phys.} \textbf{18}, no. 9-10, (2017), 570--582.
\url{https://doi.org/10.1016/j.crhy.2017.10.001}

\bibitem{Pontryagin}
Pontryagin, L. S.:
\textit{Ordinary differential equations}.
Translated from the Russian by Leonas Kacinskas and Walter B. Counts,
Addison-Wesley Publishing Co.,
(1962).
\MR{0140742}

\bibitem{Quattropani} Quattropani, M. \& Sau, F.:
Mixing of the averaging process and its discrete dual on finite-dimensional geometries. 
\textit{Ann. Appl. Probab.} \textbf{33}, no. 2, (2023), 1136--1171.
\url{https://doi.org/10.1214/22-AAP1838}
\MR{4564423} 


\bibitem{Reiss} Reiss, R.:
\textit{Approximate distributions of order statistics.
With applications to nonparametric statistics.} Springer Series in Statistics. Springer-Verlag, New York, (1989).
\url{https://doi.org/10.1007/978-1-4613-9620-8}  
\MR{0988164}

\bibitem{SCO} Saloff-Coste, L.: 
\textit{Random walks on finite groups}. Probability on discrete structures, 263--346.
Encyclopaedia Math. Sci. \textbf{110}, Probab. Theory 1, Springer, Berlin, (2004).
\url{https://doi.org/10.1007/978-3-662-09444-0_5}
\MR{2023654}

\bibitem{SI} Siegert, W.: 
\textit{Local Lyapunov exponents.
Sublimiting growth rates of linear random differential equations}. 
 Lecture Notes in Mathematics \textbf{1963}. Springer-Verlag, Berlin, (2009).
\url{https://doi.org/10.1007/978-3-540-85964-2}
\MR{2465137}

\bibitem{SV} Stroock, D. \& Varadhan, S.: \textit{Multidimensional diffusion processes}. Springer-Verlag, Berlin, (2006).
\url{https://doi.org/10.1007/3-540-28999-2}
\MR{2190038} 

\bibitem{Trefethen} 
Trefethen, L.N. \& Trefethen, L.M.: 
How many shuffles to randomize a deck of cards?.
\textit{R. Soc. Lond. Proc. Ser. A Math. Phys. Eng. Sci.} \textbf{456}, no. 2002, (2000), 2561--2568.
\url{https://doi.org/10.1098/rspa.2000.0625}
\MR{1796496}


\bibitem{Veretennikov} Veretennikov, A.: 
Note on local mixing techniques for stochastic differential equations.
\textit{Mod. Stoch. Theory Appl.} \textbf{8}, no. 1, (2021), 1--15.
\url{https://doi.org/10.15559/21-VMSTA174}
\MR{4235561}

\bibitem{Villani} Villani, C.: 
\textit{Optimal transport. Old and new}.  Grundlehren der mathematischen Wissenschaften  \textbf{338}. Springer-Verlag, Berlin, (2009). 
\url{https://doi.org/10.1007/978-3-540-71050-9}
\MR{2459454} 



\bibitem{Ycart1999} Ycart, B.:  
Cutoff for samples of Markov chains. 
\textit{ESAIM Probab. Stat.} \textbf{3}, (1999), 89--106.
\url{https://doi.org/10.1051/ps:1999104}
\MR{1716128}


\end{thebibliography}
\end{document}